\documentclass[12pt]{article}

\usepackage{times}

\usepackage{enumitem}
\usepackage[utf8]{inputenc}
\usepackage{commath}
\usepackage{amsthm, amscd}
\usepackage{amsmath}
\usepackage{amssymb}
\usepackage{cite}
\usepackage{setspace}
\usepackage{ stmaryrd }
\usepackage{tikz-cd}
\usepackage{tikz}
\usepackage{pgfplots}

\usepackage{tocloft}
\usepackage{sectsty}
\usepackage{xparse}

\newcommand*\tasklabelformat[1]{#1)}

\usepackage{titlesec}

\usepackage{tasks}[newest]
\settasks{
  label = \alph* ,
  label-format = \tasklabelformat ,
  label-width  = 12pt
}

\usepackage{bigints}
\usepackage{comment}
\usepackage{mathtools}
\usepackage{mathrsfs}
\usepackage{fancyhdr}

\allowdisplaybreaks[4]

\numberwithin{equation}{section}
\usepackage{etoolbox}
\patchcmd{\thebibliography}
  {\settowidth}
  {\setlength{\itemsep}{0pt plus -10pt}\settowidth}
  {}{}
\apptocmd{\thebibliography}
  {
  }
  {}{}
\makeatletter
\newtheorem*{rep@theorem}{\rep@title}
\newcommand{\newreptheorem}[2]{%
\newenvironment{rep#1}[1]{%
 \def\rep@title{#2 \ref{##1}}%
 \begin{rep@theorem}}%
 {\end{rep@theorem}}}
\makeatother

\theoremstyle{theorem}

\newreptheorem{theorem}{Theorem}
\newtheorem{thm}{Theorem}[section]
\newtheorem*{thm*}{Theorem}
\theoremstyle{definition}
\newtheorem{prop}[thm]{Proposition}
\newtheorem*{prop*}{Proposition}

\newtheorem{lem}[thm]{Lemma}
\newtheorem{cor}[thm]{Corollary}
\newtheorem*{cor*}{Corollary}
\theoremstyle{remark}

\newtheorem{rem}[thm]{Remark}

\title{\vspace*{-1.5cm} About Wess-Zumino-Witten equation \\
and Harder-Narasimhan potentials
}

\author
{Siarhei Finski
}

\date{}

\usepackage[%
    left=1in,%
    right=1in,%
    top=1.1in,%
    bottom=0.8in,%
    paperheight=11in,%
    paperwidth=8.5in%
]{geometry}

\newcommand{\imun} {\sqrt{-1}}

\newcommand{\res}{{\rm{Res}}}
\newcommand{\ext}{{\rm{Ext}}}

\newcommand{\comp}{\mathbb{C}}
\newcommand{\real}{\mathbb{R}}

\newcommand{\nat}{\mathbb{N}}
\newcommand{\integ}{\mathbb{Z}}

\newcommand{\enmr}[1]{\text{End}{(#1)}}

\newcommand{\ccal}{\mathscr{C}}

\newcommand{\dbar}{ \overline{\partial} }

\newcommand{\rk}[1]{{\rm{rk}} ( #1 )}
\newcommand{\tr}[1]{{\rm{Tr}} \big[ #1 \big]}

\newcommand{\scal}[2]{\langle #1, #2 \rangle}

\makeatletter
\DeclareFontFamily{OMX}{MnSymbolE}{}
\DeclareSymbolFont{MnLargeSymbols}{OMX}{MnSymbolE}{m}{n}
\SetSymbolFont{MnLargeSymbols}{bold}{OMX}{MnSymbolE}{b}{n}
\DeclareFontShape{OMX}{MnSymbolE}{m}{n}{
    <-6>  MnSymbolE5
   <6-7>  MnSymbolE6
   <7-8>  MnSymbolE7
   <8-9>  MnSymbolE8
   <9-10> MnSymbolE9
  <10-12> MnSymbolE10
  <12->   MnSymbolE12
}{}
\DeclareFontShape{OMX}{MnSymbolE}{b}{n}{
    <-6>  MnSymbolE-Bold5
   <6-7>  MnSymbolE-Bold6
   <7-8>  MnSymbolE-Bold7
   <8-9>  MnSymbolE-Bold8
   <9-10> MnSymbolE-Bold9
  <10-12> MnSymbolE-Bold10
  <12->   MnSymbolE-Bold12
}{}

\let\llangle\@undefined
\let\rrangle\@undefined
\DeclareMathDelimiter{\llangle}{\mathopen}%
                     {MnLargeSymbols}{'164}{MnLargeSymbols}{'164}
\DeclareMathDelimiter{\rrangle}{\mathclose}%
                     {MnLargeSymbols}{'171}{MnLargeSymbols}{'171}
                     
\DeclareMathOperator*{\esssup}{ess\,sup}
\DeclareMathOperator*{\essinf}{ess\,inf}
\makeatother

\newenvironment{sciabstract}{}

\cftsetindents{section}{0em}{1.7em}

\sectionfont{\large}

\titlespacing*{\section}
{0pt}{5pt}{5pt}

\begin{document}

\maketitle

\begin{sciabstract}
  \textbf{Abstract.}
	For a polarized family of complex projective manifolds, we identify the algebraic obstructions that govern the existence of approximate solutions to the Wess-Zumino-Witten equation.
	When this is specialized to the fibration associated with a projectivization of a vector bundle, we recover a version of Kobayashi-Hitchin correspondence.
	\par 
	More broadly, we demonstrate that a certain auxiliary Monge-Ampère type equation, generalizing the Wess-Zumino-Witten equation by taking into account the weighted Bergman kernel associated with the Harder-Narasimhan filtrations of direct image sheaves, admits approximate solutions over any polarized family.
	These approximate solutions are shown to be the closest counterparts to true solutions of the Wess-Zumino-Witten equation whenever the latter do not exist, as they minimize the associated Yang-Mills functional.
	\par 
	As an application, in a fibered setting, we prove an asymptotic converse to the Andreotti-Grauert theorem conjectured by Demailly.
\end{sciabstract}

\pagestyle{fancy}
\lhead{}
\chead{Wess-Zumino-Witten equation and Harder-Narasimhan potentials}
\rhead{\thepage}
\cfoot{}


\newcommand{\Addresses}{{
  \bigskip
  \footnotesize
  \noindent \textsc{Siarhei Finski, CNRS-CMLS, École Polytechnique F-91128 Palaiseau Cedex, France.}\par\nopagebreak
  \noindent  \textit{E-mails }: \texttt{finski.siarhei@gmail.com} $\quad$ or  $\quad$  \texttt{siarhei.finski@polytechnique.edu}.
}} 

\vspace*{0.25cm}

\par\noindent\rule{1.25em}{0.4pt} \textbf{Table of contents} \hrulefill

\vspace*{-1.5cm}

\tableofcontents

\vspace*{-0.2cm}

\noindent \hrulefill


\section{Introduction}\label{sect_intro}
	Consider a holomorphic submersion $\pi : X \to B$ between compact Kähler manifolds $X$ and $B$ of dimensions $n + m$ and $m$ respectively.
	We fix a Kähler form $\omega_B$ on $B$.
	We say that a smooth closed relatively Kähler $(1, 1)$-form $\alpha$ on $X$ satisfies the Wess-Zumino-Witten equation if
	\begin{equation}\label{eq_wzw}
		\alpha^{n + 1} \wedge \pi^* \omega_B^{m - 1} = 0.
	\end{equation}
	This terminology was introduced by Donaldson in \cite{DonaldSymSp} when $B$ is an annulus in $\mathbb{C}$, due to its resemblance to certain equations from mathematical physics \cite{WittenWZW}. 
	This particular instance of (\ref{eq_wzw}) is significant in Kähler geometry, as its solutions correspond to the Mabuchi geodesics \cite{Mabuchi}, \cite{Semmes}, \cite{DonaldSymSp}, \cite{ChenGeodMab}, which are essential in the study of constant scalar curvature metrics.
	\par 
	The main focus of this paper, however, is the case where $B$ has no boundary, leading to a markedly different theory. 
	Specifically, we address the following two questions:
	\par 
	\vspace*{0.3cm}
	\noindent
	A)
	What are the algebraic obstructions for the existence of approximate solutions to (\ref{eq_wzw})?
	\vspace*{0.3cm}
	\par 
	\noindent
	B)
	When approximate solutions to (\ref{eq_wzw}) do not exist, what are the “best" alternatives to them?
	\vspace*{0.3cm}
	\par 
	\par  
	To clarify both questions, we define the Wess-Zumino-Witten functional as follows
	\begin{equation}\label{defn_wzwfunc}
		{\rm{WZW}}(\alpha, \omega_B) := \int_X \big| \alpha^{n + 1} \wedge \pi^* \omega_B^{m - 1} \big|, \qquad {\rm{WZW}}([\alpha], \omega_B) = \inf {\rm{WZW}}(\alpha, \omega_B),
 	\end{equation}
 	where the infimum is taken over all smooth closed relatively Kähler $(1, 1)$-forms $\alpha$ in the class $[\alpha]$, and $| \cdot |$ is the absolute value of a volume form, evaluated with respect to the orientation given by the complex structure.
 	By Question A) we mean that we would like to give an algebraic description of classes $[\alpha]$, verifying ${\rm{WZW}}([\alpha], \omega_B) = 0$.
 	By Question B) we mean that we would like to have an explicit construction of a sequence of the forms $\alpha_{\epsilon} \in [\alpha]$, for any $\epsilon > 0$, verifying ${\rm{WZW}}(\alpha_{\epsilon}, \omega_B) \leq {\rm{WZW}}([\alpha], \omega_B) + \epsilon$.
 	We answer both of these questions for $[\alpha] \in H^2(X, \integ)$, i.e. such that there is a relatively ample line bundle $L$ over $X$, verifying $[\alpha] = c_1(L)$.
 	\par 
 	Our motivation for studying (\ref{eq_wzw}) comes from the fact that it generalizes the Hermite-Einstein equation on vector bundles.
 	Indeed, let $F$ be a holomorphic vector bundle over $B$. 
	Let $L := \mathcal{O}(1)$ be the hyperplane bundle over $X := \mathbb{P}(F^*)$, and let $\pi : \mathbb{P}(F^*) \to B$ be the natural projection.
	For a coherent sheaf $\mathscr{E}$ over $B$, we define its degree by $\deg(\mathscr{E}) := \int_B c_1(\det (\mathscr{E})) \cdot [\omega_B]^{m - 1}$, where $\det \mathscr{E}$ is Knudsen-Mumford determinant of $\mathscr{E}$, see \cite{Knudsen1976}.
	We assume for simplicity that $F$ is normalized in the sense that $\deg(F) = 0$.
	A Hermitian metric $h^F$ on $F$ solves the \textit{Hermite-Einstein equation} if the curvature $R^{h^F}$ of its Chern connection satisfies $\frac{\imun}{2 \pi} R^{h^F} \wedge \omega_B^{m-1} = 0$. 
	\par 
	It is then a classical calculation, originally due to Kobayashi \cite{KobFinEins}, which says that $h^F$ solves the Hermite-Einstein equation if and only if for the associated metric $h^L$ on $L$, the relatively Kähler $(1, 1)$-form $\omega := c_1(L, h^L)$ solves (\ref{eq_wzw}).
	Moreover, following a question raised by Kobayashi \cite{KobFinEins}, it was established by Feng-Liu-Wan in \cite[Proposition 3.5]{FengLiuWang} that if there is a relatively positive metric $h^L$ on $L$ such that $\omega := c_1(L, h^L)$ solves (\ref{eq_wzw}) over $\mathbb{P}(F^*)$, then one can cook up from it a Hermitian metric on $F$, solving the Hermite-Einstein equation.
	\par 
	Hence the Kobayashi-Hitchin correspondence (or Donaldson-Uhlenbeck-Yau theorem), see \cite{DonaldASD}, \cite{UhlYau}, \cite[Theorems 6.10.13]{KobaVB}, which relates the existence of solutions to the Hermite-Einstein equation on $F$ with slope-stability, can be formulated purely in terms of solvability of (\ref{eq_wzw}) over $\mathbb{P}(F^*)$.
	Questions A) and B) serve to generalize this correspondence for fibrations not necessarily associated with vector bundles.
 	\par
	In order to state our results, recall that a \textit{slope} (or $[\omega_B]$-slope) of a coherent sheaf $\mathscr{E}$ over $B$ is defined as $\mu(\mathscr{E}) := \deg(\mathscr{E}) / (\rk{\mathscr{E}} \cdot \int_B [\omega_B]^m)$.
	A torsion-free coherent sheaf $\mathscr{E}$ is called \textit{semistable} (or $[\omega_B]$-\textit{semistable}) if for every coherent subsheaf $\mathcal{F}$ of $\mathscr{E}$, verifying $\rk{\mathcal{F}} > 0$, we have $\mu(\mathcal{F}) \leq \mu(\mathscr{E})$.
	Remark the unusual normalization of the slope by $\int_B [\omega_B]^m$.
	\par 
	Recall that any vector bundle $E$ on $(B, [\omega_B])$ admits a unique filtration by subsheaves
	\begin{equation}\label{eq_HN_filt}
		E = \mathcal{F}^{HN}_{\lambda_1} \supset \mathcal{F}^{HN}_{\lambda_2} \supset \cdots \supset \mathcal{F}^{HN}_{\lambda_q} \supset \{0\} =: \mathcal{F}^{HN}_{\lambda_{q + 1}},
	\end{equation}
	also called the \textit{Harder-Narasimhan filtration} and defined so that for any $1 \leq i \leq q$, the quotient sheaf $\mathcal{F}^{HN}_{\lambda_i} / \mathcal{F}^{HN}_{\lambda_{i + 1}}$ is the maximal semistable (torsion-free) subsheaf of $E / \mathcal{F}^{HN}_{\lambda_{i + 1}}$, i.e. for any subsheaf of $\mathcal{F}$ of a (torsion-free) sheaf $E / \mathcal{F}^{HN}_{\lambda_{i + 1}}$, we have $\mu(\mathcal{F}) \leq \mu(\mathcal{F}^{HN}_{\lambda_i} / \mathcal{F}^{HN}_{\lambda_{i + 1}})$ and $\rk{\mathcal{F}} \leq \rk{\mathcal{F}^{HN}_{\lambda_i} / \mathcal{F}^{HN}_{\lambda_{i + 1}}}$ if $\mu(\mathcal{F}) = \mu(\mathcal{F}^{HN}_{\lambda_i} / \mathcal{F}^{HN}_{\lambda_{i + 1}})$, and $\lambda_i = \mu(\mathcal{F}^{HN}_{\lambda_i} / \mathcal{F}^{HN}_{\lambda_{i + 1}})$.
	\par 
	\begin{sloppypar}
	We define the \textit{Harder-Narasimhan slopes}, $\mu_1, \ldots, \mu_{\rk{E}}$, of $E$, so that $\lambda_i$ appears among $\mu_1, \ldots, \mu_{\rk{E}}$ exactly $\rk{\mathcal{F}^{HN}_{\lambda_i} / \mathcal{F}^{HN}_{\lambda_{i + 1}}}$ times, and the sequence $\mu_1, \ldots, \mu_{\rk{E}}$ is non-decreasing.
	We call $\mu_{\min} := \mu_1$ and $\mu_{\max} := \mu_{\rk{E}}$, the minimal and the maximal slopes respectively. 
	\end{sloppypar}
	\par
	Now, in our fibered setting, for a relatively ample line bundle $L$ on $X$ and $k \in \nat^*$, we denote the direct image sheaves by $E_k := R^0 \pi_* L^{\otimes k}$.
	For $k$ big enough, a standard argument shows that $E_k$ are locally free.
	We let $N_k := \rk{E_k}$, and denote by $\mu_1^k, \ldots, \mu_{N_k}^k$ the Harder-Narasimhan slopes of $E_k$.
	Define the probability measure $\eta_k^{HN}$ on $\real$ as
	\begin{equation}\label{eq_eta_defn}
		\eta_k^{HN} := \frac{1}{N_k} \sum_{i = 1}^{N_k} \delta \Big[ \frac{\mu_i^k}{k} \Big],
	\end{equation}
	where $\delta[x]$ is the Dirac mass at $x \in \real$. 
	\par 
	It was established by Chen in \cite[Theorem 4.3.6]{ChenHNolyg} (for Riemann surfaces $B$) and by the author \cite[Theorem 1.5]{FinHNII} (for general Kähler manifolds $B$) that the sequence of probability measures $\eta_k^{HN}$ converges weakly, as $k \to \infty$, to a probability measure of compact support $\eta^{HN}$ on $\real$. 
	Moreover, the support of $\eta^{HN}$ equals $[\essinf \eta^{HN}, \esssup \eta^{HN}]$, and $\eta^{HN}$ is absolutely continuous with respect to the Lebesgue measure, except probably for a point mass at $\esssup \eta^{HN}$.
	\par 
	We can now state the first result of this article.
	\begin{thm}\label{thm_main}
		For any $t \in \real$, we have
		\begin{equation}\label{eq_main}
			{\rm{WZW}}(c_1(L) - t \pi^* [\omega_B], \omega_B) 
			= 
			\int_{x \in \real} |x - t| d \eta^{HN}(x)
			\cdot
			\int_X c_1(L)^n \pi^* [\omega_B]^m
			\cdot
			(n + 1).
		\end{equation}
	\end{thm}
	\begin{rem}
		\begin{sloppypar} 
		a) Remark that while the left-hand side of (\ref{eq_main}) is a differential-geometric quantity, the right-hand side is purely an algebraic one.
		Also, as $\eta^{HN}$ depends only on the cohomological class $[\omega_B]$, the quantity ${\rm{WZW}}(c_1(L) - t \pi^* [\omega_B], \omega_B)$ ultimately also depends only on $[\omega_B]$.
		\end{sloppypar}
		\par  
		b) It is easy to check, cf. \cite[Proposition 5.1]{ChenMaclean}, that if $\int_{x \in \real} |x - t| d \eta_1(x) = \int_{x \in \real} |x - t| d \eta_2(x)$, for any $t \in \real$ and some Radon measures $\eta_1$, $\eta_2$ of compact support on $\real$, then $\eta_1 = \eta_2$.
		In particular, Theorem \ref{thm_main} gives a differential-geometric characterization of $\eta^{HN}$.
		\par 
		c) Theorem \ref{thm_main} establishes for $p = 1$ the conjecture of the author from \cite{FinHNII} about the optimality of the lower bound on the Fibered Yang-Mills functional.
		If we denote by
	\begin{equation}\label{eq_resol_hn_filt00}
		E_k = \mathcal{F}^{HN, k}_{\lambda_{1, k}} \supset \mathcal{F}^{HN, k}_{\lambda_{2, k}} \supset \cdots \supset \mathcal{F}^{HN, k}_{\lambda_{q_k, k}} \supset \{0\} =: \mathcal{F}^{HN, k}_{\lambda_{q_k + 1, k}}.
	\end{equation}
		the Harder-Narasimhan filtration of $E_k$, and note $\mathcal{Q}_{i}^{HN, k} := \mathcal{F}^{HN, k}_{\lambda_{i, k}} / \mathcal{F}^{HN, k}_{\lambda_{i + 1, k}}$, $i = 1, \ldots, q_k$, then from the weak convergence of $\eta_k^{HN}$ towards $\eta^{HN}$, we can reformulate (\ref{eq_main}) as
		\begin{multline}\label{eq_main_ref}
			{\rm{WZW}}(c_1(L) - t \pi^* [\omega_B], \omega_B) 
			= 
			\int_X c_1(L)^n \pi^* [\omega_B]^m
			\cdot
			(n + 1)
			\cdot
			\\
			\cdot
			\lim_{k \to \infty} \frac{1}{k \cdot N_k} \sum_{i = 1}^{q_k} \Big| \deg (\mathcal{Q}_{i}^{HN, k}) - tk \cdot \rk{\mathcal{Q}_{i}^{HN, k}} \Big|.
		\end{multline}
	\end{rem}
	\par 
	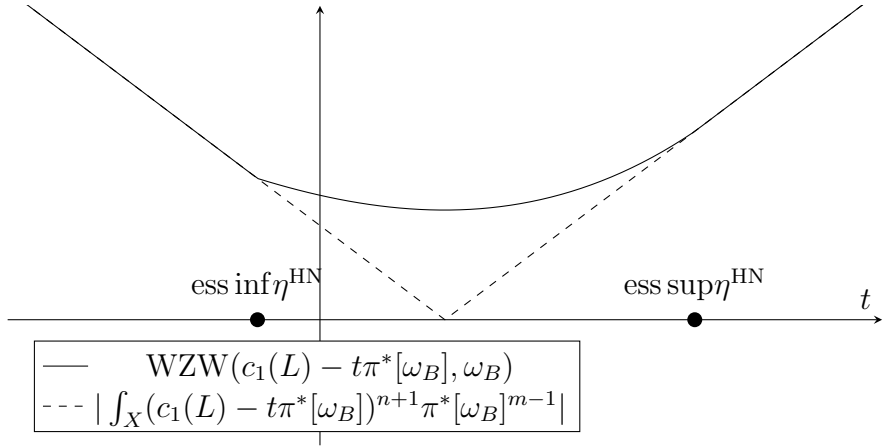
\begin{figure}[!h]\label{figure1}
\centering
\begin{tikzpicture}[
		  declare function={
		    func1(\x)= (\x<=-0.5) * (-(3/2 * \x) + 3/2)   +
		     and(\x>-0.5, \x<=3) * (- (x + 10) * (\x + 0.5) * (3 - \x)/40  + 9/4 * 2/7 * (3 - \x) + (\x + 0.5) * 3 * 2/7)     +
		                (\x>3) * ((3/2 * \x) - 3/2);
		  }
		]
		\begin{axis}[%
		  samples=100,
		  domain=-3:5,
		  xmin=-2.5, xmax=4.5,
		  ymin=-2, ymax=5,
		  axis lines=middle,
		  ticks=none,
		  xlabel={$t$},
		  ylabel={},
		  legend pos=south west,
		  width=0.8*\textwidth,
		  height=0.45*\textwidth]
		\addplot [mark options={scale=2, solid}] {func1(x)};
		\addplot [dashed,mark options={scale=2, solid}] coordinates {
		(-3,6) (1,0)  (5,6)
		};
		\node[label={${\rm{ess \, sup} \eta^{HN}}$},circle,fill,inner sep=2pt] at (axis cs:3,0) {};
\node[label={${\rm{ess \, inf} \eta^{HN}}$},circle,fill,inner sep=2pt] at (axis cs:-0.5,0) {};
		\legend{{${\rm{WZW}}(c_1(L) -t \pi^* [\omega_B], \omega_B)$}, $|\int_X (c_1(L) -t \pi^* [\omega_B])^{n + 1} \pi^* [\omega_B]^{m - 1}|$}
		\end{axis};
	\end{tikzpicture}
	\caption{Sharp and trivial lower bounds on the Wess-Zumino-Witten functional.}
\end{figure}
	\par 
	As an application of Theorem \ref{thm_main}, we obtain the following result, which precisely characterizes the link between the existence of approximate solutions to (\ref{eq_wzw}) and the associated algebraic obstructions. This result provides a comprehensive answer to Question A) stated above.
	\begin{cor}\label{cor_wzw}
		The following statements are equivalent:
		\par 
		a) The measure $\eta^{HN}$ is the Dirac mass at $t$.
		\par 
		b) We have $\mu(E_k) \sim t k$, as $k \to \infty$, and for any $\epsilon > 0$, there is $k_0 \in \nat$, such that for any $k \geq k_0$ and any coherent subsheaf $\mathcal{F}_k \subset E_k$, $\rk{\mathcal{F}_k} > 0$, we have $\mu(\mathcal{F}_k) \leq \mu(E_k) + \epsilon k$.
		\par 
		c) For any $\epsilon > 0$, there is a relatively positive smooth closed $(1, 1)$-form $\omega_{\epsilon}$ in $c_1(L)$, verifying
		\begin{equation}\label{eq_cor_wzw}
			\int_X \big| \omega_{\epsilon}^{n + 1} \wedge \pi^* \omega_B^{m - 1} - t \cdot \omega_{\epsilon}^{n} \wedge \pi^* \omega_B^{m} \cdot (n + 1) \big| < \epsilon.
		\end{equation}
		\par 
		d) 	If $\dim B = 1$, the above conditions are equivalent to the following: for any irreducible complex analytic subspace $Y \subset X$ of dimension $k + 1$, $k \in \nat$, we have 
	$\int_Y c_1(L)^{k + 1} \geq t \cdot \int_Y c_1(L)^k \pi^* [\omega_B] \cdot (k + 1)$, with an equality if $Y = X$.
	\end{cor}
	\begin{rem}
		As we explain in Section \ref{sect_j_eq}, when Corollary \ref{cor_wzw} is specialized to $\pi : \mathbb{P}(F^*) \to B$, $L := \mathscr{O}(1)$, for a holomorphic vector bundle $F$ over $B$, it gives a version of Kobayashi-Hitchin correspondence, as the semistability of $F$ is equivalent to the condition that $\eta^{HN}$ is the Dirac mass.
	\end{rem}
	\par 
	Let us now explain an application of Theorem \ref{thm_main} towards asymptotic cohomology.
	Recall that on a compact complex manifold $Y$ of dimension $n$ with a holomorphic line bundle $F$, the \textit{$q$-th asymptotic cohomology} is defined as
	\begin{equation}\label{defn_as_coh}
		\hat{h}^q(Y, F) := \limsup_{k \to \infty} \frac{n!}{k^n}	\dim H^q(Y, F^{\otimes k}).
	\end{equation}
	Refer to \cite{KuronyaAsymptCoh} and \cite{DemaillyRiemannHOLM} for the proof of some fundamental properties of $\hat{h}^q(Y, F)$. 
	Holomorphic Morse inequalities of Demailly \cite{Dem2} give the following upper bounds
	\begin{equation}\label{eq_holmi}
		\hat{h}^q(Y, F) \leq \int_{Y(\alpha, q)} (-1)^q \alpha^n,
	\end{equation}
	where $\alpha$ is an arbitrary smooth form in the class $c_1(F)$, and $Y(\alpha, q)$ is the open set of points $x \in Y$, so that $\alpha(x)$ has signature $(n - q, q)$.
	\vspace*{0.3cm}
	\par 
	\noindent \textbf{Conjecture 1.} (Demailly \cite[Question 1.13]{DemaillyRiemannHOLM})
	Do we have $\hat{h}^q(Y, F) = \inf \int_{Y(\alpha, q)} (-1)^q \alpha^n$, where the infimum is taken over all smooth closed $(1, 1)$-forms $\alpha$ in the class $c_1(F)$?
	\vspace*{0.3cm}
	\par 
	As explained in \cite[p. 3]{DemAndrGrau}, cf. also the end of Section \ref{sect_dem_conj}, the conjecture above is related to the Andreotti-Grauert vanishing theorem \cite{AndrGrauert}, and in a way it should be seen as an asymptotic converse of it.
	Besides the cases $q = 0$ and $n \leq 2$, proved in \cite[Theorems 1.3, 1.4]{DemAndrGrau}, the conjecture remains largely open.
	The major difficulty is, of course, to construct a differential form from the algebraic datum of asymptotic cohomology.
	As an application of Theorem \ref{thm_main}, we have the following result.
	\begin{cor}\label{thm_andgr}
		For a holomorphic submersion $p : Y \to C$ between a complex projective manifold $Y$ and a compact Riemann surface $C$, Conjecture 1 holds for any holomorphic line bundle $F$ on $Y$ which is relatively ample with respect to $p$.
	\end{cor}
	\begin{rem}
		It seems unlikely that the results of the present paper can be used to resolve other cases of the above conjecture. 
		This is because the relationship between the Wess-Zumino-Witten functional (\ref{defn_wzwfunc}) and the absolute Monge-Ampère functional (\ref{defn_ama_func}) appears to exist only when the base is a curve; also, beyond the relatively ample setting, most of our techniques for the study of the Wess-Zumino-Witten functional break down.
	\end{rem}
	We will in fact show that in the setting of Corollary \ref{thm_andgr}, the sequence of $(1, 1)$-forms, minimizing the Wess-Zumino-Witten functional, saturates the lower bound (\ref{eq_holmi}).
	\par 
	A natural question arises as to whether one can construct forms that attain the lower bound for ${\rm{WZW}}(c_1(L) - t \pi^* [\omega_B], \omega_B)$ as solutions to a specific auxiliary equation.
	This brings us to the main contribution of this article, where we show that a certain auxiliary Monge-Ampère type equation, generalizing the Wess-Zumino-Witten equation, always admits approximate solutions on arbitrary fibrations, and that these solutions saturate the lower bound.
	\par 
	The auxiliary Monge-Ampère equation is based on the \textit{weighted Bergman kernel} associated with the Harder-Narasimhan filtrations of direct image sheaves.
	In order to define it, for a relatively positive Hermitian metric $h^L$ on $L$, we denote $\omega := c_1(L, h^L)$, and associate a Hermitian metric ${\rm{Hilb}}_k^{\pi}(h^L)$ on $E_k$, $k \in \nat$, defined for $f, f' \in E_{k, b}$, $b \in B$, as follows
	\begin{equation}\label{eq_l2_prod}
		\scal{f}{f'}_{{\rm{Hilb}}_k^{\pi}(h^L)_b} := \frac{1}{n!} \int_{X_b} \scal{f(x)}{f'(x)}_{h^{L^{\otimes k}}} \cdot \omega^n(x).
	\end{equation}
	\par 
	Now, the restriction of the filtration (\ref{eq_resol_hn_filt00}) to $b \in B$ provides a filtration of $E_{k, b}$.
	We denote $n_{i, k} := \dim \mathcal{F}^{HN, k}_{\lambda_{i, k}, b}$, $i = 1, \ldots, q_k + 1$, and consider the orhonormal basis $s_{1, k}, \ldots, s_{N_k, k} \in E_{k, b}$, of $(E_{k, b}, {\rm{Hilb}}_k^{\pi}(h^L)_b)$ adapted to (\ref{eq_resol_hn_filt00}) in the sense that $s_{1, k}, \ldots, s_{n_{q_k, k}, k}$, form an orthonormal basis of $\mathcal{F}^{HN, k}_{\lambda_{q_k, k}, b}$, the elements $s_{1, k}, \ldots, s_{n_{q_k - 1, k}, k}$, form an orthonormal basis of $\mathcal{F}^{HN, k}_{\lambda_{q_k - 1, k}}$, etc. 
	We define the \textit{Harder-Narasimhan potential}, ${\rm{HN}}_k(\omega) : X \to \real$ of $\omega := c_1(L, h^L)$ as
	\begin{equation}\label{eq_hn_potent}
		{\rm{HN}}_k(\omega)(x)
		=
		\sum_{i = 1}^{q_k} \lambda_{i, k} \cdot \sum_{j = n_1 - n_i + 1}^{n_1 - n_{i + 1}} \big| s_{i, k}(x) \big|^2_{h^{L^{\otimes k}}}.
	\end{equation}
	The reader will verify that for connected manifolds $X$, ${\rm{HN}}_k(\omega)$ depends only on $\omega$ and not on $h^L$, as the notation suggests. 
	It is also independent of the choice of the basis $s_{1, k}, \ldots, s_{N_k, k}$ as long as the latter stays orthonormal and adapted to the filtration in the sense described above.
	The definition in (\ref{eq_hn_potent}) closely resembles the \textit{Bergman kernel}, cf. \cite{TianBerg}, \cite{MaHol}, with the only difference being that we use a weighted version of it, where the weights are prompted by (\ref{eq_resol_hn_filt00}).
	\par 
	Relying on our previous work \cite{FinSubmToepl}, we shall establish in Corollary \ref{prop_conv_hn_pot} that the sequence of functions $x \mapsto \frac{1}{k^{n + 1}} {\rm{HN}}_k(\omega)(x)$, $x \in X$, converges almost everywhere, as $k \to \infty$, to a bounded function, which we denote by ${\rm{HN}}(\omega) \in L^{\infty}(X)$, and call the \textit{Harder-Narasimhan potential} of $\omega$.
	Moreover, ${\rm{HN}}(\omega)$ coincides almost everywhere with the speed of the geodesic ray associated with the Harder-Narasimhan filtration, cf. Section \ref{sect_subm_hn_filt} for details.
	\par 
	Let us consider the following auxiliary Monge-Ampère type equation
	\begin{equation}\label{eq_aux_ma}
		\omega^{n + 1} \wedge \pi^* \omega_B^{m - 1}
		=
		{\rm{HN}}(\omega) \cdot \omega^n \wedge \pi^* \omega_B^m \cdot (n + 1).
	\end{equation}
	We draw the reader's attention to the fact that (\ref{eq_aux_ma}) is not a PDE, since the dependence of ${\rm{HN}}(\omega)$ on $\omega$ is quite intricate. 
	Moreover, the optimal regularity of ${\rm{HN}}(\omega)$ remains an interesting open question.
	As explained in (\ref{eq_geod_ray_bnd}), ${\rm{HN}}(\omega)$ can alternatively be interpreted as the speed of the Mabuchi geodesic associated with the Harder-Narasimhan filtration. 
	At present, it is not known whether this filtration is finitely generated; see \cite[\S 7]{FinHNI} for a discussion. 
	If finite generation were known, the results of Chu-Tosatti-Weinkove \cite{ChuTossVeinC11} would imply that ${\rm{HN}}(\omega)$ is Lipschitz continuous. 
	For now, the only property that is known with certainty is that ${\rm{HN}}(\omega)$ is bounded.
	\par 
	It is, hence, quite premature to inquire whether (\ref{eq_aux_ma}) has explicit solutions.
	The main contribution of this article, however, says that the existence of approximate solutions to (\ref{eq_aux_ma}) is -- quite surprisingly -- unobstructed. 
	\begin{thm}\label{thm_main2}
		For any $\epsilon > 0$, there is a relatively Kähler $(1,1)$-form $\omega_{\epsilon} \in c_1(L)$ over $X$, verifying 
		\begin{equation}\label{eq_thm_main2}
			\int_X
			\Big|
			\omega_{\epsilon}^{n + 1} \wedge \pi^* \omega_B^{m - 1}
			-
			{\rm{HN}}(\omega_{\epsilon}) \cdot \omega_{\epsilon}^n \wedge \pi^* \omega_B^m \cdot (n + 1)
			\Big|
			\leq
			\epsilon.
		\end{equation}
	\end{thm}
	\begin{rem}
		The approximate solutions are, of course, highly non-unique. 
		Indeed, one can easily obtain new approximate solutions by locally perturbing a given one. 
		It is therefore natural to ask whether such approximate solutions can be constructed in a more canonical way -- for instance, as solutions of a suitable geometric flow. 
		We refer the reader to the discussion at the end of Section \ref{sect_j_eq} for further details.
	\end{rem}
	\par 
	We now outline the relationship between Theorems \ref{thm_main} and \ref{thm_main2}. 
	The proof of (\ref{eq_main}) proceeds by establishing both lower and upper bounds on ${\rm{WZW}}(c_1(L) - t \pi^* [\omega_B], \omega_B)$. 
	The lower bound, which is significantly simpler to establish, follows directly from our previous work \cite{FinHNII}.
	To obtain the upper bound, we need to construct some explicit relatively Kähler $(1,1)$-forms in the class $c_1(L) - t \pi^* [\omega_B]$ saturating the established lower bound.
	Let us demonstrate that, for any $\epsilon > 0$, $t \in \real$, the forms $\omega_{\epsilon}$ from Theorem \ref{thm_main2}, verify
	\begin{equation}\label{eq_main_cor}
		{\rm{WZW}}(\omega_{\epsilon} - t \pi^* \omega_B, \omega_B) 
		\leq 
		\int_{x \in \real} |x - t| d \eta^{HN}(x)
		\cdot
		\int_X c_1(L)^n \pi^* [\omega_B]^m
		\cdot
		(n + 1)
		+
		\epsilon.
	\end{equation}
	Once we establish (\ref{eq_main_cor}), we will obtain the upper bound on ${\rm{WZW}}(c_1(L) - t \pi^* [\omega_B], \omega_B)$ from Theorem \ref{thm_main}.
	The verification of (\ref{eq_main_cor}) follows straightforwardly from a result we establish in Proposition \ref{prop_conv_hn_pot}, related with previous works \cite{ChenHNolyg}, \cite{BouckChen}, \cite{NystOkounTest}, \cite{HisamSpecMeas} and \cite{FinSubmToepl}. 
	This result states that for any continuous function $g : \real \to \real$ and any relatively Kähler $(1,1)$-form $\omega \in c_1(L)$, we have
	\begin{equation}\label{eq_hn_pt_func}
		\int g({\rm{HN}}(\omega)) \cdot \omega^n \wedge \pi^* \omega_B^m
		=
		\int_{x \in \real} g(x) d \eta^{HN}(x)
		\cdot
		\int_X c_1(L)^n \pi^* [\omega_B]^m.
	\end{equation}
	The reader will immediately check that (\ref{eq_main_cor}) is a direct consequence of Theorem \ref{thm_main2} and (\ref{eq_hn_pt_func}), applied for the continuous function $g(x) := |x - t|$, $x \in \real$.
	\par 
	We further assert that Theorem \ref{thm_main2} provides a generalization of a part of Corollary \ref{cor_wzw}. Specifically, by setting $g(x) := |x - t|$, $x \in \real$, in (\ref{eq_hn_pt_func}), it becomes apparent that $\eta^{HN}$ is a Dirac mass at $t$ if and only if ${\rm{HN}}(\omega)$ coincides with $t$ almost everywhere. 
	Consequently, (\ref{eq_thm_main2}) specializes to (\ref{eq_cor_wzw}). 
	These observations collectively suggest that the forms in Theorem \ref{thm_main2} serve as the closest analogs to actual solutions of the Wess-Zumino-Witten equation in cases where true solutions are absent. 
	Thus, Theorem \ref{thm_main2} offers a resolution to Question B).
	\par 
	To conclude, we briefly outline our approach to proving Theorem \ref{thm_main2} and situate it within the context of earlier research. 
	We construct the forms $\omega_{\epsilon}$, $\epsilon > 0$, from Theorem \ref{thm_main2} by geometric quantization, i.e. from certain Hermitian metrics on $E_k$, for $k \in \nat$ big enough.
	\par 
	The use of geometric quantization for investigating PDEs was introduced by Donaldson \cite{DonaldScalBalan}, but our approach diverges from his in a significant way. 
	We establish the existence of a solution through geometric quantization, whereas in \cite{DonaldScalBalan}, the existence of a solution was assumed to obtain convergence in the quantization procedure. 
	\par 
	Our approach builds upon several results.
	First, we make extensive use of the fact initially established by Chen \cite{ChenHNolyg} and further developed by the author \cite{FinHNI}, saying that Harder-Narasimhan filtrations constitute a bounded, submultiplicative filtration on $\oplus_{k = 0}^{+\infty} E_k$ away from a negligible subset of $B$. 
	This result allows us to integrate the developments concerning submultiplicative filtrations from Boucksom-Chen \cite{BouckChen}, Witt Nystr{\"o}m \cite{NystOkounTest}, Hisamoto \cite{HisamSpecMeas}, and the author \cite{FinTits}, \cite{FinSubmToepl} into the study of the Wess-Zumino-Witten equation.
	\par 
	Second, our approach is deeply rooted on the result due to Atiyah-Bott \cite{AtiyahBott}, Daskalopoulos-Wentworth, \cite{DaskWent} and Sibley \cite{SibleyYangMills}, establishing the existence of \textit{$L^1$-approximate critical Hermitian structures} on holomorphic vector bundles, cf. (\ref{eq_appr_crhs}).
	These metrics serve us to construct the Hermitian metrics on $E_k$ leading to $\omega_{\epsilon}$.
	Indeed, our main technical contribution consists in establishing that these $L^1$-approximate critical Hermitian structures can be constructed on $E_k$, $k \in \nat$, while respecting the algebraic structure of the ring bundle $\oplus_{k = 0}^{+ \infty} E_k$ in a certain sense.
	This hinges on the well-known fact that $\oplus_{k = 0}^{+ \infty} E_k$ is a finitely generated ring bundle, and on the already mentioned result concerning submultiplicativity of Harder-Narasimhan filtrations.
	Our forms $\omega_{\epsilon}$ arise from the dequantization of geodesic rays associated with Harder-Narasimhan filtrations and emanating from $L^1$-approximate critical Hermitian structures, and submultiplicativity of Harder-Narasimhan filtrations is crucial here because it enables the application of certain estimates on these geodesic rays from \cite{FinTits}, which are essential for showing the compatibility with algebraic structure of $\oplus_{k = 0}^{+ \infty} E_k$.
	\par 
	The importance of this compatibility of $L^1$-approximate critical Hermitian structures with the algebraic structure of the ring bundle $\oplus_{k = 0}^{+ \infty} E_k$ in our analysis stems from our previous work \cite{FinSecRing}, where we show that such compatible metrics are necessarily given as the $L^2$-metrics of some positive Hermitian metric on the polarization.
	This connection with $L^2$-metrics allows us to draw on several results from geometric quantization, which form the technical backbone of this article.
	In addition to the now-classical result of Tian \cite{TianBerg} and Dai-Liu-Ma \cite{DaiLiuMa}, we rely significantly on the asymptotic formula for the curvature of direct images by Ma-Zhang \cite{MaZhangSuperconnBKPubl}, as well as on the semiclassical Ohsawa-Takegoshi extension theorem established by the author in \cite{FinOTAs},  \cite{FinOTRed} refining previous results due to Zhang \cite{ZhangPosLinBun} and Bost \cite{BostDwork}.
	\par 
	Let us briefly now explain the structure of this paper.
	In Section \ref{sect_subm_hn_filt}, we review the submultiplicativity property of Harder-Narasimhan filtrations and establish the convergence of Harder-Narasimhan potentials from (\ref{eq_hn_potent}).
	In Section \ref{sect_min_seq}, we establish the lower bound from (\ref{eq_main}) on the Wess-Zumino-Witten functional.
	We do so by studying asymptotically, as $k \to \infty$, the lower bound on the Hermitian Yang-Mills functionals of $E_k$.
	We describe a precise relation between the minimization of the Wess-Zumino-Witten functional and Hermitian Yang-Mills functionals.
	We describe a construction of the sequence of $(1, 1)$-forms $\omega_{\epsilon}$, which provide solutions to Theorem \ref{thm_main2}.
	We prove that this sequence solves Theorem \ref{thm_main2} in Section \ref{sect_dequant}, modulo a number of technical results which are treated in Sections \ref{sect_curv}-\ref{sect_geod_appr}.
	In Section \ref{sect_mr}, we describe an application of Theorem \ref{thm_main} which gives a Mehta-Ramanathan type formula for the Wess-Zumino-Witten functional.
	In Section \ref{sect_dem_conj}, we establish Corollary \ref{thm_andgr}. 
	Finally, in Section \ref{sect_j_eq}, we describe a connection between Corollary \ref{cor_wzw}, Kobayashi-Hitchin correspondence and Hessian quotient equations.
	\par 
	\textbf{Notations}. 
	We use the notation $N_k := \rk{E_k}$ throughout the text. 
	For $b \in B$, we denote by $X_b$, $E_{k, b}$, etc., the fibers of $X, E_k$, etc., at $b$. 
	For a Hermitian vector bundle $(E, h^E)$ on $B$, a bounded section $A$ of $\enmr{E}$ and a positive volume form $\eta$ on $B$, we define
	\begin{equation}
		\| A \cdot \eta \|_{L^1(B, h^E)} = \int_{b \in B} \| A(b) \| \cdot \eta(b), 
		\qquad 
		\| A \cdot \eta \|_{L^1(B, h^E)}^{{\rm{tr}}} = \int_{b \in B} \tr{|A(b)|} \cdot \eta(b),
	\end{equation}
	where $\| \cdot \|$ is the subordinate operator norm, calculated with respect to $h^E$, and $|A|$ is the absolute value of an operator, defined as $\sqrt{A A^*}$.
	Clearly, $\| \cdot \|_{L^1(B, h^E)}$ is a norm. 
	In fact, $\| \cdot \|_{L^1(B, h^E)}^{{\rm{tr}}}$ is also a norm; the triangle inequality is satisfied by Ky Fan inequalities, cf. \cite[Exercise II.1.15]{BhatiaMatr}.
	\par 
	Let $(V, H)$ be a Hermitian vector space.
	For Hermitian $A_0, A_1 \in \enmr{V}$, we note $A_0 \geq A_1$ if the difference $A_0 - A_1$ is positive semi-definite.
	When the choice of the Hermitian structure is not clear from the context, we use notation $A_0 \geq_H A_1$. 
	\par 
	We endow ${\rm{Sym}}^l V$ with a Hermitian metric ${\rm{Sym}}^l H$ induced by the induced metric on $V^{\otimes l}$ and the inclusion ${\rm{Sym}}^l V \to V^{\otimes l}$, defined as
	\begin{equation}\label{eq_sym_emb_tens}
		v_1 \odot \ldots \odot v_l \mapsto \frac{1}{l!} \sum v_{\sigma(1)} \otimes \ldots \otimes v_{\sigma(l)},
	\end{equation}
	where the sum runs over all permutations $\sigma$ on $l$ indices.
	Clearly, if $v_1, \cdots, v_r$ form an orthonormal basis of $V$, then $\sqrt{l! / \alpha!} \cdot v^{\odot \alpha}$, $\alpha \in \nat^l$, $|\alpha| = l$, forms an orthonormal basis of ${\rm{Sym}}^l V$ with respect to ${\rm{Sym}}^l H$.
	Similarly, for an arbitrary filtration $\mathcal{F}$ of $V$, we define the filtration ${\rm{Sym}}^l \mathcal{F}$ on ${\rm{Sym}}^l V$.
	\par 
	For any $A \in \enmr{V}$, we define ${\rm{Sym}}^l A \in \enmr{{\rm{Sym}}^l V}$ as the symmetrization of the map $l \cdot A \otimes {\rm{Id}}_V \otimes \cdots \otimes {\rm{Id}}_V$.
	In other words, if $A$ is self-adjoint and $(v_1, \cdots, v_r)$ form a basis of $V$, consisting of eigenvectors of $A$ corresponding to the eigenvalues $\lambda := (\lambda_1, \cdots, \lambda_r)$, then $v^{\odot \alpha}$, $\alpha = (\alpha_1, \ldots, \alpha_r) \in \nat^l$, $|\alpha| = l$, forms a basis of eigenvectors of ${\rm{Sym}}^l A$ corresponding to the eigenvalues $\alpha \cdot \lambda := \alpha_1 \lambda_1 + \cdots + \alpha_r \lambda_r$.
	\par 
	Consider now a surjection $p : V \to Q$ between two complex vector bundles.
	Once we fix a Hermitian metric $H$ on $V$, one can naturally identify $V$ with $Q \oplus \ker p$ using the dual to $p$ map $p^* : Q \to V$.
	Using this identification, for any $A \in \enmr{V}$, we then can define the operator $A|_Q \in \enmr{Q}$ by $A|_Q (q) = p(A( p^* (q)))$.
	\par 
	A filtration $\mathcal{F}$ on $V$ is a map from $\real$ to vector subspaces of $V$, $t \mapsto \mathcal{F}_t V$, verifying $\mathcal{F}_t V \subset \mathcal{F}_s V$ for $t > s$, and such that $\mathcal{F}_t V  = V$ for sufficiently small $t$ and $\mathcal{F}_t V = \{0\}$ for sufficiently big $t$.
	We always assume that it left-continuous, i.e. for any $t \in \real$, there is $\epsilon_0 > 0$, such that $\mathcal{F}_t V = \mathcal{F}_{t - \epsilon} V $ for any $0 < \epsilon < \epsilon_0$.
	Sometimes, we define filtrations by prescribing their jumping numbers and respective vector subbundles. 
	In this way, the corresponding map from $\real$ is defined as the only left-continuous map, which is constant between the jumping numbers.
	\par 
	A norm $N_V = \| \cdot \|_V$ on $V$ naturally induces the quotient norm $\| \cdot \|_Q := [N_V]$ on $Q$ as follows
	\begin{equation}\label{eq_defn_quot_norm}
		\| f \|_Q
		:=
		\inf \big \{
		 \| g \|_V
		 :
		 \quad
		 g \in V, 
		 p(g) = f
		\big\},
		\qquad f \in Q.
	\end{equation}
	Similarly, for any filtration $\mathcal{F}$ on $V$, we can form a quotient filtration $[\mathcal{F}]$ on $Q$.
	More precisely, recall that a filtration $\mathcal{F}$ on $V$ defines the norm and weigh functions $\chi_{\mathcal{F}} : V \to [0, +\infty[$,  $\chi_{\mathcal{F}} : V \to ]- \infty, +\infty]$, as follows
	\begin{equation}\label{eq_na_norm}
		w_{\mathcal{F}}(s) := \sup \{ \lambda \in \real : s \in \mathcal{F}_{\lambda} V \},
		\qquad
		\chi_{\mathcal{F}}(s) := \exp(- w_{\mathcal{F}}(s)).
	\end{equation}
	Clearly, $\chi_{\mathcal{F}}$ is a non-Archimedean norm on $V$ with respect to the trivial absolute value on $\comp$, i.e. it satisfies the following axioms: a) $\chi_{\mathcal{F}}(f) = 0$ if and only if $f = 0$,
	b) $\chi_{\mathcal{F}}(\lambda f) = \chi_{\mathcal{F}}(f)$, for any $\lambda \in \comp^*$, $k \in \nat^*$, $f \in V$,
	c) $\chi_{\mathcal{F}}(f + g) \leq \max \{ \chi_{\mathcal{F}}(f), \chi_{\mathcal{F}}(g) \}$, for any $k \in \nat^*$, $f, g \in V$.
	Moreover, any function verifying the above properties is associated with a filtration. 
	If we now use the definition (\ref{eq_defn_quot_norm}) to define the quotient norm $[\chi_{\mathcal{F}}]$ from $\chi_{\mathcal{F}}$, it will satisfy the same properties of a non-Archimedean norm and, hence, defines a filtration, which we denote by $[\mathcal{F}]$.
	\par 
	For a coherent sheaf $\mathscr{E}$ on $B$, we denote by ${\rm{Sat}}(\mathscr{E})$ the saturation of $\mathscr{E}$, defined as the minimal subsheaf with torsion free quotient containing $\mathscr{E}$. 
	A sheaf $\mathscr{E}$ is saturated if ${\rm{Sat}}(\mathscr{E}) = \mathscr{E}$.
	\par 
	\textbf{Acknowledgement}. 
	I am grateful to Sébastien Boucksom, Paul Gauduchon, Duong H. Phong, Lars M. Sektnan, Jacob Sturm, and Richard Wentworth for insightful discussions related to this work. 
	I also acknowledge the support of CNRS and École Polytechnique, and thank the anonymous referee for the helpful suggestions.
	Above all, I deeply wish I could have expressed my gratitude in person to Jean-Pierre Demailly, with whom I first discussed the subject of this paper in late 2021 and who, to my great sorrow, passed away in March 2022. 
	His encouragement and insights were invaluable in bringing this project to completion.
	
	\section{Harder-Narasimhan potentials and submultiplicativity}\label{sect_subm_hn_filt}
	The primary aim of this section is to review the submultiplicativity property of Harder-Narasimhan filtrations, which plays a central role in this article, and to establish, based on this property, the convergence of Harder-Narasimhan potentials from (\ref{eq_hn_potent}).
	\par 
	We fix some definitions first.	
	Consider a graded filtration $\mathcal{F} := \oplus_{k = 0}^{\infty} \mathcal{F}^k$ on the section ring $R(Y, F) := \oplus_{k = 0}^{\infty} H^0(Y, F^{\otimes k})$ of a complex projective manifold $Y$ polarized by an ample line bundle $F$.
	We say that $\mathcal{F}$ is \textit{submultiplicative} if for any $t, s \in \real$, $k, l \in \nat$, we have 
	\begin{equation}\label{eq_sm_cond}
		\mathcal{F}^k_t H^0(Y, F^{\otimes k}) \cdot \mathcal{F}^l_s H^0(Y, F^{\otimes l}) \subset \mathcal{F}^{k + l}_{t + s} H^0(Y, F^{\otimes (k + l)}).
	\end{equation}
	We say that $\mathcal{F}$ is \textit{bounded} if there is $C > 0$, such that for any $k \in \nat^*$, $\mathcal{F}^k_{C k} H^0(Y, F^{\otimes k}) = \{0\}$.
	It is an immediate consequence of the submultiplicativity and the fact that $R(Y, F)$ is a finitely generated ring, cf. \cite[Example 2.1.30]{LazarBookI}, that there is $C > 0$, such that $\mathcal{F}^k_{ - C k} H^0(Y, F^{\otimes k}) = H^0(Y, F^{\otimes k})$.
	For a bounded submultiplicative filtration $\mathcal{F}$, we denote by $\| \mathcal{F} \|$ the minimal constant $C > 0$, such that $\mathcal{F}^k_{ - C k} H^0(Y, F^{\otimes k}) = H^0(Y, F^{\otimes k})$ and $\mathcal{F}^k_{ C k} H^0(Y, F^{\otimes k}) = \{0\}$ for any $k \in \nat^*$.
	\par 
	\begin{sloppypar}
	We fix a positive Hermitian metric $h^F$ on $F$, and denote by ${\textrm{Hilb}}_k(h^F)$ the $L^2$-metric on $H^0(Y, F^{\otimes k})$ induced by $h^F$, see (\ref{eq_l2_prod}).
	For a continuous function $g : \real \to \real$, we consider the \textit{weighted Bergman kernel}, $B_k^{\mathcal{F}, g}(x) \in \real$, $k \in \mathbb{N}$, $x \in Y$, defined as 
	\begin{equation}\label{eq_weight_berg}
		B_k^{\mathcal{F}, g}(x) 
		= 
		\sum_{i = 1}^{N_k} g \Big( \frac{w_{\mathcal{F}^k}(s_{i, k})}{k} \Big) \cdot \big| s_{i, k}(x) \big|_{h^{L^{\otimes k}}}^2
	\end{equation}
	where $N_k := \dim H^0(Y, F^{\otimes k})$ and $s_{i, k}$, $i = 1, \ldots, N_k$, is an orthonormal basis of $(H^0(Y, F^{\otimes k}), {\textrm{Hilb}}_k(h^F))$ adapted to $\mathcal{F}^k$.
	\end{sloppypar}
	\par 
	\begin{sloppypar}
	This definition is modeled after the Bergman kernel, $B_k(x)$, defined as $B_k(x) = \sum_{i = 1}^{N_k} |s_{i, k}(x)|_{h^{F^{\otimes k}}}^2$.
	Recall that a well-known result of Tian, \cite{TianBerg}, says that 
	\begin{equation}\label{eq_thm_tian}
		\frac{1}{k^n} B_k(x) \quad \text{converges uniformly to 1, as } k \to \infty,
	\end{equation}
	see also \cite{ZeldBerg}, \cite{Caltin}, \cite{Bouche}, \cite{MaHol} for more refined convergence statements.
	One of the main results of \cite{FinSubmToepl} generalizes (\ref{eq_thm_tian}) in realms of weighted Bergman kernels, (\ref{eq_weight_berg}).
	To state it, recall that Phong-Sturm \cite[Theorem 3]{PhongSturmDirMA} and Ross-Witt Nystr{\"o}m \cite{RossNystAnalTConf} constructed a \textit{geodesic ray} $h^{\mathcal{F}}_t$, $t \in [0, +\infty[$, of Hermitian metrics on $F$, emanating from $h^F$.
	The term geodesic here stands for the fact that the resulting ray of metrics is a metric geodesic in the space of positive metrics on $F$, endowed with the so-called Mabuchi distance, \cite{Mabuchi}.
	\end{sloppypar}
	\par 
	To recall the definition of the geodesic ray, we first recall a much simpler construction of geodesic rays on a complex vector space $V$, $\dim V = r$. 
	We fix a filtration $\mathcal{F}$ and say that the Hermitian products $H_s$, $s \in [0, +\infty[$, on $V$ form a \textit{geodesic ray} departing from $H$ associated with $\mathcal{F}$, if $e_i \cdot \exp(s w_{\mathcal{F}}(e_i) / 2)$, $i = 1, \ldots, r$, form an orthonormal basis for $H_s$, where $e_1, \ldots, e_r$ is an orthonormal basis on $(V, H)$ adapted to the filtration $\mathcal{F}$ in the sense as described in (\ref{eq_hn_potent}).
	\par 
	The construction of the geodesic rays also makes sense in the family setting: when a vector space is replaced by a vector bundle over a manifold, a Hermitian product is replaced by a Hermitian metric, and the filtration is replaced by filtrations by subsheaves over the manifold. 
	When the filtration is given by subbundles (and not by subsheaves), and the Hermitian metric is smooth, it is immediate to see that the the resulting ray is a ray of smooth metrics.
	\par 
	The terminology “geodesic ray" comes from the fact that the above rays are metric geodesics in the space of all Hermitian products on $V$ for the invariant metric coming from the $SL(V) / SU(V)$-homogeneous structure.
	Finite segments of these rays will be called geodesics.
	\par 
	Now, recall that an arbitrary Hermitian norm $H_k$ on $H^0(Y, F^{\otimes k})$, for $k \in \nat$ so that $F^{\otimes k}$ is very ample, induces a positive metric $FS(H_k)$ on $F^{\otimes k}$, constructed as follows.
	Consider the Kodaira embedding ${\rm{Kod}}_k : Y \hookrightarrow \mathbb{P}(H^0(Y, F^{\otimes k})^*)$.
	We denote by $\mathscr{O}(1)$ the hyperplane bundle on $\mathbb{P}(H^0(Y, F^{\otimes k})^*)$, and define the metric $FS(H_k)$ on $F^{\otimes k}$ as the pull-back of the Fubini-Study metric on $\mathscr{O}(1)$ induced by $H_k$ through the isomorphism ${\rm{Kod}}_k^* \mathscr{O}(1) \to F^{\otimes k}$.
	Alternatively, $FS(H_k)$ is the only metric on $F^{\otimes k}$, which for any $x \in Y$, and for an orthonormal basis $s_1, \ldots, s_{N_k}$ of $(H^0(Y, F^{\otimes k}), H_k)$ satisfies the following equation
	\begin{equation}\label{eq_fs_alt_defn}
		\sum_{i = 1}^{N_k} \big| s_i(x) \big|^2_{FS(H_k)} = 1.
	\end{equation}
	\par 
	Now, for any $t \in [0, +\infty[$, $k \in \nat$, we define, following Phong-Sturm \cite{PhongSturmDirMA} and Ross-Witt Nystr{\"o}m \cite{RossNystAnalTConf}, $H^{\mathcal{F}}_{t, k}$ as the (geodesic) ray of Hermitian norms on $H^0(Y, F^{\otimes k})$ emanating from ${\rm{Hilb}}_k(h^F)$ and associated with the restriction $\mathcal{F}^k$ of $\mathcal{F}$ to $H^0(Y, F^{\otimes k})$.
	We denote by $h^{\mathcal{F}}_t$, $t \in [0, +\infty[$, the ray of metrics on $L$, constructed as follows
	\begin{equation}\label{eq_geod_ray_filt}
		h^{\mathcal{F}}_t := \Big( \lim_{k \to \infty} \inf_{l \geq k} \big( FS(H^{\mathcal{F}}_{t, l})^{\frac{1}{l}} \big) \Big)_{*}.
	\end{equation}
	\par 
	In general $h^{\mathcal{F}}_t$ is not smooth. But it is bounded for any $t \in [0, + \infty[$, and one can always define its derivative at $t = 0$, $\dot{h}^{\mathcal{F}}_0 := (h^{\mathcal{F}}_0)^{-1} \frac{d}{dt} h^{\mathcal{F}}_t|_{t = 0}: X \to \real$, which is also bounded.
	More specifically, in \cite[\S 2.2]{BernBrunnMink}, \cite[Theorem 9.2]{RossNystAnalTConf}, authors established that $h^{\mathcal{F}}_t$ converges to $h^{\mathcal{F}}_0$, as $t \to 0$, uniformly on $Y$. 
	Due to convexity in $t$-variable of the potential of these metrics, cf. \cite[Theorem I.5.13]{DemCompl}, the one-sided derivative at $t := 0$ is well-defined.
	We denote $\phi(h^L, \mathcal{F}) = - \dot{h}^{\mathcal{F}}_0$ for brevity.
	It is easy to establish, cf. \cite[Lemma 2.4]{FinTits}, that the following bound holds
	\begin{equation}\label{eq_geod_ray_bnd}
		\sup_{x \in X} \big| \phi(h^L, \mathcal{F})(x) \big|
		\leq
		\| \mathcal{F} \|.
	\end{equation}
	\par 
	\begin{thm}[{\cite[Theorem 1.1]{FinSubmToepl}}]\label{thm_berg_conv}
		For a bounded submultiplicative filtration $\mathcal{F}$ on $R(Y, F)$, the sequence of functions $x \mapsto \frac{1}{k^n} B_k^{\mathcal{F}, g}(x)$, $x \in Y$, $k \in \nat$, is uniformly bounded and converges pointwise to a function which equals $g(\phi(h^L, \mathcal{F}))$ almost everywhere.
	\end{thm}
	\par 
	Let us explain the relation between Theorem \ref{thm_berg_conv} and the respective convergence for jumping measures.
	Remark the following basic identity 
	\begin{equation}\label{eq_basic_id}
		\frac{1}{n!} \int B_k^{\mathcal{F}, g}(x) c_1(L, h^L)^n = {\rm{Tr}} \Big[ g \Big( \frac{A({\textrm{Hilb}}_k(h^L), \mathcal{F}^k)}{k} \Big) \Big].
	\end{equation}
	We define the jumping measures, $\mu_{\mathcal{F}, k}$, on $\real$ associated with $\mathcal{F}^k$ as 
	\begin{equation}\label{eq_jump_meas_d}
		\mu_{\mathcal{F}, k} := \frac{1}{N_k} \sum_{j = 1}^{N_k} \delta \Big[ \frac{e_{\mathcal{F}}(j, k)}{k} \Big], 
	\end{equation}
	where $\delta[x]$ is the Dirac mass at $x \in \real$ and $e_{\mathcal{F}}(j, k)$ are the \textit{jumping numbers}, defined as follows
	\begin{equation}\label{eq_defn_jump_numb}
		e_{\mathcal{F}}(j, k) := \sup \Big\{ t \in \real : \dim \mathcal{F}_t H^0(X, L^{\otimes k}) \geq j \Big\}.
	\end{equation}
	\par 
	Directly from Theorem \ref{thm_berg_conv}, Lebesgue dominated convergence theorem and the asymptotic Riemann-Roch-Hirzebruch theorem saying that $N_k \sim k^n \cdot \int c_1(F)^n / n!$, we see that the sequence of measures $\mu_{\mathcal{F}, k}$ converge weakly, as $k \to \infty$, to a measure $\mu_{\mathcal{F}}$ on $\real$, which satisfies
	\begin{equation}\label{eq_weak_lim_geod_r}
		\int_{\real} g(x) d \mu_{\mathcal{F}}(x)
		=
		\frac{\int_Y g(\phi(h^L, \mathcal{F})) c_1(F, h^F)^n}{\int_Y c_1(F)^n}.
	\end{equation}
	Weak convergence of jumping measures was first established by Chen \cite{ChenHNolyg} and Boucksom-Chen \cite{BouckChen}. 
	Subsequently, Witt Nyström \cite{NystOkounTest} proved (\ref{eq_weak_lim_geod_r}) for filtrations associated with a $\mathbb{C}^*$-action, Hisamoto extended it in \cite{HisamSpecMeas} for finitely generated filtrations, and the author \cite[Theorem 5.4]{FinSecRing} further extended it for bounded submultiplicative filtrations, as stated above.
	\par 
	We need to consider the family version of Theorem \ref{thm_berg_conv}.
	Consider a holomorphic submersion $\pi : X \to B$ between compact Kähler manifolds $X$ and $B$ of dimensions $n + m$ and $m$ respectively.
	For a relatively ample line bundle $L$ over $X$, we denote by $E_k := R^0 \pi_* L^{\otimes k}$.
	For $k \in \nat$, we consider a filtration
	\begin{equation}\label{eq_resol_hn_filt001}
		E_k = \mathcal{F}^{k}_{\lambda_{1, k}} \supset \mathcal{F}^{k}_{\lambda_{2, k}} \supset \cdots \supset \mathcal{F}^{k}_{\lambda_{q_k, k}} \supset \{0\} =: \mathcal{F}^{k}_{\lambda_{q_k + 1, k}},
	\end{equation}
	by coherent subsheaves.
	We assume that there is a subset $S \subset B$, negligible with respect to Lebesgue measure on $B$, such that for any $b \notin S$, the restriction of (\ref{eq_resol_hn_filt001}) to $b$ induces a bounded submultiplicative filtration of $\oplus_{k = 0}^{+\infty} E_{k, b} = R(X_b, L|_{X_b})$.
	Later on, to simplify the notations, we simply call the induced filtration on $\oplus_{k = 0}^{\infty} E_k$ \textit{bounded and submultiplicative away from $S$}.
	\par 
	Remark that the jumping numbers of the restriction of the filtration (\ref{eq_resol_hn_filt001}) at a generic point over the base $B$ do not depend on the choice of the point.
	Due to this, we can define the jumping measures, $\mu_{\mathcal{F}, k}$, just as in (\ref{eq_jump_meas_d}), and by (\ref{eq_weak_lim_geod_r}), they will converge weakly, as $k \to \infty$, to a measure on $\real$ that we denote by $\mu_{\mathcal{F}}$. 
	\par 
	To make a connection between the theory of submultiplicative filtrations and the Harder-Narasimhan filtrations, we need the following result, which for $\dim B = 1$ is due to Chen \cite{ChenHNolyg}, and for $\dim B \geq 2$ is due to the author, see \cite{FinHNII}, cf. also \cite{FinHNI} for the projective setting. 
	Below, we use the notations from (\ref{eq_resol_hn_filt00}) for the Harder-Narasimhan filtration of $E_k$.
	We introduce the following proper analytic subsets of $B$:
	\begin{equation}\label{eq_sk0_defn}
		S_k^0 := \cup_{i = 1}^{q_k} {\rm{Singsupp}}(\mathcal{F}^{HN, k}_{\lambda_i}).
	\end{equation}
	Remark that when $\dim B = 1$, the sets $S_k^0$ are empty.
	\begin{thm}\label{thm_HN_subm}
		The Harder-Narasimhan filtrations (\ref{eq_resol_hn_filt00}) on $E_k$, $k \in \nat$, induce on $\oplus_{k = 0}^{\infty} E_k$ the filtration which is bounded and submultiplicative away from $\cup_{k = 0}^{+ \infty} S_k^0$.
	\end{thm}
	\par 
	Directly from Theorem \ref{thm_HN_subm} and the result described after (\ref{eq_resol_hn_filt001}), we deduce the weak convergence of the probability measures $\eta_k^{HN}$ from (\ref{eq_eta_defn}).
	To discuss the convergence of the Harder-Narasimhan potentials, we fix a relatively positive Hermitian metric $h^L$ on $L$, $\omega := c_1(L, h^L)$, a continuous function $g : \real \to \real$, and define the \textit{fiberwise weighted Bergman kernel}, $B_k^{\mathcal{F}, g, \pi}(x) \in \real$, $k \in \mathbb{N}$, $x \in X$, by gluing (\ref{eq_weight_berg}).
	\par 
	We denote by $\phi^{\pi}(h^L, \mathcal{F})$ the fiberwise geodesic ray associated with the restriction of (\ref{eq_resol_hn_filt001}) to the fibers. 
	Remark that $\phi^{\pi}(h^L, \mathcal{F})$  is only well defined for $x \in X$, so that $\pi(x) \notin S$, but since $S$ is a negligible set, and the bound (\ref{eq_geod_ray_bnd}) holds, $\phi^{\pi}(h^L, \mathcal{F})$ makes sense as an element of $L^{\infty}(X)$. 
	\begin{prop}\label{prop_conv_berg_family}
		The sequence of functions $x \mapsto \frac{1}{k^n} B_k^{\mathcal{F}, g, \pi}(x)$, $x \in X$, $k \in \nat$, is uniformly bounded over $X$, and converges almost everywhere to $g(\phi^{\pi}(h^L, \mathcal{F}))$.
		Moreover, for any relatively Kähler $(1,1)$-form $\omega \in c_1(L)$, we have
		\begin{equation}\label{eq_hn_pt_func_gen}
			\int g(\phi^{\pi}(h^L, \mathcal{F})) \cdot \omega^n \wedge \pi^* \omega_B^m
			=
			\int_{x \in \real} g(x) d \mu_{\mathcal{F}}(x)
			\cdot
			\int_X c_1(L)^n \pi^* [\omega_B]^m.
		\end{equation}
	\end{prop}
	\begin{rem}\label{rem_conv_berg_family}
		Note that the function $B_k^{\mathcal{F}, g, \pi}(x)$ is continuous outside the singular locus of the filtration.
		Proposition \ref{prop_conv_berg_family} then implies that $\phi^{\pi}(h^L, \mathcal{F})$ is measurable, since it arises as the pointwise limit of functions that are continuous almost everywhere.
	\end{rem}
	\begin{proof}
		Directly from the definition of $B_k^{\mathcal{F}, g, \pi}(x)$, we deduce that 
		\begin{equation}\label{eq_bnd_hn1}
			|B_k^{\mathcal{F}, g, \pi}(x)| \leq B^{\pi}_k(x) \cdot \sup_{|x| \leq  \| \mathcal{F} \|} |g(x)|,
		\end{equation}
		where $B^{\pi}_k(x)$ is the fiberwise Bergman kernel.
		Recall the following family version of (\ref{eq_thm_tian}):
		\begin{equation}\label{eq_thm_tian_fam}
			\frac{1}{k^n} B^{\pi}_k(x) \quad \text{converges uniformly to 1, as } k \to \infty,
		\end{equation}
		which follows directly from the proof \cite{DaiLiuMa}, cf. also \cite{MaZhangSuperconnBKPubl}.
		The boundness statement now follows directly from (\ref{eq_bnd_hn1}) and (\ref{eq_thm_tian_fam}).
		\par 
		By applying Theorem \ref{thm_berg_conv} to each fiber away from $S$, we deduce that for any $b \notin S$, the function $B_k^{\mathcal{F}, g, \pi}$ converges pointwise over $X_b$, as $k \to \infty$.
		Moreover, the limit coincides almost everywhere with $g(\phi^{\pi}(h^L, \mathcal{F}))$.
		Since the set $S$ is negligible, by the measurability from Remark \ref{rem_conv_berg_family} and Fubini's theorem, we deduce that the sequence of functions $x \mapsto \frac{1}{k^n} B_k^{\mathcal{F}, g, \pi}(x)$, $x \in X$, $k \in \nat$, converges almost everywhere to $g(\phi^{\pi}(h^L, \mathcal{F}))$.
		The identity (\ref{eq_hn_pt_func_gen}) is then established exactly as (\ref{eq_weak_lim_geod_r}), the only difference is that we add the integration along $B$ as well.
	\end{proof}
	\par
	Directly from Theorems \ref{thm_berg_conv}, \ref{thm_HN_subm} and Proposition \ref{prop_conv_berg_family}, we conclude the following.
	\begin{cor}\label{prop_conv_hn_pot}
		The sequence of functions $x \mapsto \frac{1}{k^{n + 1}} {\rm{HN}}_k(\omega)(x)$, $x \in X$, $k \in \nat$, is uniformly bounded over $X$, and converges almost everywhere to ${\rm{HN}}(\omega)$, which is a measurable function.
		Moreover, the identity (\ref{eq_hn_pt_func}) holds.
	\end{cor}

	\section{A minimizing sequence for the Wess-Zumino-Witten functional}\label{sect_min_seq}
	The main goal of this section is to give an outline of the proofs of Theorems \ref{thm_main}, \ref{thm_main2}.
	More precisely, we first show that the lower bound on the Wess-Zumino-Witten functional follows directly from the previous work \cite{FinHNII} of the author.
	We then describe a relation between the fact that this lower bound is sharp and the fact that, asymptotically, the infimum of the Hermitian Yang-Mills functional on direct images is saturated by the $L^2$-metrics.
	Finally, we describe how to construct the forms $\omega_{\epsilon}$ from Theorem \ref{thm_main2}.
	\par 
	We follow the notations introduced before Theorem \ref{thm_main}.
	Fix an arbitrary smooth closed relatively positive $(1, 1)$-form $\alpha$ in the class $c_1(L) - t \pi^* [\omega_B]$, $t \in \real$. 
	\begin{prop}\label{prop_low_bnd}
		For any $\alpha$ as above, we have
		\begin{equation}\label{eq_low_bnd}
			{\rm{WZW}}(\alpha, \omega_B) \geq \int_{x \in \real} |x - t| d \eta^{HN}(x)
			\cdot
			\int_X c_1(L)^n \pi^* [\omega_B]^m
			\cdot
			(n + 1).
		\end{equation}
	\end{prop}
	Let us introduce some notations which will be useful in the proof of Proposition \ref{prop_low_bnd} and later on.
	We define the $(1, 1)$-form $\omega := \alpha + t \pi^* \omega_B$.
	As $\omega$ is positive along the fibers, it provides a (smooth) decomposition of the tangent space $TX$ of $X$ into the vertical component $T^V X$, corresponding to the tangent space of the fibers, and the horizontal component $T^H X$, corresponding to the orthogonal complement of $T^V X$ with respect to $\omega$.
	The form $\omega$ then decomposes as $\omega = \omega_V + \omega_H$, $\omega_V \in \ccal^{\infty}(X, \wedge^{1, 1} T^{V*} X)$, $\omega_H \in \ccal^{\infty}(X, \wedge^{1, 1} T^{H*} X)$. 
	Upon the natural identification of $T^H X$ with $\pi^* TB$, we may view $\omega_H$ as an element from $\ccal^{\infty}(X, \wedge^{1, 1} \pi^* T^* B)$.
	We define $\wedge_{\omega_B} \omega_H \in \ccal^{\infty}(X)$, as $\wedge_{\omega_B} \omega_H := \omega_H \wedge \omega_B^{m - 1} / \omega_B^m$.
	We also fix a relatively positive Hermitian metric $h^L$ on $L$ verifying $c_1(L, h^L) = \omega$.
	\par 
	Recall now that for a Hermitian metric $h^E$ on a holomorphic vector bundle $E$ over $B$, for any $t \in \real$, the \textit{Hermitian Yang-Mills functional} is defined as
	\begin{equation}\label{eq_HYM}
		{\rm{HYM}}_t(E, h^E, \omega_B) 
		:=
		\Big\| \frac{\imun}{2 \pi} R^{h^E} \wedge \omega_B^{m - 1} - t {\rm{Id}}_E \cdot \omega_B^m \Big\|_{L^1(B, h^E)}^{{\rm{tr}}},
	\end{equation}
	where here and after $R^{h^E}$ is the curvature of the Chern connection of $(E, h^E)$.
	Remark that our terminology is slightly different from the generally accepted one, where the Hermitian Yang-Mills functional is related to the $L^2$-norm instead of the $L^1$-norm. 
	\begin{proof}[Proof of Proposition \ref{prop_low_bnd}]
		We will show that Proposition \ref{prop_low_bnd} is an easy consequence of a more general result giving lower bounds on the Fibered Yang-Mills functionals introduced in \cite[Theorem 1.7]{FinHNII}.
		Directly from the definition of $\wedge_{\omega_B} \omega_H$, we have
		\begin{equation}\label{eq_omega_h_wed_defn}
			\alpha^{n + 1} = ( \wedge_{\omega_B} \omega_H - t ) \cdot \omega^n \wedge \pi^* \omega_B^m  \cdot (n + 1).
		\end{equation}
		Hence, by the relative positivity of $\omega$, we have
		\begin{equation}\label{eq_wzw_fym}
			{\rm{WZW}}(\alpha, \omega_B) = \int_X |\wedge_{\omega_B} \omega_H - t| \cdot \omega^n \wedge \pi^* \omega_B^m  \cdot (n + 1).
		\end{equation}
		But for a Hermitian metric $h^L$ on $L$, such that $\omega$ coincides with the first Chern form, $c_1(L, h^L)$, of $L$, the right-hand side of (\ref{eq_wzw_fym}) corresponds (up to a multiplication by $(n + 1)$) to the Fibered Yang-Mills functional, ${\rm{FYM}}_{1, t}(\pi, h^L)$, introduced by the author in \cite[(1.5)]{FinHNII}.
		The result now 	follows directly from (\ref{eq_wzw_fym}) and \cite[Theorem 1.7]{FinHNII}.
		For further purposes, let us recall the crucial steps from the argument.
		First, by \cite[(2.15)]{FinHNII} and (\ref{eq_wzw_fym}), we obtain
		\begin{equation}\label{eq_wzw_hym}
			\lim_{k \to \infty} \frac{1}{k N_k} {\rm{HYM}}_{t k}(E_k, {\rm{Hilb}}_k^{\pi}(h^L), \omega_B)
			=
			\frac{1}{\pi_* c_1(L)^n \cdot (n + 1)} {\rm{WZW}}(\alpha, \omega_B),
		\end{equation}
		where the $L^2$-norm ${\rm{Hilb}}_k^{\pi}(h^L)$ was defined in (\ref{eq_l2_prod}).
		From the lower bounds on the Hermitian Yang-Mills functional due to Atiyah-Bott \cite[Proposition 8.20]{AtiyahBott} (for $\dim B = 1$) and Daskalopoulos-Wentworth \cite[Lemma 2.17, Corollary 2.22, Proposition 2.25]{DaskWent} (for general Kähler $B$), applied for $E_k$ for $k$ big enough, so that $E_k$ is locally free, for any Hermitian metric $H_k$ on $E_k$, we have
		\begin{equation}\label{eq_low_bnd_he}
			\frac{1}{k N_k} {\rm{HYM}}_{t k}(E_k, H_k, \omega_B)
			\geq
			\int_{x \in \real} |x - t| d \eta_k^{HN}(x)
			\cdot
			\int_B [\omega_B]^m.
		\end{equation}
		From the weak convergence, as $k \to \infty$, of the measures $\eta_k^{HN}$, (\ref{eq_eta_defn}), established by Chen in \cite[Theorem 4.3.6]{ChenHNolyg} (in the case $\dim B = 1$) and then by the author \cite[Theorem 1.5]{FinHNII} (for general Kähler $B$), we conclude that 
		\begin{equation}\label{eq_low_bnd_he2}
			\lim_{k \to \infty}
			\int_{x \in \real} |x - t| d \eta_k^{HN}(x)
			=
			\int_{x \in \real} |x - t| d \eta^{HN}(x).
		\end{equation}
		The proof now follows from (\ref{eq_wzw_hym}) and (\ref{eq_low_bnd_he}), (\ref{eq_low_bnd_he2}).
	\end{proof}
	\par 
	Now, Theorem \ref{thm_main} claims that the lower bound established in Proposition \ref{prop_low_bnd} is sharp.
	As described after Theorem \ref{thm_main2}, it suffices to establish Theorem \ref{thm_main2} to show this.
	We will concentrate on this from now on, but before this, let us explain a connection between Theorem \ref{thm_main} and the asymptotic minimization of the Hermitian Yang-Mills functionals.
	\begin{thm}\label{thm_expl_const10001}
		For any $\epsilon > 0$, there is a relatively positive Hermitian metric $h^L_{\epsilon}$ on $L$, and $k_0 \in \nat$, such that for any $k \geq k_0$, $t \in \real$, we have
		\begin{equation}\label{eq_expl_const10001}
			{\rm{HYM}}_{t k}(E_k, {\rm{Hilb}}_k^{\pi}(h^L_{\epsilon}), \omega_B)
			\leq
			\inf_{H_k} {\rm{HYM}}_{t k}(E_k, H_k, \omega_B)
			+
			\epsilon k N_k,
		\end{equation}
		where the infimum is taken over all Hermitian metrics $H_k$ on $E_k$.
	\end{thm}
	\begin{rem}
		a) By the sharpness of the bound (\ref{eq_low_bnd_he}), established in \cite{AtiyahBott}, \cite{DaskWent}, \cite{SibleyYangMills}, and (\ref{eq_low_bnd_he2}), $\inf_{H_k} {\rm{HYM}}_{t k}(E_k, H_k, \omega_B)$ is comparable with $k N_k$.
		\par 
		b)
		Not every Hermitian metric on $E_k$ is the $L^2$-metric of some metric on the line bundle, cf. \cite{JingSunImage}. 
		Hence, even the existence of $h^L_{\epsilon}$, verifying (\ref{eq_expl_const10001}) for one fixed $k \in \nat^*$ seems to be non trivial.
		Theorem \ref{thm_expl_const10001}, claims much more: such $L^2$-metrics can be chosen in a related manner.
	\end{rem}
	\begin{proof}[Proof of Theorem \ref{thm_expl_const10001} assuming Theorem \ref{thm_main}]
		First of all, for any $\epsilon > 0$, Theorem \ref{thm_main} assures the existence of a relatively positive metric $h^L_{\epsilon}$ on $L$, verifying
		\begin{equation}\label{eq_expl_const10001101}
			{\rm{WZW}}(c_1(L, h^L_{\epsilon}) - t \pi^* \omega_B, \omega_B) 
			\leq 
			\int_{x \in \real} |x - t| d \eta^{HN}(x)
			\cdot
			\int_X c_1(L)^n \pi^* [\omega_B]^m
			\cdot
			(n + 1) + \frac{\epsilon}{2}.
		\end{equation}
		It then follows by (\ref{eq_wzw_hym}), (\ref{eq_low_bnd_he}) and (\ref{eq_low_bnd_he2})  that (\ref{eq_expl_const10001}) holds for any $h^L_{\epsilon}$ verifying (\ref{eq_expl_const10001101}).
	\end{proof}
	\par 
	To explain our construction of $\omega_{\epsilon}$ from Theorem \ref{thm_main2}, remark that the Fubini-Study operator from (\ref{eq_fs_alt_defn}) can be considered in the relative setting, meaning that for any $k$ sufficiently large so that $L^{\otimes k}$ is relatively very ample, we can associate for any Hermitian metric $H_k$ on $E_k$ a relatively positive Hermitian metric $FS(H_k)$ on $L^{\otimes k}$ using the relative Kodaira embedding ${\rm{Kod}}_k : X \hookrightarrow \mathbb{P}(E_k^*)$, which can be put into the following commutative diagram
	\begin{equation}\label{eq_kod}
	\begin{tikzcd}
		X \arrow[hookrightarrow]{r}{{\rm{Kod}}_k} \arrow{rd}{\pi} & \mathbb{P}(E_k^*) \arrow{d}{p} \\
 		& B.
	\end{tikzcd}
	\end{equation}
	and the isomorphism ${\rm{Kod}}_k^* \mathscr{O}(1) \to L^{\otimes k}$, where $\mathscr{O}(1)$ is the relative hyperplane bundle on $\mathbb{P}(E_k^*)$.
	\par 
	Now, our construction of $\omega_{\epsilon}$ from Theorem \ref{thm_main2} will be done by \textit{dequantization} (i.e. an application of the Fubini-Study operator) to some sequence of metrics on $E_k$, which saturate the lower bound on the respective Hermitian Yang-Mills functionals.
	One of the main difficulties in our analysis lies in the fact that these Hermitian metrics on $E_k$ have a priori nothing to do with the Hermitian metrics constructed by the quantization (the $L^2$-metrics), which were used in (\ref{eq_wzw_hym}) to get the lower bounds for the Wess-Zumino-Witten functional.
	\par 
	We will now describe a specific choice of the minimizing sequence of metrics for the Hermitian Yang-Mills functional on $E_k$ one has to choose.
	For this, let us recall a definition of \textit{approximate critical Hermitian structures}.
	Roughly, an approximate critical Hermitian structure on a vector bundle $E$ over $B$ is a Hermitian metric on $E$, which is in some sense well-adapted to the Harder-Narasimhan filtration, $\mathcal{F}^{HN}$, of $E$. 
	\par 
	In order to state the definition precisely, we first define the weight operator of a filtration, which is another object playing a fundamental role in this paper.
	Remark that any (decreasing) filtration $\mathcal{F}$ on the Hermitian vector space $(V, H)$ induces the \textit{weight operator} $A(H, \mathcal{F}) \in {\textrm{End}}(V)$, as
	\begin{equation}\label{eq_weight_op_dd}
		A(H, \mathcal{F}) e_i = w_{\mathcal{F}}(e_i) \cdot e_i,
	\end{equation}
	where $e_1, \ldots, e_{r}$, $r := \dim V$, is an orthonormal basis of $(V, H)$ adapted to the filtration $\mathcal{F}$ in the same sense as described after (\ref{eq_l2_prod}).
	\par 
	The definition (\ref{eq_weight_op_dd}) also makes sense in the family setting: when a vector space is replaced by a vector bundle over a manifold, a Hermitian product is replaced by a Hermitian metric, and the filtration is replaced by filtrations by subsheaves over the manifold. 
	When the filtration is given by subbundles (and not by subsheaves), and the Hermitian metric is smooth, it is immediate to see that the weight operator becomes a smooth section of the respective endomorphism bundle.
	\par 
	Following Kobayashi \cite{KobaVB}, we say that a Hermitian metric $h^E$ on a holomorphic vector bundle $E$ over $B$ is a \textit{critical Hermitian structure} on $E$ if the curvature of it satisfies $\frac{\imun}{2 \pi}   R^{h^E} \wedge \omega_B^{m - 1} = A(h^E, \mathcal{F}^{HN}) \cdot \omega_B^m$, where $b \mapsto A(h^E, \mathcal{F}^{HN})(b) \in {\rm{End}}(E_b)$ is the weight operator associated with the Harder-Narasimhan filtration.
	Critical Hermitian structures correspond to the minimizers of the Hermitian-Yang-Mills functional. 
	Unfortunately, these do not exist on arbitrary vector bundles, see \cite{UhlYau}, \cite{DonaldASD} and \cite[Theorem 4.3.27]{KobaVB}, and so cannot be used for our purposes.
	\par 
	\begin{sloppypar}
	To circumvent this, following Daskalopoulos-Wentworth, \cite{DaskWent}, we say that $h^E$ is an \textit{$L^1$ $\delta$-approximate critical Hermitian structure} on $E$ if
	\begin{equation}\label{eq_appr_crhs}
		\Big\|
			\frac{\imun}{2 \pi} R^{h^E} \wedge \omega_B^{m - 1}
			-
			A(h^E, \mathcal{F}^{HN}) \cdot \omega_B^m
		\Big\|_{L^1(B, h^E)}
		\leq
		\delta.
	\end{equation}
	A result of Atiyah-Bott \cite[proof of Proposition 8.20]{AtiyahBott} (for $\dim B = 1$), Daskalopoulos-Wentworth, \cite[Theorem 3.11]{DaskWent} (for $\dim B = 2$) and Sibley \cite[Theorem 1.3]{SibleyYangMills} (for any dimension), says that, unlike critical Hermitian structures, $L^1$ $\delta$-approximate critical Hermitian structures exist on arbitrary holomorphic vector bundles over compact manifolds for any $\delta > 0$.
	It is then an easy verification that, as $\delta \to 0$, these metrics saturate the sharp lower bounds (as in (\ref{eq_low_bnd_he})) on the Hermitian Yang-Mills functional.
	Later, for brevity, we omit $L^1$ from the above notation.
	\end{sloppypar}
	\par 
	We establish in Theorem \ref{thm_ray_apprx} that the construction of geodesic rays associated with the Harder-Narasimhan filtration and $\delta$-approximate critical Hermitian structures are in certain sense compatible. 
	Motivated by this, our construction of $\omega_{\epsilon}$ from Theorem \ref{thm_main2} is given by the dequantization of these geodesic rays.
	However, as the Harder-Narasimhan filtrations are given by subsheaves and not by subbundles (unless $\dim B = 1$), the resulting sequence of metrics would not be smooth in general. 
	To overcome this issue, we need to resolve the singularities of the filtration first.
	\par 
	Using the resolution of indeterminacy of meromorphic maps, see Hironaka \cite{HironakaI}, \cite{HironakaII}, it is classical, cf. \cite[Proposition 4.3]{SibleyYangMills}, that for any filtration of a holomorphic vector bundle $E$ over $B$ by saturated subsheaves $E = \mathcal{F}_{\lambda_1} \supset \mathcal{F}_{\lambda_2} \supset \cdots \supset \mathcal{F}_{\lambda_q}$, there is a modification $\mu_0 : B_0 \to B$ of $B$ such that $\tilde{\mu}_0^* \mathcal{F}_{\lambda_i} := {\rm{Sat}}(\mu_0^* \mathcal{F}_{\lambda_i})$, $i = 1, \ldots, q$, are locally free, and form a filtration
	\begin{equation}\label{eq_resol_filtr}
		\mu_0^* E = \tilde{\mu}_0^* \mathcal{F}_{\lambda_1} \supset \tilde{\mu}_0^* \mathcal{F}_{\lambda_2} \supset \cdots \supset \tilde{\mu}_0^* \mathcal{F}_{\lambda_q},
	\end{equation}
	which we denote by $\tilde{\mu}_0^* \mathcal{F}$.
	\par 
	We denote by $\mu_k : B_k \to B$ a modification of $B$, corresponding to the resolution of the Harder-Narasimhan filtration, (\ref{eq_resol_hn_filt00}), (given by the saturated subsheaves, see \cite[Lemma 5.7.5]{KobaVB})
	\begin{equation}\label{eq_resol_hn_filt}
		\mu_k^* E_k = \tilde{\mu}_k^* \mathcal{F}^{HN, k}_{\lambda_1} \supset \tilde{\mu}_k^* \mathcal{F}^{HN, k}_{\lambda_2} \supset \cdots \supset \tilde{\mu}_k^* \mathcal{F}^{HN, k}_{\lambda_{q_k}},
	\end{equation}
	constructed as in (\ref{eq_resol_filtr}), i.e. $\tilde{\mu}_k^* \mathcal{F}^{HN, k}_{\lambda_i} = {\rm{Sat}}(\mu_k^* \mathcal{F}^{HN, k}_{\lambda_i})$. 
	Denote by $X_k$ the pull-back of $\mu_k$ and $\pi$, and by $\pi_k : X_k \to B_k$, $p_k : X_k \to X$ be the corresponding projection maps, i.e. such that the following diagram is commutative
	\begin{equation}\label{eq_sh_exct_seq}
	\begin{tikzcd}
		X_k \arrow[twoheadrightarrow]{r}{p_k} \arrow{d}{\pi_k} & X \arrow{d}{\pi} \\
 		B_k \arrow[twoheadrightarrow]{r}{\mu_k} & B.
	\end{tikzcd}
	\end{equation}
	For further purposes, we introduce the following subsets
	\begin{equation}\label{eq_sk_defn}
		S_k := \cup_{i = 1}^{q_k} {\rm{supp}}(\tilde{\mu}_k^* \mathcal{F}^{HN, k}_{\lambda_i} / \mu_k^* \mathcal{F}^{HN, k}_{\lambda_i} ).
	\end{equation}
	Clearly, $S_k$ are proper analytic subsets of $B_k$ and $\mu_k (S_k) \subset S_k^0$, where $S_k^0$ were defined in (\ref{eq_sk0_defn}).
	\par 
	Now, remark that for a smooth relatively Kähler $(1, 1)$-form $\omega$ in the class $c_1(L)$, the value of the Harder-Narasimhan potential ${\rm{HN}}_k(\omega)$, from (\ref{eq_hn_potent}), along a fixed fiber depends only on the restriction of $\omega$ to this fiber.
	In particular, if $\omega$ is now a smooth relatively Kähler form in the class $p_k^* c_1(L)$, one can still make sense of its Harder-Narasimhan potential as a function defined away from $\pi_k^{-1}(\mu_k^{-1}(\cup_{k = 0}^{+\infty} S_k^0))$.
	We denote the resulting function by $p_k^* {\rm{HN}}_k(\omega)$, and the uniform bound from (\ref{eq_geod_ray_bnd}) says that it is an element of $L^{\infty}(X_k)$.
	The remainder of the article, up to Section \ref{sect_mr}, is devoted to proving the following result.
	\begin{thm}\label{thm_expl_const1}
		For any $\epsilon > 0$, there are $\delta > 0$, $k \in \nat$, such that for any $\delta$-approximate critical Hermitian structure $H_{\delta, k}$ on $E_k$, there is $s \in [0, + \infty[$, such that for the geodesic ray of Hermitian metrics $H_{\delta, k, s}$ on $\mu_k^* E_k$, departing from $\mu_k^* H_{\delta, k}$ and associated with the resolution of the Harder-Narasimhan filtration (\ref{eq_resol_hn_filt}), the $(1, 1)$-form $\omega_{\delta, k, s} := c_1(p_k^* L, FS(H_{\delta, k, s})^{\frac{1}{k}})$ verifies
		\begin{equation}\label{eq_expl_const1}
			\int_{X_k}
			\Big|
			\omega_{\delta, k, s}^{n + 1} \wedge \pi_k^* \mu_k^* \omega_B^{m - 1}
			-
			p_k^* {\rm{HN}}(\omega_{\delta, k, s}) \cdot \omega_{\delta, k, s}^n \wedge  \pi_k^* \mu_k^* \omega_B^m \cdot (n + 1)
			\Big|
			\leq
			\epsilon.
		\end{equation}
	\end{thm}
	\par 
	Let us explain why Theorem \ref{thm_expl_const1} implies Theorem \ref{thm_main2}.
	For this, inspired by \cite[Proposition 2.1]{DemAndrGrau}, we show that the value of the Wess-Zumino-Witten functional is not affected by the birational modifications of the base. 
	More precisely, let $\mu_0 : B_0 \to B$ be a modification. 
	We define $X_0$ through the pull-back of $\mu$ and $\pi$, and let $\pi_0 : X_0 \to B_0$, $p_0 : X_0 \to X$ be the corresponding projection maps, i.e. for $k := 0$, the diagram (\ref{eq_sh_exct_seq}) is commutative.
	\begin{prop}\label{prop_bir_mod}
		For any $\epsilon > 0$, and a relatively Kähler $(1, 1)$-form $\omega_0$ on $X_0$ in the class $p_0^* c_1(L)$, there is a compact subset $K \subset B \setminus {\rm{Crit}}(\mu_0)$, such that for any compact subset $K' \subset B \setminus {\rm{Crit}}(\mu_0)$, verifying $K \subset K'$, there is a relatively Kähler $(1, 1)$-form $\omega$ on $X$ in the class $c_1(L)$, so that over $\pi_0^{-1}(\mu_0^{-1}(K'))$, the forms $\omega_0$ and $p_0^* \omega$ coincide, and 
		\begin{equation}\label{eq_prop_bir_mod}
			\int_{X_0 \setminus \pi_0^{-1}(\mu_0^{-1}(K'))} \big| \omega_0^{n + 1} \wedge \pi_0^* \mu_0^* \omega_B^{m - 1} \big| < \epsilon, 
			\qquad 
			\int_{X \setminus \pi^{-1}(K')} \big| \omega^{n + 1} \wedge \pi^* \omega_B^{m - 1} \big| < \epsilon.
		\end{equation}
	\end{prop}
	\begin{proof}
		We first note that the statement is nontrivial only in the case where $\dim B = m \geq 2$, which we assume from now on.
		We fix $\epsilon > 0$ and choose a compact subset $K \subset B \setminus {\rm{Crit}}(\mu_0)$ in such a way that the first bound from (\ref{eq_prop_bir_mod}) is satisfied for $K' := K$ with $\epsilon := \frac{\epsilon}{2}$.
		This is clearly always possible since $\pi_0^{-1}(\mu_0^{-1} ({\rm{Crit}}(\mu_0)))$ has Lebesgue measure zero.
		\par 
		We take an arbitrary compact subset $K' \subset B \setminus {\rm{Crit}}(\mu_0)$, verifying $K \subset K'$.
		We fix an arbitrary relatively Kähler $(1, 1)$-form $\omega'$ on $X$ in the class $c_1(L)$, and write $\omega_0 = p_0^* \omega' + \imun \partial \dbar \phi$, for a smooth function $\phi : X_0 \to \real$.
		Since ${\rm{Crit}}(\mu_0)$ is a complex-analytic subset of $B$ of codimension at least $2$, it has Hausdorff codimension at least $4$.
		From \cite[Theorem 5.1.9]{AdansHedberg}, it implies that the Sobolev $(2, \frac{2m}{m + 2})$-capacity of ${\rm{Crit}}(\mu_0)$ is zero, cf. \cite[Definition 2.2.1]{AdansHedberg} for the definition of the Sobolev capacity.
		By \cite[Definition 2.7.1, Corollary 3.3.4 and p. 208]{AdansHedberg}, this means that there is a sequence of smooth functions $\rho_k : B \to \real$, $k \in \nat$, taking value $1$ in a neighborhood of ${\rm{Crit}}(\mu_0)$ (with neighborhoods varying for different $k$) and with support away from $K'$, such that their $(2, \frac{2m}{m + 2})$-Sobolev norms tend to zero. 
		As a consequence, we have $\int_B |\Delta_{\omega_B} \rho_k| \omega_B^m \to 0$, as $k \to \infty$, and by Sobolev embedding theorem, we also have $\int_B |d \rho_k|^2 \omega_B^m \to 0$, as $k \to \infty$.
		\par 
		We now consider the form $\omega_k := \omega' + \imun \partial \dbar ((1 - \pi^* \rho_k) \cdot \phi)$.
		Remark that since $\rho_k$ takes value $1$ in a neighborhood of ${\rm{Crit}}(\mu_0)$, $\omega_k$ coincides with $\omega'$ in a neighborhood of ${\rm{Crit}}(\mu_0)$, and hence makes sense as a $(1, 1)$-form on $X$.
		Moreover, as $\rho_k$ has support away from $K'$, over $\pi_0^{-1}(\mu_0^{-1}(K'))$, the forms $\omega_0$ and $p_0^* \omega_k$ coincide.
		Directly from the above bounds on the derivatives of $\rho_k$, we deduce
		\begin{equation}
			\lim_{k \to \infty} \int_{X_0} \Big| \omega_0^{n + 1} \wedge \pi_0^* \mu_0^* \omega_B^{m - 1} - p^* \big( \omega_k^{n + 1} \wedge \pi^* \omega_B^{m - 1} \big) \Big| = 0, 
		\end{equation}
		which easily implies that for $k \in \nat$ big enough, $\omega := \omega_k$ will satisfy the second bound (\ref{eq_prop_bir_mod}).
	\end{proof}

	\begin{proof}[Proof of Theorem \ref{thm_main2}]
		We fix $\epsilon > 0$ and consider $\omega_{\delta, k, s}$ given by Theorem \ref{thm_expl_const1}.
		We fix a compact subset $K \subset B \setminus {\rm{Crit}}(\mu_k)$, given by Proposition \ref{prop_bir_mod}.
		We take a compact subset $K' \subset B \setminus {\rm{Crit}}(\mu_0)$, $K \subset K'$, which we specify later.
		Using Proposition \ref{prop_bir_mod}, we find a relatively Kähler $(1, 1)$-form $\omega_{\epsilon}$ on $X$ in the class $c_1(L)$, so that over $\pi_k^{-1} (\mu_k^{-1}(K'))$, the forms $\omega_{\delta, k, s}$ and $p_k^* \omega_{\epsilon}$ coincide, and
		\begin{equation}\label{eq_eps_away_from_soln}
			\int_{X_k \setminus \pi_k^{-1} (\mu_k^{-1}(K'))} \big| \omega_{\delta, k, s}^{n + 1} \wedge \pi_k^* \mu_k^* \omega_B^{m - 1} \big| < \epsilon, 
			\qquad 
			\int_{X \setminus \pi^{-1} (K')} \big| \omega_{\epsilon}^{n + 1} \wedge \pi^* \omega_B^{m - 1} \big| < \epsilon.
		\end{equation}
		Remark that due to coincidence of $\omega_{\delta, k, s}$ with $p_k^* \omega_{\epsilon}$ over $\pi_k^{-1} (\mu_k^{-1}(K'))$, we have $p_k^* {\rm{HN}}(\omega_{\delta, k, s}) = p_k^*( {\rm{HN}}(\omega_{\epsilon}))$ over $\pi_k^{-1} (\mu_k^{-1}(K'))$.
		From this and (\ref{eq_expl_const1}), the following bound is then immediate 
		\begin{equation}\label{eq_eps_away_from_soln2}
		\begin{aligned}
			\int_{X}
			\Big|
			\omega_{\epsilon}^{n + 1} \wedge & \pi^* \omega_B^{m - 1}
			-
			{\rm{HN}}(\omega_{\epsilon}) \cdot \omega_{\epsilon}^n \wedge  \pi^* \omega_B^m \cdot (n + 1)
			\Big|
			\\
			&
			\leq
			\int_{X_k \setminus \pi_k^{-1} (\mu_k^{-1}(K'))} \big| \omega_{\delta, k, s}^{n + 1} \wedge \pi_k^* \mu_k^* \omega_B^{m - 1} \big|
			+
			\int_{X \setminus \pi^{-1} (K')} \big| \omega_{\epsilon}^{n + 1} \wedge \pi^* \omega_B^{m - 1} \big|
			\\
			&
			{\phantom{\leq}}
			+
			(n + 1)
			\cdot
			\int_{X_k \setminus \pi_k^{-1} (\mu_k^{-1}(K'))}
			\Big|
			p_k^* {\rm{HN}}(\omega_{\delta, k, s}) \cdot \omega_{\delta, k, s}^n \wedge  \pi_k^* \mu_k^* \omega_B^m
			\Big|
			\\
			&
			{\phantom{\leq}}
			+
			(n + 1)
			\cdot
			\int_{X \setminus \pi^{-1} (K')}
			\Big|
			{\rm{HN}}(\omega_{\epsilon}) \cdot \omega_{\epsilon}^n \wedge  \pi^* \omega_B^m
			\Big| + \epsilon.
		\end{aligned}
		\end{equation}
		Remark, however, that by (\ref{eq_geod_ray_bnd}), there is $C > 0$, such that $|{\rm{HN}}(\omega_{\delta, k, s})| < C$, $|{\rm{HN}}(\omega_{\epsilon})| < C$.
		From this, (\ref{eq_eps_away_from_soln}) and (\ref{eq_eps_away_from_soln2}), we conclude that 
		\begin{multline}\label{eq_eps_away_from_soln3}
			\int_{X}
			\Big|
			\omega_{\epsilon}^{n + 1} \wedge \pi^* \omega_B^{m - 1}
			-
			{\rm{HN}}(\omega_{\epsilon}) \cdot \omega_{\epsilon}^n \wedge  \pi^* \omega_B^m \cdot (n + 1)
			\Big|
			\\
			\leq
			2 C \cdot (n + 1) \cdot
			\pi_* (c_1(L)^n)
			\cdot
			\int_{B \setminus K'} \omega_B^m 
			+
			3 \epsilon.
		\end{multline}
		The right-hand side of (\ref{eq_eps_away_from_soln3}) can be made smaller than $4 \epsilon$ by taking bigger $K'$.
		In particular, the forms $\omega_{\epsilon}$ satisfy Theorem \ref{thm_main2}, for $\epsilon := 4 \epsilon$.
		But as $\epsilon > 0$ was chosen in an arbitrary way, this finishes the proof.
	\end{proof}
	\begin{sloppypar}
	\begin{proof}[Proof of Theorem \ref{thm_main}]
		By Propositions \ref{prop_low_bnd}, it suffices to establish the upper bound on ${\rm{WZW}}(c_1(L) - t \pi^* [\omega_B], \omega_B)$. 
		As explained in Introduction, it follows from Theorem \ref{thm_main2}.
	\end{proof}
	\end{sloppypar}
	
	\section{Dequantization of approximate critical Hermitian structures}\label{sect_dequant}
	The main goal of this section is to prove Theorem \ref{thm_expl_const1} modulo a number of technical results which will be treated later in this article.
	We will conserve the notations from Introduction and Section \ref{sect_min_seq}.
	\par 
	Our proof of Theorem \ref{thm_expl_const1} is based on the well-known fact that section rings of polarized projective manifolds are finitely generated. 
	More precisely, the following family version of this result is used: there is $k_0 \in \nat^*$, such that for any $k \in \nat^*$, $k_0 | k$, $l \in \nat^*$, the multiplication map 
	\begin{equation}\label{eq_fin_gen}
		{\rm{Mult}}_{k, l} : {\rm{Sym}}^l E_k \to E_{kl},
	\end{equation}
	is surjective, see \cite[Example 2.1.30]{LazarBookI} or \cite[Proposition 3.1]{FinSecRing}, for a proof of a non-family version of this result, which easily adapts to the family setting considered here.
	To simplify further presentation, we shall assume that $k_0 = 1$.
	\par 
	The surjectivity of (\ref{eq_fin_gen}) allows us to apply the constructions of the quotient norms from (\ref{eq_defn_quot_norm}) for the map (\ref{eq_fin_gen}).
	For an arbitrary Hermitian metric $H_k$ on $E_k$ (resp. $\mu_k^* E_k$), it would yield the induced quotient metric on $E_{kl}$ (resp. $\mu_k^* E_{kl}$), which we denote by $[{\rm{Sym}}^l H_k]$. 
	Similar notations are used for the induced filtrations.
	Our proof of Theorem \ref{thm_expl_const1} is based on a detailed study of the metric $[{\rm{Sym}}^l H_{\delta, k, s}]$ on $E_{kl}$, as $l \to \infty$, where $H_{\delta, k, s}$ was defined in Theorem \ref{thm_expl_const1}.
	\par 
	To explain our proof, we need to recall the definition of a \textit{Toeplitz operator}.
	We recall first a non-family version of it, and for this, we fix a complex projective manifold $Y$ polarized by an ample line bundle $F$ endowed with a positive Hermitian metric $h^F$ on $F$.
	Recall that for any $f \in L^{\infty}(Y)$, the Toeplitz operator, $T_k^{h^F}(f) \in {\rm{End}}(H^0(Y, F^{\otimes k}))$, is defined as follows 
	\begin{equation}\label{eq_defn_toepl}
		T_k^{h^F}(f) := B_k \circ M_{f, k},
	\end{equation}
	where $B_k : L^{\infty}(Y, F^{\otimes k}) \to H^0(Y, F^{\otimes k})$ is the orthogonal (with respect to the associated $L^2$-norm ${\rm{Hilb}}_{k}(h^F)$) projection to $H^0(Y, F^{\otimes k})$, and $M_{f, k} : H^0(Y, F^{\otimes k}) \to L^{\infty}(Y, F^{\otimes k})$ is the multiplication map by $f$, acting as $s \mapsto f \cdot s$.
	We call $f \in L^{\infty}(Y)$ the \textit{symbol of the Toeplitz operator}.
	\par 
	For a smooth vector bundle $G$ on $B$, and $f \in L^{\infty}(X, \pi^* G)$, we define the family version of Toeplitz operator, $T_k^{\pi, h^L}(f) \in L^{\infty}(B, \enmr{E_k} \otimes G)$, by gluing (\ref{eq_defn_toepl}) fiberwise.
	\par 
	The proof of Theorem \ref{thm_expl_const1} then decomposes into several parts, which we roughly summarize as follows.
	First, we show that for any Hermitian metric $H_k$ on $E_k$, the metric $[{\rm{Sym}}^l H_k]$ is very close to the one provided by the quantization, ${\rm{Hilb}}_{kl}^{\pi}(FS(H_k)^{1/k})$.
	Most importantly, we show that the curvatures of these two Hermitian metrics are close enough.
	Second, we roughly show that for carefully chosen parameters $l, \delta, k, s$, the Hermitian metric $[{\rm{Sym}}^l H_{\delta, k, s}]$ on $\mu_k^* E_{kl}$ is very close to being an approximate critical Hermitian structure in the sense that its curvature is very close to the weight operator associated with Harder-Narasimhan filtration.
	Third, we show that the weight operator associated with the Harder-Narasimhan filtration is very close to the Toeplitz operator with the symbol given by the Harder-Narasimhan potential of the metric, defined in (\ref{eq_hn_potent}).
	\par 
	\begin{sloppypar}
	A combination of the second and the third arguments above show that the curvature of $[{\rm{Sym}}^l H_{\delta, k, s}]$ is close to the Toeplitz operator with the symbol given by the Harder-Narasimhan potential.
	On another hand, a result of Ma-Zhang \cite{MaZhangSuperconnBKPubl} says that for a relatively positive Hermitian metric $h^L$ on $L$, the curvature of $L^2$-metric ${\rm{Hilb}}_{k}^{\pi}(h^L)$ is also a Toeplitz operator, and the symbol is given by the horizontal component of $c_1(L, h^L)$. 
	However, since the curvatures of $[{\rm{Sym}}^l H_{\delta, k, s}]$ and ${\rm{Hilb}}_{kl}^{\pi}(FS(H_{\delta, k, s})^{1/k})$ are very close by the first argument, the two expressions above should essentially coincide. 
	As both expressions are given by Toeplitz operators, this means that the respective symbols should be close. 
	But this means precisely that the Harder-Narasimhan potential of $\omega_{\delta, k, s}$ essentially coincides with the horizontal curvature of $\omega_{\delta, k, s}$. 
	A simple manipulation shows that this is just a reformulation of Theorem \ref{thm_expl_const1}, which finishes its proof.
	\end{sloppypar}
	\par 
	We now provide a detailed description, starting with a more precise discussion of the key components involved in the proof.
	First, in Section \ref{sect_curv}, relying on the semiclassical Ohsawa-Takegoshi extension theorem from \cite{FinOTAs} and on subsequent works \cite{FinSecRing}, \cite{FinToeplImm}, we establish the following result.
	\par 
	\begin{thm}\label{thm_appl_ot}
		For any $k \in \nat$ and a Hermitian metric $H_k$ on $E_k$, there are $l_0 \in \nat$, $C > 0$, such that for any $l \geq l_0$, we have
		\begin{equation}\label{eq_appl_ot}
		\Big\|  
			\frac{\imun}{2 \pi} R^{[{\rm{Sym}}^l H_k]} - \frac{\imun}{2 \pi} R^{{\rm{Hilb}}_{kl}^{\pi}(FS(H_k)^{1/k})}
		\Big\| \leq C \sqrt{l},
		\end{equation}
		where the norm $\| \cdot \|$ is for a norm induced by a fixed metric on $TB$, and the subordinate norm on $\enmr{E_k}$ associated with ${\rm{Hilb}}_{kl}^{\pi}(FS(H_k)^{1/k})$.
	\end{thm}
	\par 
	We will also need to compare the metrics themselves. 
	In this direction, the following easy consequence of \cite{FinToeplImm}, \cite{FinOTRed}, will be explained in details in Section \ref{sect_curv}.
	\begin{prop}\label{thm_appl_ot2}
		For any $k \in \nat$ and a Hermitian metric $H_k$ on $E_k$, there are $l_0 \in \nat$, $C > 0$, such that for any $l \geq l_0$, we have
		\begin{equation}
			1 - \frac{C}{\sqrt{l}}
			\leq
			\frac{[{\rm{Sym}}^l H_k]}{{\rm{Hilb}}_{kl}^{\pi}(FS(H_k)^{1/k})} \cdot \frac{1}{k^m l^n}
			\leq
			1 + \frac{C}{\sqrt{l}}.
		\end{equation}
	\end{prop}
	\par 
	We now fix a relatively positive Hermitian metric $h^L$ on $L$, and denote $\omega := c_1(L, h^L)$.
	We denote by $\omega_H \in \ccal^{\infty}(X, \wedge^{1, 1} \pi^* T^* B)$ the horizontal part of the curvature, $\omega$, defined as after Proposition \ref{prop_low_bnd}.
	The following result, which was already used in our proof of (\ref{eq_wzw_hym}) from \cite{FinHNII}, will continue to play a crucial role in our work.
	\begin{thm}[{ Ma-Zhang \cite[Theorem 0.4]{MaZhangSuperconnBKPubl} }]\label{thm_mazh}
		There are $C > 0$, $l_0 \in \nat$, such that for any $l \geq l_0$,
		\begin{equation}
			\Big\|
				\frac{\imun}{2 \pi}
				R^{{\rm{Hilb}}_l^{\pi}(h^L)}
				-
				l \cdot T_l^{\pi, h^L}(\omega_H)
			\Big\|
			\leq
			C,
		\end{equation}
		where the norm $\| \cdot \|$ here is as in (\ref{eq_appl_ot}), but with ${\rm{Hilb}}_{l}^{\pi}(h^F)$ instead of ${\rm{Hilb}}_{kl}^{\pi}(FS(H_k)^{1/k})$.
	\end{thm}
	\par 
	The following result is technically the most difficult part of this paper. 
	Roughly, it shows that approximate critical Hermitian structures can be constructed inductively.
	It is the only statement of the paper which depends truly on the fact that we are dealing with the Harder-Narasimhan filtrations on $\oplus_{k = 0}^{+ \infty} E_k$, and not with arbitrary bounded submultiplicative filtrations.
	\par 
	To state it, for any $\delta > 0$, $k \in \nat$, we denote by $H_{\delta, k}$ a $\delta$-approximate critical Hermitian structure on $E_k$.
	We use the notation $\tilde{\mu}_k^* \mathcal{F}^{HN, k}$ for the resolution of the Harder-Narasimhan filtration as in (\ref{eq_resol_hn_filt}).
	For any $s \in [0, + \infty[$, we denote by $H_{\delta, k, s}$ the geodesic ray of Hermitian metrics on $\mu_k^* E_k$ departing from $\mu_k^* H_{\delta, k}$ and associated with $\tilde{\mu}_k^* \mathcal{F}^{HN, k}$.
	We also use the notation for the weight operator from (\ref{eq_weight_op_dd}).
	\begin{thm}\label{thm_compat_approx_sol}
		For any $\epsilon > 0$, there is $k_0 \in \nat$, such that for any $k \geq k_0$, there are $l_0 \in \nat$, $\delta > 0$, $s > 0$, such that for any $l \geq l_0$, we have
		\begin{multline}\label{eq_thm_compat_approx_sol}
			\Big\|
				\frac{\imun}{2 \pi} R^{[{\rm{Sym}}^l H_{\delta, k, s}]}
				\wedge
				\mu_k^* \omega_B^{m - 1}
				-
				A([{\rm{Sym}}^l H_{\delta, k, s}], [{\rm{Sym}}^l \tilde{\mu}_k^* \mathcal{F}^{HN, k}])
				\cdot
				\mu_k^* \omega_B^m
			\Big\|_{L^1(B_k, [{\rm{Sym}}^l H_{\delta, k, s}])}^{{\rm{tr}}}
			\\
			\leq \epsilon l \cdot N_{kl}.
		\end{multline}
	\end{thm}
	\begin{rem}
		This result highlights the importance of considering geodesic rays in our construction from Theorem \ref{thm_expl_const1}: (\ref{eq_thm_compat_approx_sol}) can be made small only by taking $s$ big enough.	 
	\end{rem}
	\par 
	The most crucial part of the proof of Theorem \ref{thm_compat_approx_sol} is based on the precise study of the weight operators along geodesic rays.
	The full proof occupies Sections \ref{sect_induc_const}-\ref{sect_geod_appr}.
	\par 
	The next result shows that the number operator behaves well with respect to the change of the metric. 
	The proof will be based on the analysis from \cite{FinSubmToepl}, and is postponed until Section \ref{sect_weight_oper}.
	\begin{thm}\label{thm_compat_approx_sol1}
		For any $k \in \nat$, a Hermitian metric $H_k$ on $E_k$ and a filtration $\mathcal{F}^k$ on $E_k$, there are $l_0 \in \nat$, $C > 0$, such that for any $l \geq l_0$, we have
		\begin{equation}
			\Big\|
				A([{\rm{Sym}}^l H_k], [{\rm{Sym}}^l \mathcal{F}^k])
				-
				A({\rm{Hilb}}_{kl}^{\pi}(FS(H_k)^{1/k}), [{\rm{Sym}}^l \mathcal{F}^k])
			\Big\|
			\leq
			C \sqrt{l} \log(l),
		\end{equation}
		where the norm $\| \cdot \|$ here is as in (\ref{eq_appl_ot}).
	\end{thm}
	\par 
	The final essential step is to compare the weight operator to the Toeplitz operator.
	The following result, based on our recent work \cite{FinSubmToepl}, will be crucial for this.
	\begin{thm}\label{thm_compat_approx_sol2}
		For any bounded filtration $\mathcal{F} := \oplus_{k = 0}^{+ \infty} \mathcal{F}^{k}$ on $\oplus_{k = 0}^{+ \infty} E_k$, which is submultiplicative away from a Lebesgue negligible subset, a relatively positive Hermitian metric $h^L$ on $L$ and $\epsilon > 0$, there is $k_1 \in \nat$, such that for any $k \geq k_1$, there is $l_0 \in \nat$, such that for any $l \geq l_0$, we have
		\begin{equation}
			\Big\|
				A({\rm{Hilb}}_{kl}^{\pi}(h^L), [{\rm{Sym}}^l \mathcal{F}^{k}])
				\cdot \omega_B^m
				-
				kl
				\cdot
				T_{kl}^{\pi, h^L}(\phi^{\pi}(h^L, \mathcal{F}))
				\cdot \omega_B^m
			\Big\|_{L^1(B, {\rm{Hilb}}_{kl}^{\pi}(h^L))}^{{\rm{tr}}}
			\leq
			\epsilon \cdot kl \cdot N_{kl},
		\end{equation}
		where $\phi^{\pi}(h^L, \mathcal{F})$ is the fiberwise geodesic ray, defined in Proposition \ref{prop_conv_berg_family}.
	\end{thm}
	\par 
	
	We will also need the following basic statement.
	\begin{prop}\label{prop_l1_toepl_conv}
		For any relatively positive Hermitian metric $h^L$ on $L$, $\omega := c_1(L, h^L)$, $f \in L^{\infty}(X)$, we have 
		\begin{equation}
			\lim_{k \to \infty}
			\frac{\| 
				T^{\pi, h^L}_{k}(f)  \cdot \omega_B^m
			\|_{L^1(B, {\textrm{Hilb}}^{\pi}_k(h^L))}^{{\rm{tr}}}}{N_k \cdot \int_B [\omega_B]^m}
			=
			\frac{\int_X |f(x)| \omega^n \wedge \pi^* \omega_B^m}{\int_X c_1(L)^n \cdot \pi^* [\omega_B]^m}.
		\end{equation}
	\end{prop}
	\begin{proof}
		Let us describe first the analogous statement in the non-family setting.
		Recall that in \cite[Theorem 5.4]{FinSubmToepl}, we established that for non-smooth symbols, the asymptotic spectral theory of Boutet de Monvel-Guillemin \cite[Theorem 13.13]{BoutGuillSpecToepl} continues to hold.
		In particular, following the notations from (\ref{eq_defn_toepl}), for any $g \in L^{\infty}(Y)$, we have
		\begin{equation}\label{eq_spec_conv_toepl_nonnfam}
			\lim_{k \to \infty}
			\frac{ {\rm{Tr}}[| T^{h^F}_{k}(g)|]}{N_k}
			=
			\frac{\int_Y |g(x)| c_1(F, h^F)^n}{\int_Y c_1(F)^n}.
		\end{equation}
		Recall also, cf. \cite[Lemma 5.7]{FinSubmToepl}, that for the operator norm the following trivial bound holds
		\begin{equation}\label{eq_toepl_esssup}
			\| T^{h^F}_{k}(g) \| \leq {\rm{esssup}}_{x \in Y} |g(x)|.
		\end{equation}
		Directly from (\ref{eq_toepl_esssup}), we deduce that the sequence of functions $b \mapsto \frac{1}{N_k}  {\rm{Tr}}[| T^{\pi, h^L}_{k}(f)(b)|]$, $b \in B$, is uniformly bounded in $k \in \nat$.
		From (\ref{eq_spec_conv_toepl_nonnfam}), we also have the pointwise convergence. 
		Proposition \ref{prop_l1_toepl_conv} then follows by the Lebesgue dominated convergence theorem.
	\end{proof}

	\begin{proof}[Proof of Theorem \ref{thm_expl_const1}.]
		We fix an arbitrary $\epsilon > 0$.
		For simplicity of the notation, we let $h^L_{\delta, k, s} := FS(H_{\delta, k, s})^{1/k}$.
		Remark that a combination of Theorems \ref{thm_appl_ot}, \ref{thm_mazh}, \ref{thm_compat_approx_sol} and Proposition \ref{thm_appl_ot2}, yield that for the horizontal part, $\omega_{\delta, k, s, H}$, defined as in Theorem \ref{thm_mazh}, of the form $\omega_{\delta, k, s}$ from Theorem \ref{thm_expl_const1}, the following holds.
		There is $k_1 \in \nat$, such that for any $k \geq k_1$, there are $C_1 > 0$, $\delta > 0$, $s > 0$, $l_1 \in \nat$, so that for any $l \geq l_1$, we have
		\begin{multline}\label{eq_thm_compat_approx_sol_111}
			\Big\|
				kl T_l^{\pi, h^L_{\delta, k, s}}(\omega_{\delta, k, s, H})
				\wedge
				\mu_k^* \omega_B^{m - 1}
				-
				A([{\rm{Sym}}^l H_{\delta, k, s}], [{\rm{Sym}}^l \tilde{\mu}_k^* \mathcal{F}^{HN, k}])
				\cdot
				\mu_k^* \omega_B^m
			\Big\|_{L^1(B_k, {\textrm{Hilb}}^{\pi}_{kl}(h^L_{\delta, k, s}))}^{{\rm{tr}}}
			\\
			\leq
			\Big( C_1 \sqrt{l} + \epsilon l \Big) \cdot N_{kl}.
		\end{multline}
		\par 
		A combination of Theorems \ref{thm_HN_subm}, \ref{thm_compat_approx_sol1} and \ref{thm_compat_approx_sol2} yields that there is $k_2 \in \nat$, such that for any $k \geq k_2$, $\delta > 0$, $s > 0$, there are $l_2 \in \nat$, $C_2 > 0$, such that for any $l \geq l_2$, we have
		\begin{multline}\label{eq_thm_compat_approx_sol_222}
			\Big\|
				A([{\rm{Sym}}^l H_{\delta, k, s}], [{\rm{Sym}}^l \tilde{\mu}_k^* \mathcal{F}^{HN, k}])
				\cdot
				\mu_k^* \omega_B^m
				-
				kl
				\cdot
				T_{kl}^{\pi, h^L_{\delta, k, s}}(p_k^* {\rm{HN}}(\omega_{\delta, k, s}))
				\wedge
				\mu_k^* \omega_B^m
			\Big\|_{L^1(B_k, {\textrm{Hilb}}^{\pi}_{kl}(h^L_{\delta, k, s}))}^{{\rm{tr}}}
			\\
			\leq
			C_2 \sqrt{l} \log(l)  N_{kl} + \epsilon l N_{kl}.
		\end{multline}
		We now fix $k_3 := \max\{ k_1, k_2\}$. 
		Let $\delta_3 > 0$, $s_3 > 0$ be such that (\ref{eq_thm_compat_approx_sol_111}) holds.
		If we combine (\ref{eq_thm_compat_approx_sol_111}) with (\ref{eq_thm_compat_approx_sol_222}), we obtain that there is $l_3 \in \nat$, such that for any $l \geq l_3$, we have
		\begin{multline}\label{eq_thm_compat_approx_sol_333}
			\Big\|
				T_{k_3 l}^{\pi, h^L_{\delta_3, k_3, s_3}}(\omega_{\delta_3, k_3, s_3, H})	\wedge
				\mu_{k_3}^* \omega_B^{m - 1}
				- 
				T_{k_ 3l}^{\pi, h^L_{\delta_3, k_3, s_3}}(p_{k_3}^* {\rm{HN}}(\omega_{\delta_3, k_3, s_3})) \wedge
				\mu_{k_3}^* \omega_B^m
			\Big\|_{L^1(B_{k_3}, {\textrm{Hilb}}^{\pi}_{k_3 l}(h^L_{\delta_3, k_3, s_3}))}^{{\rm{tr}}}
			\\
			\leq
			4 \epsilon N_{k_3 l}.
		\end{multline}
		Directly from Proposition \ref{prop_l1_toepl_conv}, by dividing both sides of (\ref{eq_thm_compat_approx_sol_333}) by $N_{k_3 l}$ and taking $l \to \infty$, we get
		\begin{equation}\label{eq_thm_compat_approx_sol_444}
			\int_{X_{k_3}}
			\Big| \wedge_{\omega_B} \omega_{\delta_3, k_3, s_3, H} - p_{k_3}^* {\rm{HN}}(\omega_{\delta_3, k_3, s_3}) \Big| \cdot \omega_{\delta_3, k_3, s_3}^n \wedge \pi_{k_3}^* \mu_{k_3}^* \omega_B^m
			\leq
			\frac{4 \epsilon \cdot \int_X c_1(L)^n \cdot \pi^* [\omega_B]^m}{\int_B [\omega_B]^m}.
		\end{equation}
		Remark, however, that by (\ref{eq_omega_h_wed_defn}) and the definition of $\omega_{\delta_3, k_3, s_3, H}$, (\ref{eq_thm_compat_approx_sol_444}) yields
		\begin{multline}
			\int_{X_{k_3}}
			\Big|
			\omega_{\delta_3, k_3, s_3}^{n + 1} \wedge \pi_{k_3}^* \mu_{k_3}^* \omega_B^{m - 1}
			-
			p_k^* {\rm{HN}}(\omega_{\delta_3, k_3, s_3}) \cdot \omega_{\delta_3, k_3, s_3}^n \wedge  \pi_{k_3}^* \mu_{k_3}^* \omega_B^m \cdot (n + 1)
			\Big|
			\\
			\leq
			\frac{4 \epsilon \cdot \int_X c_1(L)^n \cdot \pi^* [\omega_B]^m}{(n + 1) \cdot \int_B [\omega_B]^m}.
		\end{multline}
		Since $\epsilon > 0$ is arbitrary, this establishes Theorem \ref{thm_expl_const1}.
	\end{proof}

	\section{Curvature of direct images and semiclassical extension theorem}\label{sect_curv}
	The main goal of this section is to prove Theorem \ref{thm_appl_ot}. 
	The proof primarily relies on the semiclassical version of the Ohsawa-Takegoshi extension theorem developed in \cite{FinOTAs}. 
	Recall that the classical version of the extension theorem \cite{OhsTak1} states that under certain positivity conditions on a vector bundle, any holomorphic section defined on a submanifold can be extended to the entire ambient manifold, with a specific control on the $L^2$-norm of the extension.
	In the semiclassical version, we consider not just a single vector bundle but a sequence of vector bundles given by high tensor powers of a fixed ample line bundle, and we examine not only the control of the $L^2$-norm but also an asymptotic formula for the optimal extension itself. 
	Previous results by Bost \cite{BostDwork} and Zhang \cite{ZhangPosLinBun} provided some results concerning subexponential bounds on the optimal constant of the norm comparison in this semiclassical setting, but the version required here is from \cite{FinOTAs} (see also \cite{FinOTRed} for an alternative proof), which not only determines the optimal constant but also the asymptotics of the optimal extension itself.
	\par 
	To put Theorem \ref{thm_appl_ot} in the framework of this theorem, consider a holomorphic submersion $\pi : X \to B$ between compact Kähler manifolds $X$ and $B$ of dimensions $n + m$ and $m$ respectively.
	We fix  a compact complex manifold $P$ of dimensions $n' + m$, $n' > n$, with a submersion $p : P \to B$ and a complex embedding $\iota : X \to P$, so that we have the following commutative diagram 
	\begin{equation}\label{eq_sh_exct_seq2}
	\begin{tikzcd}
		X \arrow[hookrightarrow]{r}{\iota} \arrow{rd}{\pi} & P \arrow{d}{p} \\
 		& B
	\end{tikzcd}
	\end{equation}
	Consider a relatively positive Hermitian line bundle $(F, h^F)$ over $P$.
	For any $l \in \nat$, we define $G_l := R^0 \pi_* (\iota^* F^{\otimes l})$ and $K_l := R^0 p_* F^{\otimes l}$. 
	It is a classical consequence of Serre's vanishing theorem that for $l \in \nat$ big enough, the restriction map 
	\begin{equation}\label{eq_res_oper}
		\res_l: K_l \to G_l
	\end{equation}
	is surjective.
	From now on, we only work with $l \in \nat$ verifying this, and such that both $G_l$ and $K_l$ are locally free.
	\par 
	We denote by ${\rm{Hilb}}_l^{\pi}(\iota^* h^F)$ (resp. ${\rm{Hilb}}_l^p(h^F)$) the $L^2$-product on $G_l$ (resp. $K_l$) induced by $h^F$ as in (\ref{eq_l2_prod}).
	Recall that in \cite{FinToeplImm}, \cite{FinOTRed}, we compared the Hermitian metrics $[{\rm{Hilb}}_l^p(h^F)]$ and ${\rm{Hilb}}_l^{\pi}(\iota^* h^F)$ on $G_l$.
	Let us recall this statement.
	\begin{thm}\label{thm_appl_ot2a}
		There are $C > 0$, $l_0 \in \nat$, such that for any $l \geq l_0$, we have
		\begin{equation}
			1 - \frac{C}{\sqrt{l}}
			\leq
			\frac{[{\rm{Hilb}}_l^p(h^F)]}{{\rm{Hilb}}_l^{\pi}(\iota^* h^F)} \cdot l^{n' - n}
			\leq
			1 + \frac{C}{\sqrt{l}}.
		\end{equation}
	\end{thm}
	\begin{rem}
		In \cite{FinToeplImm}, \cite{FinOTRed}, Theorem \ref{thm_appl_ot2a} appeared in the non-family version, i.e. for $B$ equal to a point.
		However, as the proof ultimately depends only on the off-diagonal asymptotic expansion of the Bergman kernel, which holds in families, it adapts to the version we need here.
		For details, consult the proof of Theorem \ref{thm_dk_as}, which establishes a stronger version of Theorem \ref{thm_appl_ot2a}.
	\end{rem}
	The main result of this section goes further and compares the curvatures of these metrics.
	\begin{thm}\label{cor_quot_nm}
		There are $l_0 \in \nat$, $C > 0$, such that we have
		\begin{equation}\label{eq_appl_ot2}
		\Big\|  
			\frac{\imun}{2 \pi} R^{[{\rm{Hilb}}_l^p(h^F)]} - \frac{\imun}{2 \pi} R^{{\rm{Hilb}}_l^{\pi}(\iota^* h^F)}
		\Big\| \leq C \sqrt{l},
		\end{equation}
		where the norm $\| \cdot \|$ is for a norm induced by a fixed metric on $TB$, and the subordinate norm on $\enmr{G_l}$ associated with ${\rm{Hilb}}_l^{\pi}(\iota^* h^F)$.
	\end{thm}
	Before providing the details of the proofs of these statements, let us show how they can be adapted to the setting required for Theorem \ref{thm_appl_ot}.
	\begin{proof}[Proof of Theorem \ref{thm_appl_ot}.]
		Proof of Theorem \ref{thm_appl_ot} will be based on the application of Theorem \ref{cor_quot_nm} to the relative Kodaira embedding defined in (\ref{eq_kod}).
		In other words, we take $P := \mathbb{P}(E_k^*)$, $\iota := {\rm{Kod}}_k$, $p : \mathbb{P}(E_k^*) \to B$ the usual projection, $F := \mathscr{O}(1)$ and $h^F$ the Fubini-Study metric induced by $H_k$. 
		Then, by the definition of the Fubini-Study metric, cf. (\ref{eq_kod}), the Hermitian vector bundle $(\iota^* F, \iota^* h^F)$ is isomorphic with $(L^k, FS(H_k))$.
		Also, $G_l$ is isomorphic with $E_{kl}$, and by the well-known calculation of the twisted structure sheaf cohomology on the projective space, $K_l$ is isomorphic with ${\rm{Sym}}^l E_k$.
		\par 
		A direct calculation, cf. \cite[Lemma 4.15]{FinTits}, shows the following relation
		\begin{equation}\label{eq_induc_1}
			{\rm{Hilb}}_l^p(h^F) = \frac{l!}{(l + N_k - 1)!} \cdot {\rm{Sym}}^l(H_k).
		\end{equation}
		Remark that in \cite[Lemma 4.15]{FinTits}, there was a square root in front of ${\rm{Sym}}^l(H_k)$ since we worked on the level of norms and not Hermitian products, as we do here.
		In particular, the identity (\ref{eq_induc_1}) shows that the norm ${\rm{Sym}}^l(H_k)$ in the statement can be replaced by ${\rm{Hilb}}_l^p(h^F)$.
		\par 
		We denote by $\res_{k, l}:K_l \to G_l$ the restriction operator associated with the Kodaira embedding.
		Recall that the multiplication map ${\rm{Mult}}_{k, l}$ was defined in (\ref{eq_fin_gen}).
		Then it is an easy verification, cf. \cite[(4.62)]{FinSecRing}, that the following diagram is commutative 
		\begin{equation}\label{eq_kod_res}
		\begin{tikzcd}
			{\rm{Sym}}^l(E_k) \arrow[equal]{r} \arrow{d}{{\rm{Mult}}_{k, l}} & K_l \arrow{d}{\res_{k, l}} \\
 			E_{kl} \arrow[equal]{r} & G_l.
		\end{tikzcd}
		\end{equation}
		Hence, the quotient of a norm with respect to ${\rm{Mult}}_{k, l}$ is identified with the quotient with respect to $\res_{k, l}$. 
		With these identifications, Theorem \ref{thm_appl_ot} then becomes a restatement of Theorem \ref{cor_quot_nm}.
	\end{proof}
	\begin{proof}[Proof of Proposition \ref{thm_appl_ot2}.]
		We use the same notations as in the proof of Theorem \ref{thm_appl_ot}.
		Directly from Theorem \ref{thm_appl_ot2a} and (\ref{eq_kod_res}), we obtain
		\begin{equation}\label{eq_appl_ot2111}
			1 - \frac{C}{\sqrt{l}}
			\leq
			\frac{[{\rm{Hilb}}_l^p(h^F)]}{{\rm{Hilb}}_{l}^{\pi}(FS(H_k))} \cdot l^{N_k - n - 1}
			\leq
			1 + \frac{C}{\sqrt{l}}.
		\end{equation}
		Remark now that ${\rm{Hilb}}_{l}^{\pi}(FS(H_k)) = k^m \cdot {\rm{Hilb}}_{kl}^{\pi}(FS(H_k)^{1/k})$, and that $l^{N_k - 1} \cdot l! \sim (l + N_k - 1)!$, as $l \to \infty$.
		The result now follows from this, (\ref{eq_induc_1}) and (\ref{eq_appl_ot2111}).
	\end{proof}
	\par 
	To prove Theorem \ref{cor_quot_nm}, we define the Hermitian section $D_l \in \ccal^{\infty}(B, \enmr{G_l})$, as follows
	\begin{equation}\label{eq_defn_dk}
		\scal{s_0}{s_1}_{[{\rm{Hilb}}_l^p(h^F)]}
		=
		\scal{D_l s_0}{s_1}_{{\rm{Hilb}}_l^{\pi}(\iota^* h^F)},
	\end{equation}
	for any $s_0, s_1 \in \ccal^{\infty}(B, G_l)$.
	The main technical result of this section goes as follows.
	\begin{thm}\label{thm_dk_as}
		For any $r \in \nat$, there are $C > 0$, $l_0 \in \nat$, such that for any $l \geq l_0$, we have
		\begin{equation}\label{eq_dk_as}
			\Big\| 
				D_l
				-
				\frac{1}{l^{n' - n}} {\rm{Id}}_{G_l}
			\Big\|_{\ccal^r(B)}
			\leq 
			\frac{C l^{r / 2}}{l^{n' - n + 1/2}},
		\end{equation}
		where the norm $\| \cdot \|_{\ccal^r(B)}$ is defined for $s \in \ccal^{\infty}(B, \enmr{G_l})$, as follows 
		\begin{equation}
			\| s \|_{\ccal^{l}(B)}
			:=
			\sup \sup_{b \in B} \| \nabla^{\enmr{G_l}}_{U_1} \cdots \nabla^{\enmr{G_l}}_{U_r}  s(b) \|,
		\end{equation}
		where the first supremum is taken over all possible vector fields $U_1, \ldots, U_r$ over $B$ of unit $\mathscr{C}^r$-norm (with respect to some fixed metric on $B$),  $\nabla^{\enmr{G_l}}_{\cdot}$ is the Chern connection associated with the norm ${\rm{Hilb}}_l^{\pi}(\iota^* h^F)$, and $\| \cdot \|$ is the operator norm subordinate to ${\rm{Hilb}}_l^{\pi}(\iota^* h^F)$.
	\end{thm}
	\begin{rem}
		In particular, we see that there are $C > 0$, $l_0 \in \nat$, such that for any $l \geq l_0$, we have $\| D_l \| \leq l^{n - n'}(1 + C / \sqrt{l})$, $\| D_l^{-1} \| \leq l^{n' - n}(1 + C / \sqrt{l})$, which refines Theorem \ref{thm_appl_ot2a} due to (\ref{eq_defn_dk}).
	\end{rem}
	In order to prove Theorem \ref{thm_dk_as}, we define the extension operator
	\begin{equation}\label{eq_ext_op}
		\ext_l :  G_l \to K_l,
	\end{equation}
	so that for $g \in G_{l, b}$, $b \in B$, we put $\ext_l (g) = f$, $f \in K_{l, b}$, where $\res_l(f) = g$, and $f$ has the minimal norm with respect to ${\rm{Hilb}}_l^p(h^F)$ among all $\tilde{f}$, $\tilde{f} \in K_{l, b}$, verifying $\res_l(\tilde{f}) = g$. 
	It is immediate that $\ext_l$ is a linear map.
	We shall establish in the proof of Theorem \ref{thm_dk_as} that $\ext_l$ is smooth, i.e. $\ext_l \in \ccal^{\infty}(B, {\rm{Hom}}(G_l, K_l))$.
	\par 
	We will also need a technical lemma which makes a connection between the operator norm and the Schwartz kernel.
	We fix $R_l \in \ccal^{\infty}(B, \enmr{G_l})$ and for any $x_1, x_2 \in X$, verifying $\pi(x_1) = \pi(x_2) = b$, we denote its fiberwise Schwartz kernel by $R_l(x_1, x_2) \in F_{x_1}^{\otimes l} \otimes (F_{x_2}^{\otimes l})^*$.
	By the definition, for any $s \in \ccal^{\infty}(B, G_l)$, we have 
	\begin{equation}
		(R_l s) (x_1) 
		=
		\int_{x_2 \in X_b}
		R_l(x_1, x_2) \cdot s(x_2) dv_{X_b}(x_2),
	\end{equation}
	where $dv_{X_b}$ is the volume form on the fiber $X_b$ induced by the Kähler form $\iota^* c_1(F, h^F)|_{X_b}$.
	We assume that there are $c, C > 0, r \in \nat$, such that for any $l \in \nat^*$, the following bound holds
	\begin{equation}\label{eq_young0}
		\big\| R_l(x_1, x_2) \big\|_{\ccal^{r}(X \times_B X)} \leq C l^{n + r/2} \exp(- c \sqrt{l} {\rm{dist}}(x_1, x_2)),
	\end{equation}
	where ${\rm{dist}}(x_1, x_2)$ is the distance induced by $\iota^* c_1(F, h^F)|_{X_b}$ (in $P_b$) between $\iota(x_1)$ and $\iota(x_2)$.
	\begin{lem}\label{lem_young}
		Under the above assumptions, there is $C > 0$, such that $\| R_l \|_{\ccal^r(B)} \leq C l^{r / 2}$, where $\| \cdot \|_{\ccal^r(B)}$ is defined as in Theorem \ref{thm_dk_as}.
	\end{lem}
	\begin{proof}
		Directly from (\ref{eq_young0}), there is $C > 0$, such that for any $x \in X$, $b := \pi(x)$, $l \in \nat^*$, we have
		\begin{equation}
		\begin{aligned}
			&
			\int_{z \in X_b}
			|\nabla^r R_l(x, z)|_{(h^F)^{\otimes l}} dv_{X_b}(z)
			\leq
			C l^{r / 2},
			\\
			&
			\int_{z \in X_b}
			|\nabla^r R_l(z, x)|_{(h^F)^{\otimes l}} dv_{X_b}(z)
			\leq
			C l^{r / 2}.
		\end{aligned}
		\end{equation}
		The result is then a direct consequence of Young's inequality for integral operators, cf. \cite[Proposition 2.9 and Corollary 2.10]{FinToeplImm}.
	\end{proof}
	\begin{proof}[Proof of Theorem \ref{thm_dk_as}]
		By the definition of the quotient metric and $\ext_l$, we have
		\begin{equation}\label{eq_ext_dk}
			\ext_l^* \circ \ext_l = D_l.
		\end{equation}
		Let us recall that the \textit{multiplicative defect}, $A_l$, is a section of  $\enmr{G_l}$, defined in \cite[Theorem 4.3]{FinToeplImm}, as the only operator verifying
		\begin{equation}\label{eq_def_ak}
			\res_l^* = \ext_l \circ A_l.
		\end{equation}
		As it is explained in \cite[Theorem 4.3]{FinToeplImm}, the existence of $A_l$ is a consequence of the surjectivity of (\ref{eq_res_oper}).
		Let us now establish that there are $l_0 \in \nat$, $C > 0$, so that for any $l \geq l_0$, $r \in \nat$, we have
		\begin{equation}\label{eq_ak_as}
			\big\| 
				A_l
				-
				l^{n' - n} {\rm{Id}}_{G_l}
			\big\|_{\ccal^r(B)}
			\leq 
			C l^{n' - n + (r - 1)/2}.
		\end{equation}
		Remark that (\ref{eq_ak_as}) appeared in \cite[Theorem 4.3]{FinToeplImm} in the non-family version (i.e. for $B$ equal to a point). As we explain below, the proof generalizes to our setting here.
		\par 
		Directly from (\ref{eq_def_ak}), we see that 
		\begin{equation}\label{eq_al_res}
			A_l = \res_l \circ \res_l^*.
		\end{equation}
		Now, for any $y_1, y_2 \in P$, verifying $p(y_1) = p(y_2) = b$, we denote by $B_l^{p, h^F}(y_1, y_2) \in F_{y_1}^{\otimes l} \otimes (F_{y_2}^{\otimes l})^*$ the fiberwise Bergman kernel, i.e. the section such that in the notations of (\ref{eq_bnd_hn1}), for any $s \in \ccal^{\infty}(B, p_* F^{\otimes l})$, the following identity is satisfied
		\begin{equation}
			(B_l^{p, h^F} s) (y_1) 
			=
			\int_{y_2 \in P_b}
			B_l^{p, h^F}(y_1, y_2) \cdot s(y_2) dv_{P_b}(y_2),
		\end{equation}
		where $dv_{P_b}$ is the volume form on the fiber $P_b$ induced by the Kähler form $c_1(F, h^F)|_{P_b}$.
		Similarly, for any $x_1, x_2 \in X$, verifying $\pi(x_1) = \pi(x_2) = b$, we define $B_l^{\pi, \iota^* h^F}(x_1, x_2) \in F_{x_1}^{\otimes l} \otimes (F_{x_2}^{\otimes l})^*$.
		From (\ref{eq_al_res}), we see that the Schwartz kernel, $A_l(x_1, x_2) \in F_{x_1}^{\otimes l} \otimes (F_{x_2}^{\otimes l})^*$, of $A_l$, is given by
		\begin{equation}\label{eq_bound_bk0}
			A_l(x_1, x_2) = B_l^{p, h^F}(\iota(x_1), \iota(x_2)).
		\end{equation}
		Now, as the Bergman kernel in smooth families is smooth, see \cite{DaiLiuMa} and \cite{MaZhangSuperconnBKPubl}, we conclude that the section $A_l(x_1, x_2)$ is also smooth.
		As a consequence, we get $A_l \in \ccal^{\infty}(B, {\enmr{G_l}})$.
		\par 
		As $B_l^{\pi, \iota^* h^F}$ acts as identity on $K_l$, the estimate (\ref{eq_ak_as}) follows directly from Lemma \ref{lem_young} and the following bound: there are $c, C > 0$, such that for any $l \in \nat$, we have 
		\begin{multline}\label{eq_bound_bk}
			\Big|
			B_l^{p, h^F}(\iota(x_1), \iota(x_2)) - l^{n' - n} \cdot B_l^{\pi, \iota^* h^F}(x_1, x_2)
			\Big|_{\ccal^r(X \times_B X)}
			\\
			\leq
			C l^{n' + (r - 1)/2} \exp(- c \sqrt{l} {\rm{dist}}(x_1, x_2)).
		\end{multline}
		The estimate (\ref{eq_bound_bk}) for $r = 0$ was deduced in \cite[proof of Theorem 4.3]{FinToeplImm} from the off-diagonal expansion of the Bergman kernel due to Dai-Liu-Ma \cite[Theorem 4.18]{DaiLiuMa} and the exponential decay of the Bergman kernel of Ma-Marinescu \cite[Theorem 1]{MaMarOffDiag}.
		Both of the latter bounds are established for $\mathscr{C}^r$-norm, cf. \cite[Theorem 1.6]{MaZhangSuperconnBKPubl}, and so the estimate (\ref{eq_bound_bk}) is valid for an arbitrary $r \in \nat$.
		\par 
		In particular, by (\ref{eq_ak_as}), we see that there is $l_1 \in \nat$, so that for $l \geq l_1$, the section $A_l$ has an inverse $A_l^{-1}$, defined using the infinite sum. 
		Using the fact that the space of operators with exponential decay as in (\ref{eq_bound_bk}), is an algebra, see \cite[\S 3]{FinOTAs}, cf. also \cite{MaMarToepl}, we deduce, following the lines of \cite[\S 5.4]{FinSecRing}, that the Schwartz kernel of $A_l^{-1}$ is smooth, which implies that $A_l^{-1} \in \ccal^{\infty}(B, \enmr{G_l})$.
		Then it is a direct consequence of (\ref{eq_ext_dk}) and (\ref{eq_def_ak}), cf. \cite[(5.6)]{FinToeplImm}, that the following relation holds
		\begin{equation}\label{eq_dl_al_rel}
			D_l = (A_l^*)^{-1}.
		\end{equation}		
		The statement (\ref{eq_dk_as}) now follows directly from (\ref{eq_ak_as}) and (\ref{eq_dl_al_rel}).
		\par 
		Remark also that from (\ref{eq_def_ak}), we have
		\begin{equation}\label{eq_def_ak202}
			\ext_l = \res_l^* \circ A_l^{-1}.
		\end{equation}
		From the smoothness of the Schwartz kernel of $A_l^{-1}$ and of the Bergman kernel, we deduce that the Schwartz kernel of $\ext_l$ is smooth, which implies that $\ext_l \in \ccal^{\infty}(B, {\rm{Hom}}(G_l, K_l))$.
	\end{proof}
	\begin{proof}[Proof of Theorem \ref{cor_quot_nm}]
		From the definition (\ref{eq_defn_dk}) of $D_l$, we conclude, cf. \cite[(1.9.1)]{SiuLectures}, that we have
		\begin{equation}\label{eq_siu_ident}
			R^{[{\rm{Hilb}}_l^p(h^F)]} - R^{{\rm{Hilb}}_l^{\pi}(\iota^* h^F)}
			=
			\dbar( (\nabla^{\enmr{G_l}, (1, 0)} D_l) D_l^{-1}).
		\end{equation}		 
		The result now follows directly from Theorem \ref{thm_dk_as} and (\ref{eq_siu_ident}).
	\end{proof}

	\section{Weight operators of submultiplicative filtrations}\label{sect_weight_oper}
	This section aims to study the weight operators associated with submultiplicative filtrations and to prove Theorems \ref{thm_compat_approx_sol1} and \ref{thm_compat_approx_sol2}. 
	The proofs of these theorems will build on the analysis recently developed by the author in \cite{FinSubmToepl}.
	\par 
	We fix a complex projective manifold $Y$ polarized by an ample line bundle $F$ endowed with a positive Hermitian metric $h^F$.
	Let $\mathcal{F}$ be a bounded submultiplicative filtration on $R(Y, F)$, and $\phi(h^F, \mathcal{F}) \in L^{\infty}(Y)$ be the speed of the associated geodesic ray, defined in (\ref{eq_geod_ray_bnd}).
	Recall that Toeplitz operators were defined in (\ref{eq_defn_toepl}) and weight operators in (\ref{eq_weight_op_dd}).
	The following result says that asymptotically weight operators of bounded submultiplicative filtrations are Toeplitz.
	\begin{sloppypar}
	\begin{thm}[{\cite[Theorem 1.6]{FinSubmToepl}}]\label{thm_main2prev}
		As $k \to \infty$, we have
		\begin{equation}
			\frac{1}{N_k}
			{\rm{Tr}} 
			\Big[
			\Big|
				\frac{1}{k} A({\rm{Hilb}}_k(h^F), \mathcal{F}^k)
				-
				T^{h^F}_{k}(\phi(h^F, \mathcal{F}))
			\Big|
			\Big]
			\to
			0.
		\end{equation}
	\end{thm}
	\end{sloppypar}
	\par 
	Later on, we will approximate arbitrary filtrations by the finitely-generated filtrations induced by the truncations of the original filtration.
	To understand this process, originally suggested in \cite{SzekeTestConf}, in detail, let $k \in \nat$ be such that the multiplication maps ${\rm{Sym}}^l H^0(Y, F^{\otimes k}) \to H^0(Y, F^{\otimes k l})$ are surjective for any $l \in \nat$. 
	We denote by $[{\rm{Sym}}^l \mathcal{F}^k]$ the filtration on $H^0(Y, F^{\otimes k l})$ induced by a filtration $\mathcal{F}^k$ on $H^0(Y, F^{\otimes k})$ and the above map.
	Clearly, the induced filtration $[{\rm{Sym}} \mathcal{F}^k] := \oplus_{l = 0}^{\infty} [{\rm{Sym}}^l \mathcal{F}^k]$ on $R(Y, F^{\otimes k})$ is submultiplicative.
	We denote by $\phi(h^{L^{\otimes k}}, [{\rm{Sym}} \mathcal{F}^k])$ the speed of the induced geodesic ray emanating from $h^{L^{\otimes k}}$, as constructed in (\ref{eq_geod_ray_bnd}).
	\begin{prop}[{\cite[(3.6)]{FinSubmToepl}}]\label{prop_incr_geod_rays}
		As $k \to \infty$, the sequence of functions $x \mapsto \frac{1}{k} \phi(h^{L^{\otimes k}}, [{\rm{Sym}} \mathcal{F}^k])$, $x \in Y$, increases almost everywhere towards $\phi(h^L, \mathcal{F})$, when we restrict over a multiplicative sequence in $k$ (as for example $k = 2^q$, $q \in \nat$).
	\end{prop}
	\begin{rem}
		In \cite{FinSubmToepl}, we deal with slightly different approximations: instead of considering filtrations induced by $\mathcal{F}^k$, we consider those induced by $\oplus_{l = 0}^{k} \mathcal{F}^l$.
		But our methods carry over in an identical manner.
		Indeed, as it is explained in \cite[(3.14)]{FinSubmToepl}, it suffices to establish the convergence of the respective geodesic rays.
		This was done in \cite[Theorem 5.10]{FinNarSim}.
	\end{rem}
	\par 
	To establish Theorem \ref{thm_compat_approx_sol2}, we need family versions of the above results.
	Consider a holomorphic submersion $\pi : X \to B$ between compact Kähler manifolds $X$ and $B$ of dimensions $n + m$ and $m$ respectively.
	For a relatively ample line bundle $L$ over $X$, we denote by $E_k := R^0 \pi_* L^{\otimes k}$.
	For $k \in \nat$, we consider a filtration $\mathcal{F} := \oplus_{k = 0}^{+\infty} \mathcal{F}^k$ of $\oplus_{k = 0}^{+\infty} E_k$ as in (\ref{eq_resol_hn_filt001}), which is bounded and submultiplicative away from a negligible set $S \subset B$.
	We fix a relatively positive Hermitian metric $h^L$ on $L$, $\omega := c_1(L, h^L)$, and use the notations from Proposition \ref{prop_conv_berg_family}.
	\begin{prop}\label{prop_toepl_schat_family}
		As $k \to \infty$, we have
		\begin{equation}
			\frac{1}{N_k}
			\Big\| 
				\frac{1}{k} A^{\pi}({\textrm{Hilb}}^{\pi}_k(h^L), \mathcal{F}^k) \cdot \omega_B^m
				-
				T^{\pi, h^L}_{k}(\phi^{\pi}(h^L, \mathcal{F}))  \cdot \omega_B^m
			\Big\|_{L^1(B, {\textrm{Hilb}}^{\pi}_k(h^L))}^{{\rm{tr}}}
			\to
			0.
		\end{equation}
	\end{prop}
	\begin{sloppypar}
	\begin{proof}
		By (\ref{eq_geod_ray_bnd}) and (\ref{eq_toepl_esssup}), we have $\| T^{\pi, h^L}_{k}(\phi^{\pi}(h^L, \mathcal{F}))(b) \| \leq \| \mathcal{F} \|$, where $\| \cdot \|$ is the operator norm associated with ${\textrm{Hilb}}^{\pi}_k(h^L)$.
		Moreover, directly from the definitions, $\| A^{\pi}({\textrm{Hilb}}^{\pi}_k(h^L), \mathcal{F}^k)(b) \| \leq k \| \mathcal{F} \|$, for any $k \in \nat$.
		In particular, the function $b \mapsto \frac{1}{N_k}
			{\rm{Tr}}[|
				\frac{1}{k} A^{\pi}({\textrm{Hilb}}^{\pi}_k(h^L), \mathcal{F}^k)(b)
				-
				T^{\pi, h^L}_{k}(\phi^{\pi}(h^L, \mathcal{F}))(b)|]$ is uniformly bounded.
		It also converges pointwise to $0$ away from $S$ by Theorem \ref{thm_main2prev}.
		Proposition \ref{prop_toepl_schat_family} follows directly from this, Lebesgue dominated convergence theorem and the fact that $S$ is negligible.
	\end{proof}
	\end{sloppypar}
	We also need to know the relation between convergence of functions and convergence of Toeplitz operators.
	\begin{prop}\label{prop_approx_toepl}
		Let $f_i \in L^{\infty}(X)$ be a uniformly bounded sequence of functions, converging in $L^1(X)$ towards $f \in L^{\infty}(X)$.
		Then for any $\epsilon > 0$, there is $k_0, l_0 \in \nat$, such that for any $k \geq k_0$, $l \geq l_0$, we have
		\begin{equation}
			\Big\| 
				T^{\pi, h^L}_{l}(f_k)  \cdot \omega_B^m
				-
				T^{\pi, h^L}_{l}(f)  \cdot \omega_B^m
			\Big\|_{L^1(B, {\textrm{Hilb}}^{\pi}_l(h^L))}^{{\rm{tr}}}
			\leq
			\epsilon N_l.
		\end{equation}
	\end{prop}
	\begin{proof}
		By linearity, it suffices to establish the above statement for $f := 0$.
		Remark that for any $f \in L^{\infty}(X)$, we have $\| T^{\pi, h^L}_{k}(f)  \cdot \omega_B^m \|_{L^1(B, {\textrm{Hilb}}^{\pi}_k(h^L))}^{{\rm{tr}}} \leq \frac{1}{n!} \int_{X} | T_k^{\pi, h^L}(f)(x) | c_1(L, h^L)^n \wedge \pi^* \omega_B^m$, where $T_k^{\pi, h^L}(f)(x) \in \real$ is the diagonal of the associated Schwartz kernel.
		Indeed, it is immediate if $f$ is positive, as we have an equality instead of inequality above.
		For general $f$, the result is obtained by a decomposition into the positive and negative components.
		\par 
		Directly from \cite[(5.14)]{FinSubmToepl}, we have
		\begin{equation}
			\int_{X} | T_l^{\pi, h^L}(f_k)(x) | c_1(L, h^L)^n \wedge \pi^* \omega_B^m
			\leq
			\int_{X} |f_k(x)| B_l^{\pi}(x) c_1(L, h^L)^n \wedge \pi^* \omega_B^m,
		\end{equation}
		where $B_l^{\pi}(x)$ is the fiberwise Bergman kernel, see (\ref{eq_thm_tian_fam}).
		The result now follows directly from this, (\ref{eq_thm_tian_fam}) and the asymptotic Riemann-Roch-Hirzebruch theorem.
	\end{proof}
	\par 
	Let $k \in \nat$ be such that the multiplication maps (\ref{eq_fin_gen}) are surjective for any $l \in \nat$. 
	We denote by $[{\rm{Sym}}^l \mathcal{F}^k]$ the filtration on $E_{kl}$ induced by $\mathcal{F}^k$ and (\ref{eq_fin_gen}).
	Clearly, the induced filtration $[{\rm{Sym}} \mathcal{F}^k] := \oplus_{l = 0}^{\infty} [{\rm{Sym}}^l \mathcal{F}^k]$ on $\oplus_{k = 0}^{\infty} E_{kl}$ is bounded and submultiplicative.
	We denote by $\phi^{\pi}(h^{L^{\otimes k}}, [{\rm{Sym}} \mathcal{F}^k])$ the speed of the induced geodesic ray emanating from $h^L$, defined as in (\ref{eq_geod_ray_bnd}).
	\begin{prop}\label{prop_incr_geod_rays_fam}
		As $k \to \infty$, $\frac{1}{k} \phi^{\pi}(h^{L^{\otimes k}}, [{\rm{Sym}} \mathcal{F}^k])$, increases almost everywhere to $\phi^{\pi}(h^L, \mathcal{F})$.
	\end{prop}
	\begin{proof}
		It follows from Proposition \ref{prop_incr_geod_rays} in the same way as Proposition \ref{prop_toepl_schat_family} follows from Theorem \ref{thm_main2prev}. The details are left to the reader.
	\end{proof}
	
	\begin{proof}[Proof of Theorem \ref{thm_compat_approx_sol2}]
		It follows directly from Theorem \ref{thm_HN_subm} and Propositions \ref{prop_toepl_schat_family}, \ref{prop_approx_toepl}, \ref{prop_incr_geod_rays_fam}.
	\end{proof}
	
	\par 
	Now, in order to establish Theorem \ref{thm_compat_approx_sol1}, we need to recall the stability estimates which estimate how much the weight operators vary if one varies the Hermitian metric.
	To state it, we fix a complex vector bundle $V$, $\dim V = r$, two Hermitian products $H_0$, $H_1$ and a filtration $\mathcal{F}$ on $V$.
	\begin{sloppypar}
	\begin{thm}[{\cite[Theorem 6.5]{FinSubmToepl}}]\label{thm_cholesky}
		Assume that for $C > 0$, we have
		\begin{equation}\label{eqthm_cholesky}
			1 - C
			\leq 
			\frac{H_1}{H_0} 
			\leq
			1 + C,
		\end{equation}
		where $(1 + 2 \lceil \log_2 r \rceil)^2 C < 1$.
		Then the following bound is satisfied 
		\begin{equation}
			\Big\| 
				A(H_0, \mathcal{F}) - A(H_1, \mathcal{F})
			\Big\|
			\leq
			16  C  (1 + 2 \lceil \log_2 r \rceil) \| \mathcal{F} \|,
		\end{equation}
		where $\| \cdot \|$ is the operator norm subordinate with $H_0$, and $\| \mathcal{F} \| := \sup_{v \in V \setminus \{0\}} |w_{\mathcal{F}}(v)|$.
	\end{thm}
	\end{sloppypar}
	
	\begin{proof}[Proof of Theorem \ref{thm_compat_approx_sol1}.]
		Follows directly from Theorems \ref{thm_HN_subm}, \ref{thm_cholesky}, Proposition \ref{thm_appl_ot2} and the obvious fact that the number operator verifies the property $A(H, \mathcal{F}) = A(c \cdot H, \mathcal{F})$ for $c > 0$.
	\end{proof}
	
	\section{Inductive construction of approximate critical Hermitian structures}\label{sect_induc_const}
	The primary objective of this section is to establish Theorem \ref{thm_compat_approx_sol}, which, in essence, asserts that approximate critical Hermitian structures can be constructed through an inductive process. 
	The core of our proof focuses on analyzing the behavior of the weight operator along a geodesic ray.
	\par 
	To describe this, we consider a holomorphic submersion $\pi : X \to B$ between compact Kähler manifolds $X$ and $B$ of dimensions $n + m$ and $m$ respectively.
	Consider a relatively very ample line bundle $L$ over $X$.
	For any $l \in \nat$, we define $E_l := R^0 \pi_* (L^{\otimes l})$. 
	For simplicity, assume that for any $l \in \nat^*$, the multiplication map ${\rm{Sym}}^l E_1 \to E_l$ is surjective.
	By (\ref{eq_fin_gen}), this can always be achieved by replacing $L$ by its sufficiently big power. 
	\par 
	Consider a filtration of $E_1$ by vector subbundles (it is important in this section that the filtration is given by vector subbundles and not by subsheaves)
	\begin{equation}\label{eq_g_filt}
		E_1 = \mathcal{F}_{\lambda_1} \supset \mathcal{F}_{\lambda_2} \supset \cdots \supset \mathcal{F}_{\lambda_q}.
	\end{equation}
	We fix a Hermitian metric $H$ on $E_1$, and denote by $H_s$ the geodesic ray departing from $H$ and associated with the filtration (\ref{eq_g_filt}).
	We denote by $[{\rm{Sym}}^l H_s]$ the quotient norm on $E_l$, induced by $H_s$.
	We denote by $[{\rm{Sym}}^l \mathcal{F}]$ the quotient filtration on $E_l$, defined as before Proposition \ref{prop_incr_geod_rays} (note that the latter filtration is defined by subsheaves and not subbundles in general).
	The following theorem, which we establish in Section \ref{sect_numb_op}, says that the restriction of weight operators along a geodesic ray is very close to the weight operator.
	\begin{thm}\label{thm_rest_weight_l1_comp}
		For any $\epsilon > 0$, there are $l_0 \in \nat$, $s_0 > 0$, such that for any $l \geq l_0$, $s \geq s_0$, we have
		\begin{equation}
			\Big\|  
				\Big( {\rm{Sym}}^l A(H_s, \mathcal{F})|_{E_l}
				-
				A([{\rm{Sym}}^l H_s], [{\rm{Sym}}^l \mathcal{F}]) \Big) \cdot\omega_B^m
			\Big\|_{L^1(B, [{\rm{Sym}}^l H_s ])}^{{\rm{tr}}}
			\leq
			\epsilon l
			\cdot
			\rk{E_l}.
		\end{equation}
	\end{thm}
	\par 
	Our proof will also draw on certain general results about finitely generated approximations of filtrations. 
	For this, let us recall the concept of the volume of a filtration.
	For a complex vector space $V$, $\dim V = r$, endowed with a filtration $\mathcal{F}$ with the jumping numbers $\mu_1, \ldots, \mu_r \in \real$, we define the volume, ${\rm{vol}}(\mathcal{F})$, of $\mathcal{F}$ as
	\begin{equation}
		{\rm{vol}}(\mathcal{F}) := \mu_1 + \cdots + \mu_r.
	\end{equation}
	This concept extends naturally to a family setting. 
	In this context, rather than a filtration by vector subspaces of a single vector space, we consider a filtration of a vector bundle by subsheaves. 
	The volume is then defined as the volume of the induced filtration on the generic fiber -- specifically, away from the singular set of the subsheaves.
	\par 
	We fix a complex projective manifold $Y$ polarized by an ample line bundle $F$ endowed with a positive Hermitian metric $h^F$.
	Let $\mathcal{F}$ be a bounded submultiplicative filtration on $R(Y, F)$.
	Recall that the filtration $[ {\rm{Sym}}^l \mathcal{F}^k ]$, $l \in \nat$, was defined before Proposition \ref{prop_incr_geod_rays}. 
	Recall the following result of Boucksom-Jonsson \cite[Theorem 3.18]{BouckJohn21}, cf. \cite[Proposition 2.8]{FinTits} for a proof relying on complex pluripotential theory.
	\begin{thm}\label{thm_bj_appr}
		For any bounded submultiplicative filtration $\mathcal{F}$ on $R(Y, F)$ and any $\epsilon > 0$, there is $k_0 \in \nat$, such that for any $k \geq k_0$, $l \in \nat$, we have
		\begin{equation}
			{\rm{vol}}([ {\rm{Sym}}^l \mathcal{F}^k ]) \geq {\rm{vol}}(\mathcal{F}^{kl}) - \epsilon l \dim H^0(Y, F^{\otimes kl}).
		\end{equation}
	\end{thm}
	
	Finally, we will need an additional result regarding the compatibility between the construction of geodesic rays and approximate critical Hermitian structures.
	We fix a holomorphic vector bundle $E$ over $B$.
	We denote by $\mu_0 : B_0 \to B$ a modification which resolves the Harder-Narasimhan filtration of $E$ in the sense as described in (\ref{eq_resol_filtr}).
	We denote by $\tilde{\mu}_0^* \mathcal{F}^{HN}$ the resolution of the Harder-Narasimhan filtration of $E$.
	We establish the following result in Section \ref{sect_geod_appr}.
	\begin{thm}\label{thm_ray_apprx}
		For any $\delta$-approximate critical Hermitian structure $H$ on $E$, the geodesic ray, $H_s$, $s \in [0, +\infty[$, of Hermitian metrics on $\mu_0^* E$, departing from $\mu_0^* H$ associated with the resolution of the Harder-Narasimhan filtration, $\tilde{\mu}_0^* \mathcal{F}^{HN}$, of $E$, satisfies
		\begin{equation}\label{eq_appr_crhs_resol}
		\Big\|
			\frac{\imun}{2 \pi} R^{H_s} \wedge \mu_0^* \omega_B^{m - 1}
			-
			A(H_s, \tilde{\mu}_0^* \mathcal{F}^{HN}) \cdot \mu_0^* \omega_B^m
		\Big\|_{L^1(B_0, H_s)}
		\leq
		\delta \rk{E}^3 8^{\rk{E} + 4}.
	\end{equation}
	\end{thm}
	\begin{rem}
		In particular, if the Harder-Narasimhan filtration of $E$ is given by the vector subbundles (and not by subsheaves), then the geodesic ray associated with the Harder-Narasimhan filtration departing from an $\delta$-approximate critical Hermitian structure consists of $\delta \rk{E}^3 8^{\rk{E} + 4}$-approximate critical Hermitian structures on $E$.
	\end{rem}

	\begin{proof}[Proof of Theorem \ref{thm_compat_approx_sol}.]
		Remark first that the curvature of a Hermitian vector bundle only increases under taking quotients, cf. \cite[Theorem V.14.5]{DemCompl}, so we have 
		\begin{equation}\label{eq_curv_incccc}
			\frac{\imun}{2 \pi} R^{[{\rm{Sym}}^l H_{\delta, k, s}]} \wedge \mu_k^* \omega_B^{m - 1}
			\geq_{[{\rm{Sym}}^l H_{\delta, k, s}]}
			\frac{\imun}{2 \pi} R^{{\rm{Sym}}^l H_{\delta, k, s}}|_{E_{kl}} \wedge \mu_k^* \omega_B^{m - 1},
		\end{equation}
		where we used the notation for the restriction endomorphism from Introduction.
		\par 
		By Theorem \ref{thm_ray_apprx} and the formula for the curvature of symmetric powers, for any $l \in \nat^*$, we conclude that the curvature of the Hermitian metric ${\rm{Sym}}^l H_{\delta, k, s}$ on ${\rm{Sym}}^l E_k$ satisfies
		\begin{multline}\label{eq_sym_appr}
		\Big\|  
			 \frac{\imun}{2 \pi} R^{{\rm{Sym}}^l H_{\delta, k, s}} \wedge \mu_k^* \omega_B^{m - 1} 
			 \\
			 - {\rm{Sym}}^l A(H_{\delta, k, s}, \tilde{\mu}_k^* \mathcal{F}^{HN, k}) \cdot \mu_k^* \omega_B^m
		\Big\|_{L^1(B_k, {\rm{Sym}}^l H_{\delta, k, s})} \leq \delta l N_k^3 8^{N_k + 4}.
		\end{multline}
		Since the operator's norm only decreases under restriction, by (\ref{eq_sym_appr}), we conclude that
		\begin{multline}\label{eq_sym_appr0101}
		\Big\|  
			 \frac{\imun}{2 \pi} R^{{\rm{Sym}}^l H_{\delta, k, s}}|_{E_{kl}} \wedge \mu_k^* \omega_B^{m - 1} 
			 \\
			 - {\rm{Sym}}^l A(H_{\delta, k, s}, \tilde{\mu}_k^* \mathcal{F}^{HN, k})|_{E_{kl}} \cdot \mu_k^* \omega_B^m
		\Big\|_{L^1(B_k, [{\rm{Sym}}^l H_{\delta, k, s}])} \leq \delta l N_k^3 8^{N_k + 4}.
		\end{multline}
		\par 
		From Theorem \ref{thm_rest_weight_l1_comp}, we conclude that for any $\delta, \epsilon > 0$, $k \in \nat^*$, there are $l_0 \in \nat$, $s_0 > 0$, such that for any $l \geq l_0$, $s > s_0$, we have
		\begin{multline}\label{eq_sym_appr01012}
			\Big\|  
				\Big(
			 {\rm{Sym}}^l A(H_{\delta, k, s}, \tilde{\mu}_k^* \mathcal{F}^{HN, k})|_{E_{kl}}
			 \\
			 - A( [{\rm{Sym}}^l H_{\delta, k, s}], [ {\rm{Sym}}^l \tilde{\mu}_k^* \mathcal{F}^{HN, k}]) \Big) \cdot \mu_k^* \omega_B^m
			\Big\|_{L^1(B_k, [{\rm{Sym}}^l H_{\delta, k, s}])}^{{\rm{tr}}}
			\leq
			\epsilon l
			\cdot
			N_{kl}.
		\end{multline}
		\par 
		Then from (\ref{eq_curv_incccc}), (\ref{eq_sym_appr0101}), and (\ref{eq_sym_appr01012}), we see that for
		\begin{equation}
			g_{\delta, k, l, s} :=
			\frac{\imun}{2 \pi} R^{{\rm{Sym}}^l H_{\delta, k, s}}|_{E_{kl}}
			-
			A( [{\rm{Sym}}^l H_{\delta, k, s}], [ {\rm{Sym}}^l \tilde{\mu}_k^* \mathcal{F}^{HN, k}]) \Big),
		\end{equation}
		the following two bounds are satisfied
		\begin{align}
			& \big\| g_{\delta, k, l, s} \cdot \mu_k^* \omega_B^m \big\|_{L^1(B_k, [{\rm{Sym}}^l H_{\delta, k, s}])}^{{\rm{tr}}} \leq \delta l N_k^3 8^{N_k + 4} \cdot N_{kl} + \epsilon l \cdot N_{kl},
			\label{eq_incr_quot2000}
			\\
			\nonumber
			& \frac{\imun}{2 \pi} R^{[{\rm{Sym}}^l H_{\delta, k, s}]} \wedge \mu_k^* \omega_B^{m - 1}
			\\
			& 
			\qquad \qquad \qquad \qquad
			\label{eq_incr_quot2}
			\geq_{[{\rm{Sym}}^l H_{\delta, k, s}]}
			\Big(
			A( [{\rm{Sym}}^l H_{\delta, k, s}], [ {\rm{Sym}}^l \tilde{\mu}_k^* \mathcal{F}^{HN, k}])
			+ 
			g_{\delta, k, s}  \Big) \cdot \mu_k^* \omega_B^m.
		\end{align}
		We claim that for any $\epsilon > 0$, there are $k_0, l_0 \in \nat$, such that for any $s > 0$, $\delta > 0$, $k \geq k_0$, $l \geq l_0$, we have
		\begin{multline}\label{eq_last_est_pf16}
			\int_B {\rm{Tr}} \Big[ \frac{\imun}{2 \pi} R^{[{\rm{Sym}}^l H_{\delta, k, s}]} \wedge \mu_k^* \omega_B^{m - 1} \Big] 
			\\
			\leq
			\int_B {\rm{Tr}} \Big[ A( [{\rm{Sym}}^l H_{\delta, k, s}], [ {\rm{Sym}}^l \tilde{\mu}_k^* \mathcal{F}^{HN, k}]) \Big] \cdot \mu_k^* \omega_B^m 
			+
			\epsilon l N_{kl}.
		\end{multline}
		Remark that once (\ref{eq_last_est_pf16}) is established, Theorem \ref{thm_compat_approx_sol} follows immediately from (\ref{eq_incr_quot2000}) and (\ref{eq_incr_quot2}).
		\par 
		To establish (\ref{eq_last_est_pf16}), remark that by Chern-Weil theory, we have
		\begin{equation}\label{eq_last_est_pf12}
			\int_{B_k} {\rm{Tr}} \Big[ \frac{\imun}{2 \pi} R^{[{\rm{Sym}}^l H_{\delta, k, s}]} \wedge \mu_k^* \omega_B^{m - 1} \Big]
			=
			\int_B c_1(E_{kl}) \cdot [\omega_B]^{m - 1}. 
		\end{equation}
		By the definition of the Harder-Narasimhan slopes and the fact that the first Chern class is additive, we conclude that
		\begin{equation}\label{eq_last_est_pf13}
			\int_B c_1(E_{kl}) \cdot \omega_B^{m - 1}
			=
			{\rm{vol}}(\mathcal{F}^{HN, kl}) \cdot \int_B [\omega_B]^m.
		\end{equation}
		However, directly from the definition of the weight operator, we have
		\begin{equation}\label{eq_last_est_pf1311112}
			\int_B {\rm{Tr}} \Big[ A( [{\rm{Sym}}^l H_{\delta, k, s}], [ {\rm{Sym}}^l \tilde{\mu}_k^* \mathcal{F}^{HN, k}]) \Big] \cdot \mu_k^* \omega_B^m 
			=
			{\rm{vol}}([ {\rm{Sym}}^l \tilde{\mu}_k^* \mathcal{F}^{HN, k}]) \cdot \int_B [\omega_B]^m.
		\end{equation}
		We now see that (\ref{eq_last_est_pf16}) is just a consequence of Theorem \ref{thm_bj_appr}, (\ref{eq_last_est_pf12}), (\ref{eq_last_est_pf13}) and (\ref{eq_last_est_pf1311112}).
		This finishes the proof.
	\end{proof}

	\section{Restriction of the weight operator over the geodesic ray}\label{sect_numb_op}
	The main objective of this section is to analyze the behavior of the weight operator along a geodesic ray and to establish Theorem \ref{thm_rest_weight_l1_comp}. 
	This will be achieved by studying how the construction of geodesic rays relates to the construction of quotient metrics and filtrations.
	The proof will be based on the following assertion, which, setting aside some technical statements, will be established at the end of this section. 
	We retain the notation from Theorem \ref{thm_rest_weight_l1_comp}.
	\begin{thm}\label{thm_compar_number2}
		For any $\epsilon > 0$, there is a uniformly bounded sequence of measurable functions $h_i : B \to \real$, $i \in \nat$, which converge pointwise, as $i \to \infty$, to $0$, and $l_0 \in \nat$, $s_0 \geq 1$, such that for any $l \geq l_0$, $s \geq s_0$, there is $B_{l, s} \in \ccal^{\infty}(B, \enmr{E_l})$, which is Hermitian with respect to both $[{\rm{Sym}}^l H_s]$ and $[{\rm{Sym}}^l H_0]$, so that
		\begin{equation}\label{eq_compar_number2}
		\begin{aligned}
			&
			{\rm{Sym}}^l A(H_s, \mathcal{F})|_{E_l}
			\geq_{[{\rm{Sym}}^l H_s]}
			B_{l, s},
			\\
			&
			B_{l, s}
			\geq_{[{\rm{Sym}}^l H_0]}
			A([{\rm{Sym}}^l H_0], [{\rm{Sym}}^l \mathcal{F}])
			-
			( h_{\lfloor s \rfloor}(b) + \epsilon ) \cdot l \cdot {\rm{Id}}_{E_l}.
		\end{aligned}
		\end{equation}
		Moreover, for any $s \in [0, +\infty[$, $l \in \nat^*$, we have
		\begin{equation}\label{eq_compar_number22}
			{\rm{Sym}}^l A(H_s, \mathcal{F})|_{E_l}
			\leq_{[{\rm{Sym}}^l H_s]}
			A([{\rm{Sym}}^l H_s], [{\rm{Sym}}^l \mathcal{F}]).
		\end{equation}
	\end{thm}
	\begin{proof}[Proof of Theorem \ref{thm_rest_weight_l1_comp}]
		Directly from (\ref{eq_compar_number22}), we deduce that 
		\begin{multline}\label{eq_cor_rest_weight1}
			{\rm{Tr}}\Big[
				\Big|
				{\rm{Sym}}^l A(H_s, \mathcal{F})|_{E_l}
				-
				A([{\rm{Sym}}^l H_s], [{\rm{Sym}}^l \mathcal{F}])
				\Big|
			\Big]
			\\
			=
			{\rm{Tr}} \big[
			A([{\rm{Sym}}^l H_s], [{\rm{Sym}}^l \mathcal{F}]) \big]
			-
			{\rm{Tr}} \big[
			{\rm{Sym}}^l A(H_s, \mathcal{F})|_{E_l}
			\big].
		\end{multline}
		Remark, however, that by (\ref{eq_compar_number2}), we have
		\begin{equation}\label{eq_cor_rest_weight2}
			{\rm{Tr}} \big[
			{\rm{Sym}}^l A(H_s, \mathcal{F})|_{E_l}
			\big]
			\geq
			{\rm{Tr}} \big[
			A([{\rm{Sym}}^l H_0], [{\rm{Sym}}^l \mathcal{F}])\big]
			-
			( h_{\lfloor s \rfloor}(b) + \epsilon ) \cdot l
			\cdot
			\rk{E_l}.
		\end{equation}
		Now, by dominated convergence theorem (it is at this stage that the measurability of $h_i$, along with all the other assumptions on this sequence, is being utilized), there $i_0 \in \nat$, such that for any $i \geq i_0$, we have
		\begin{equation}
			\int_B h_i(b) \cdot \omega_B^m(b) < \epsilon.
		\end{equation}
		We conclude by this, (\ref{eq_cor_rest_weight1}) and (\ref{eq_cor_rest_weight2}).
	\end{proof}
	The proof of Theorem \ref{thm_compar_number2} is based on the comparison of the related geodesic rays.
	More precisely, we denote by $[{\rm{Sym}}^l H]_s$ the geodesic ray of Hermitian metrics on $E_l$ departing from $[{\rm{Sym}}^l H]$ and associated with the filtration $[{\rm{Sym}}^l \mathcal{F}]$.
	The following result will be established in Sections \ref{sect_test_conf}, and it lies at the heart of our approach to Theorem \ref{thm_compar_number2}.
	\begin{thm}\label{thm_compar_number3}
		For any $\epsilon > 0$, there are $h_i : B \to \real$, $i \in \nat$, as in Theorem \ref{thm_compar_number2}, and $l_0 \in \nat$, $s_0 > 1$, such that for any $l \geq l_0$, $s \geq s_0$, we have
		\begin{equation}
			[{\rm{Sym}}^l H_s]
			\leq 
			[{\rm{Sym}}^l H]_s
			\cdot
			\exp((h_{\lfloor s \rfloor}(b) + \epsilon ) \cdot l s).
		\end{equation}
		Moreover, for any $l \in \nat$, $s \geq 0$, we also have $[{\rm{Sym}}^l H]_s \leq [{\rm{Sym}}^l H_s]$.
	\end{thm}
	\begin{rem}
		In the non-family version, a related result appeared in \cite[Theorem 4.8]{FinTits}.
	\end{rem}

	\par 
	Now, to establish Theorem \ref{thm_compar_number2} from Theorem \ref{thm_compar_number3}, we recall some preliminaries.
	We say that a filtration $\mathcal{F}^1$ dominates ($\geq$) $\mathcal{F}^2$ if on the level of associated non-Archimedean norms, see (\ref{eq_na_norm}), we have $\chi_{\mathcal{F}^1} \geq \chi_{\mathcal{F}^2}$.
	The following consequence of submultiplicativity, cf. \cite[\S 3.1]{FinNarSim}, will play a crucial role in what follows.
	\begin{lem}\label{lem_rest_subm}
		For any submultiplicative filtration $\mathcal{F}$ on the section ring $R(Y, F)$ of a complex projective manifold $Y$ polarized by an ample line bundle $F$, and any $k \in \nat$, $l \in \nat^*$, we have
		\begin{equation}
			[ {\rm{Sym}}^l \mathcal{F}^k ] 
			\geq 
			\mathcal{F}^{kl}.
		\end{equation}
	\end{lem}
	\par 
	To make a connection between the order on the space of filtrations and the natural order on the weight operators, we need the following result.
	\begin{lem}\label{lem_mon_numb_oper}
		For any Hermitian product $H$ on a vector space $V$ and an ordered pair of filtrations, $\mathcal{F}^1 \geq \mathcal{F}^2$, on $V$, the associated weight operators relate as $A(H, \mathcal{F}^1) \leq A(H, \mathcal{F}^2)$.
	\end{lem}
	\par
	For the proof of Lemma \ref{lem_mon_numb_oper} and further use, let us recall the following result from \cite[Proposition 4.12]{FinTits}.
	We consider a surjection $p : V \to Q$ between two finitely dimensional vector spaces $V$, $Q$.
	We fix a Hermitian metric $H^V$ on $V$ and a filtration $\mathcal{F}$. 
	We denote by $H_s^V$ the geodesic ray of Hermitian metrics on $V$ departing from $H^V$ and associated with $\mathcal{F}$.
	We denote by $[H_s^V]$ the induced quotient metric on $Q$.
	\begin{prop}\label{prop_interpol}
		Let $H_0$ (resp. $H_1$) be a fixed Hermitian metric on $V$ (resp. $Q$) and $\mathcal{F}$ (resp. $\mathcal{G}$) is a filtration on $V$ (resp. $Q$).
		We assume that $[H_0] \geq H_1$ and $[\mathcal{F}] \geq \mathcal{G}$.
		Then the geodesic ray $H_s^V$, $s \in [0, +\infty[$, of Hermitian metrics on $V$ associated with $\mathcal{F}$ and emanating from $H_0$ compares to the geodesic ray $H_s^Q$ of Hermitian metrics on $Q$ associated with $\mathcal{G}$ and emanating from $H_1$ as
		\begin{equation}\label{eq_interpol}
			[H_s^V]
			\geq
			H_s^Q.
		\end{equation}
		Moreover, for any $s_0 \geq 0$, the following identities take place
		\begin{equation}
			\begin{aligned}
			&
			([H_s^V]^{-1} \frac{d}{ds} [H_s^V])|_{s = 0}
			=
			- A([H_0], \mathcal{F})|_Q, 
			\\
			& 
			((H_s^V)^{-1} \frac{d}{ds} H_s^V)|_{s = 0} = ((H_s^V)^{-1} \frac{d}{ds} H_s^V)|_{s = s_0}, 
			\end{aligned}
		\end{equation}
		And $((H_s^V)^{-1} \frac{d}{ds} H_s^V)|_{s = 0}$ is Hermitian with respect to $H_s^V$ for any $s \in [0, +\infty[$. 
	\end{prop}
	\begin{proof}[Proof of Lemma \ref{lem_mon_numb_oper}]
		Follows directly from comparing the associated geodesic rays using Proposition \ref{prop_interpol} and taking derivative at $t = 0$. 
	\end{proof}
	\par 
	Finally, we need to compare the geodesics associated between two endpoints, and for this, the following result will play a crucial role.
	\begin{prop}[{cf. \cite[Theorem 5.4.1]{InterpSp} or \cite[Corollary 4.22]{FinSecRing}}]\label{prop_interp2}
		Let $H^V_0$, $H^V_1$ be two Hermitian metrics on $V$ and $H^Q_0$, $H^Q_1$ be the induced quotient Hermitian metrics on $Q$.
		For $s \in [0, 1]$, we denote by $H^V_s$ the geodesics between $H^V_0$ and $H^V_1$, and by $H^Q_s$ the geodesics between $H^Q_0$ and $H^Q_1$.
		Then for any $s \in [0, 1]$, we have
		\begin{equation}
			[H^V_s]
			\geq
			H^Q_s.
		\end{equation}
	\end{prop}
	\begin{proof}[Proof of Theorem \ref{thm_compar_number2}]
		We will establish that Theorem \ref{thm_compar_number2} holds for $l_0 \in \nat, C > 0$ as in Theorem \ref{thm_compar_number3}.
		We fix $s_0 \in [1, +\infty[$, and consider the geodesic $H'_s$, $s \in [0, s_0]$, of Hermitian metrics on $E_l$, such that $H'_0 = [{\rm{Sym}}^l H]$ and $H'_{s_0} = [{\rm{Sym}}^l H_{s_0}]$.
		By Proposition \ref{prop_interp2}, for any $s \in [0, s_0]$, we have 
		\begin{equation}\label{eq_comp_quot_norm}
			[{\rm{Sym}}^l H_s] \geq H'_s.
		\end{equation}
		We let $B_{l, s_0} := - ((H'_s)^{-1} \frac{d}{ds} H'_s)|_{s = 0}$ and take the derivative of (\ref{eq_comp_quot_norm}) at $s = s_0$ to get
		\begin{equation}\label{eq_comp_quot_norm1}
			[{\rm{Sym}}^l H_{s_0}]^{-1} \frac{d}{ds} [{\rm{Sym}}^l H_s]|_{s = s_0}
			\leq_{[{\rm{Sym}}^l H_{s_0}]}
			- B_{l, s_0},
		\end{equation}
		where we implicitly used the second equation from Proposition \ref{prop_interpol}.
		The first inequality of (\ref{eq_compar_number2}) then follows directly from (\ref{eq_comp_quot_norm1}) and Proposition \ref{prop_interpol}.
		\par 
		Remark now that by Theorem \ref{thm_compar_number3}, we have the following bound
		\begin{equation}\label{eq_comp_quot_norm1002}
			\exp(- s_0 B_{l, s_0})
			\leq_{[{\rm{Sym}}^l H_0]}
			\exp(
			-s_0
			A([{\rm{Sym}}^l H], [{\rm{Sym}}^l \mathcal{F}]) )
			\exp( ( h_{\lfloor s_0 \rfloor}(b) + \epsilon ) \cdot s_0 l ).
		\end{equation}
		From (\ref{eq_comp_quot_norm1002}), and the fact that the logarithm is a matrix monotone function, i.e. if two positive definite Hermitian operators, $A, B \in {\rm{End}}(V)$ on $(V, H)$ are related as $A \geq B$, then $\ln(A) \geq \ln(B)$, cf. \cite{LoewnerArt} and \cite{SimonBarryLoewn}, we have
		\begin{equation}\label{eq_comp_quot_norm2}
			- B_{l, s_0}
			\leq_{[{\rm{Sym}}^l H_0]}
			-
			A([{\rm{Sym}}^l H], [{\rm{Sym}}^l \mathcal{F}])
			+ ( h_{\lfloor s_0 \rfloor}(b) + \epsilon ) \cdot l {\rm{Id}}_{E_l},
		\end{equation}
		which gives us the second inequality of (\ref{eq_compar_number2}).
		\par 
		We fix $s_0 > 0$ and consider the geodesic ray $H''_s$, $s \in [0, +\infty[$, of Hermitian metrics on $E_l$, departing from $H''_0 = [{\rm{Sym}}^l H_{s_0}]$ and associated with $[{\rm{Sym}}^l \mathcal{F}]$.
		Then by Lemma \ref{lem_rest_subm} and Proposition \ref{prop_interpol}, we have
		\begin{equation}\label{eq_comp_quot_norm22}
			[{\rm{Sym}}^l H_{s + s_0}] \geq H''_s.
		\end{equation}
		By taking the derivatives of (\ref{eq_comp_quot_norm22}) at $s = 0$, and using Proposition \ref{prop_interpol}, we get (\ref{eq_compar_number22}) for $s := s_0$.
	\end{proof}
	\begin{rem}
		While the geometric situation here is very different, our argument on comparing the derivatives from (\ref{eq_comp_quot_norm1}) was inspired by Berndtsson \cite[(3.2) and (3.3)]{BerndtProb}.
	\end{rem}

	\section{Symmetric powers of geodesic rays and holomorphic extension theorem}\label{sect_test_conf}
	The main goal of this section is to establish Theorem \ref{thm_compar_number3}. 
	We decompose this statement into two parts, and show that one part follows from the holomorphic extension theorem, and another one from the previous work of Phong-Sturm \cite{PhongSturmRegul}.
	We conserve the notations from Section \ref{sect_numb_op}.
	\par 
	We fix a $(n, n)$-form $\eta \in \ccal^{\infty}(X, \wedge^{n, n} T^* X)$ so that its restriction to each fiber, $\eta|_{X_b}$, $b \in B$, gives a positive volume form normalized so that the volume of each fiber equals to $1$.
	For any relatively positive Hermitian metric $h^F$ on $F$, we denote by ${\rm{Hilb}}_l^{\pi}(h^F, \eta)$ a Hermitian metric on $E_l$, $l \in \nat$, defined as in (\ref{eq_l2_prod}), but instead of the standard volume form, we use $\eta|_{X_b}$.
	\par 
	We will use the following elementary lemma.
	\begin{lem}\label{lem_norm_linf}
		For any Hermitian metric $H_l$ on $E_l$, we have $H_l \geq {\rm{Hilb}}_l^{\pi}(FS(H_l)^{\frac{1}{l}}, \eta)$.
	\end{lem}
	\begin{proof}
		First of all, directly from (\ref{eq_fs_alt_defn}), we obtain, cf. \cite[Lemma 2.1]{FinNarSim}, that for any $x \in X$, $e \in F^{\otimes l}_x$, $b = \pi(x)$, the following identity takes place $
			|e|_{FS(H_l)}
			=
			\inf
			\| s \|_{H_l}$,
		where the infimum is taken over all $s \in G_{l, b}$, verifying the constraint $s(x) = e$.
		In particular, for any $f \in G_{l, b}$, we get $|f(x)|_{FS(H_l)}^2 \leq \| f \|_{H_l}^2$.
		By integrating this inequality over the whole fiber, $X_b$, with respect to the volume form $\eta|_{X_b}$, and using the fact that $\eta|_{X_b}$ is of unit volume, we obtain the result.
	\end{proof}
	\par 
	The other two ingredients in the proof of Theorem \ref{thm_compar_number3} are given below
	\begin{thm}\label{thm_ot_numb_comp}
		There are $l_0 \in \nat, C > 0$, such that for any $l \geq l_0$, $s \geq 1$, we have
		\begin{equation}\label{eq_ot_numb_comp}
			[{\rm{Sym}}^l H_s]
			\leq 
			{\rm{Hilb}}_{l}^{\pi}(FS(H_s), \eta)
			\cdot
			\exp(Cl + Cs).
		\end{equation}
	\end{thm}
	\begin{lem}\label{thm_fs_numb_comp}
		For any $\epsilon > 0$, there are $h_i : B \to \real$, $i \in \nat$, as in Theorem \ref{thm_compar_number2}, such that for any $l \geq l_0$, $s \geq 1$, we have
		\begin{equation}
			FS(H_s)
			\leq 
			FS([{\rm{Sym}}^l H]_s)^{\frac{1}{l}}
			\cdot
			\exp(( h_{\lfloor s \rfloor}(b) + \epsilon ) \cdot s).
		\end{equation}
	\end{lem}
	The proof of Theorem \ref{thm_ot_numb_comp} and Lemma \ref{thm_fs_numb_comp} is deferred until the end of this section.
	\begin{sloppypar}
	\begin{proof}[Proof of Theorem \ref{thm_compar_number3}]
		From Lemma \ref{thm_fs_numb_comp}, for any $\epsilon > 0$, there is $l_0 \in \nat$, such that for any $l \geq l_0$, $s \geq 1$, we have
		\begin{equation}\label{eq_compar_number3_1}
			{\rm{Hilb}}_{l}^{\pi}(FS(H_s), \eta)
			\leq 
			{\rm{Hilb}}_{l}^{\pi}(FS([{\rm{Sym}}^l H]_s)^{\frac{1}{l}}, \eta)
			\cdot
			\exp( ( h_{\lfloor s \rfloor}(b) + \epsilon) \cdot ls).
		\end{equation}
		Now, by Lemma \ref{lem_norm_linf}, we have 
		\begin{equation}\label{eq_compar_number3_2}
			{\rm{Hilb}}_{l}^{\pi}(FS([{\rm{Sym}}^l H]_s)^{\frac{1}{l}}, \eta)
			\leq
			[{\rm{Sym}}^l H]_s.
		\end{equation}
		The first part of Theorem \ref{thm_compar_number3} now follows directly by a combination of Theorem \ref{thm_ot_numb_comp} and (\ref{eq_compar_number3_1}), (\ref{eq_compar_number3_2}).
		The second part is a direct consequence of Proposition \ref{prop_interpol}.
	\end{proof}
	\end{sloppypar}
	\par
	Let us now prove Theorem \ref{thm_ot_numb_comp}.
	A related result in a non-family version has already appeared in \cite[Theorem 4.5]{FinTits}.
	We will use now the notations introduced in Section \ref{sect_curv}.
	\par 
	We fix a relatively positive Hermitian metric $h^F_0$ on $F$ over $P$, an $(n, n)$-form $\eta_X$ over $X$ and an $(n', n')$-form $\eta_P$ over $P$, verifying similar hypotheses as the form $\eta$ from the beginning of this section. 
	We need the following result, which we suggest to compare with Theorem \ref{thm_appl_ot2a}, from which we borrow the notations.
	\par 
	\begin{thm}\label{thm_appl_ot2a2}
		There are $C > 0$, $l_0 \in \nat$, such that for any $l \geq l_0$, and an arbitrary relatively positive Hermitian metric $h^F$ on $F$, we have 
		\begin{equation}
			[{\rm{Hilb}}_l^p(h^F, \eta_P)]
			\leq
			{\rm{Hilb}}_l^{\pi}(\iota^* h^F, \eta_X)
			\cdot
			\Big(\sup \max \Big( \frac{h^F}{h^F_0}, \frac{h^F_0}{h^F} \Big) \Big)^C
			\cdot
			\exp(C).
		\end{equation}
	\end{thm}
	\begin{proof}
		The statement can be rephrased in the following manner: for any $b \in B$, $f \in G_{l, b}$, there is $\tilde{f} \in K_{l, b}$, so that $\res_l(\tilde{f}) = f$, and we have
		\begin{equation}
			\| \tilde{f} \|_{{\rm{Hilb}}_l^p(h^F, \eta_P)}
			\leq
			\| f \|_{{\rm{Hilb}}_l^{\pi}(\iota^* h^F, \eta_X)}
			\cdot
			\Big(\sup \max \Big( \frac{h^F}{h^F_0}, \frac{h^F_0}{h^F} \Big) \Big)^C
			\cdot
			\exp(C).
		\end{equation}
		In the non-family version, such a statement appeared in \cite[Theorem 4.14]{FinTits} and \cite[Theorem 2.5]{FinSecRing} as an easy consequence of a version of Ohsawa-Takegoshi extension theorem from \cite{DemExtRed}.
		In the end, the only thing which was used there was the inequality \cite[(2.11)]{FinSecRing} which says that there is a uniform constant $p_0 \in \nat$, so that $p_0 c_1(F, h^F_0) - c_1(\wedge^{n'} T^{(1, 0)*} P_b, h^{\eta_P}) + \imun \alpha \partial \dbar \delta_{X_b} / 2 \pi$ is a non-negative current over each fiber, for any $\alpha \in [1, 2]$, where $h^{\eta_P}$ is the Hermitian metric associated with $\eta_P$, $\delta_{X_b}$ is a quasi-psh function with logarithmic singularities along $X_b$, and $h^F_0$ is an arbitrary relatively positive metric on $F$.
		But as the latter bound clearly holds in families, the whole estimate continues to hold as well.
	\end{proof}
	\begin{proof}[Proof of Theorem \ref{thm_ot_numb_comp}]
		We follow the same line of though as in the proof of Theorem \ref{thm_appl_ot} in Section \ref{sect_curv}, from which we borrow the notations.
		Directly from (\ref{eq_induc_1}), we see that in the left-hand side of (\ref{eq_ot_numb_comp}), we can replace ${\rm{Sym}}^l H_s$ by ${\rm{Hilb}}_{l}^{p}(h^{H_s})$, where $h^{H_s}$ is the Fubini-Study metric on $F := \mathscr{O}(1)$ over $\mathbb{P}(G_1^*)$ induced by $H_s$.
		From the definition of the geodesic ray, there is $C > 0$, such that
		\begin{equation}\label{eq_ot_numb_comp11}
			 h^{H_0} \cdot \exp(-Cs) \leq h^{H_s} \leq h^{H_0} \cdot \exp(Cs).
		\end{equation}
		Moreover, as the symplectic volume form on the fibers $P_b = \mathbb{P}(G_1^*)$ of $\pi : \mathbb{P}(G_1^*) \to B$, induced by $c_1(F, h^{H_s})|_{P_b}$, coincides with the Riemannian volume form induced by the Fubini-Study metric, we conclude by (\ref{eq_ot_numb_comp11}) that there is a constant $C > 0$, such that for any $l \in \nat$, $s > 0$, we have
		\begin{equation}\label{eq_ot_numb_comp12}
			{\rm{Hilb}}_{l}^{p}(h^{H_s})
			\leq
			\exp(C + C s)
			\cdot
			{\rm{Hilb}}_{l}^{p}(h^{H_s}, \eta_P),
		\end{equation}
		where $\eta_P$ is an arbitrary relative volume form chosen as in the beginning of this section.
		The result now follows directly from Theorem \ref{thm_appl_ot2a2} and (\ref{eq_ot_numb_comp11}), (\ref{eq_ot_numb_comp12}).
	\end{proof}
	\begin{proof}[Proof of Lemma \ref{thm_fs_numb_comp}]
		Let us establish that if the filtration has integer weights, Lemma \ref{thm_fs_numb_comp} holds for $\epsilon = 0$.
		\par 
		Let us fix $b \in B$ and consider the induced filtration on the section ring $R(X_b, L|_{X_b})$.
		By definition, it is a finitely generated submultiplicative filtration with integral weights.
		Hence, it corresponds to a test configuration, \cite{NystOkounTest}, \cite{SzekeTestConf}.
		Phong-Sturm in \cite[Lemma 4]{PhongSturmRegul} established that for filtrations associated with test configurations, there is $C(b) > 0$, such that for any $l \in \nat^*$, over $X_b$, we have
		\begin{equation}\label{eq_ps_stat_expl}
			FS(H_s)
			\leq 
			FS([{\rm{Sym}}^l H]_s)^{\frac{1}{l}}
			\cdot
			\exp(C(b)).
		\end{equation}
		Unfortunately, it is not very clear to the author of the current article that the constant $C(b)$ produced in \cite{PhongSturmRegul} can be chosen to depend measurably on the parameter $b$.
		\par 
		Due to this, we introduce the following workaround.
		First of all, for any $s > 0$, $l \in \nat^*$, consider the following functions
		\begin{equation}
			g_{l, s}(b) := \frac{1}{s} \sup_{x \in X_b} \log \Big( \frac{FS(H_s)(s)}{FS([{\rm{Sym}}^l H]_s)^{\frac{1}{l}}(x)} \Big).
		\end{equation}
		Note that for any $s > 0$, $l \in \nat^*$, the function $g_{l, s}(b)$ is continuous (hence measurable).
		Moreover, by definition, we have
		\begin{equation}\label{eq_fs_bnd_triv_000}
			FS(H_s)
			\leq 
			FS([{\rm{Sym}}^l H]_s)^{\frac{1}{l}}
			\cdot
			\exp(g_{l, s}(b) \cdot s).
		\end{equation}
		Note that the boundedness of the weights of the filtration implies that there is a constant $C > 0$, so that for any $s_1, s_2 > 0$, $l \in \nat^*$, we have
		\begin{equation}\label{eq_fs_bnd_triv_111}
			\big| s_1 \cdot g_{l, s_1}(b) - s_2 \cdot g_{l, s_2}(b) \big| \leq |s_1 - s_2| \cdot C.
		\end{equation}
		Now, for a given $i \in \nat^*$, we introduce the following functions
		\begin{equation}
			h_i(b) := \max\Big\{ \sup_{l \in \nat^*} g_{l, i}(b), 0 \Big\} + \frac{2 C}{i + 1}.
		\end{equation}
		Then immediately from (\ref{eq_fs_bnd_triv_000}) and (\ref{eq_fs_bnd_triv_111}), we have
		\begin{equation}
			FS(H_s)
			\leq 
			FS([{\rm{Sym}}^l H]_s)^{\frac{1}{l}}
			\cdot
			\exp(h_{\lfloor s \rfloor}(b) \cdot s).
		\end{equation}
		We claim that the functions $h_i$, $i \in \nat^*$, verify the stated assumptions.
		Indeed, they are uniformly bounded by (\ref{eq_fs_bnd_triv_111}).
		Moreover, $h_i$ are measurable, being a supremum of a countable number of measurable uniformly bounded functions.
		Finally, pointwise $h_i$ converge to $0$ by  (\ref{eq_ps_stat_expl}).
		This establishes Lemma \ref{thm_fs_numb_comp} for integer weights.
		\par 
		Note that Lemma \ref{thm_fs_numb_comp} remains valid if all the weights are changed by a multiplicative factor.
		Hence, the result holds true for filtrations with rational weights.
		By approximation (which introduces the additional $\epsilon > 0$ in Lemma \ref{thm_fs_numb_comp}), it holds for any weights.
	\end{proof}

	\section{Approximate critical Hermitian structures and geodesic rays}\label{sect_geod_appr}
	The main goal of this section is to establish a result concerning the compatibility of the construction of geodesic rays and approximate critical Hermitian structures, i.e. Theorem \ref{thm_ray_apprx}.
	\par 
	The proof is based on some local calculations of the curvature of geodesic rays.
	More precisely, for a holomorphic vector bundle $E$ over $B$ with a filtration by subbundles $E = \mathcal{F}_{\lambda_1} \supset \mathcal{F}_{\lambda_2} \supset \cdots \supset \mathcal{F}_{\lambda_q} \supset \mathcal{F}_{\lambda_{q + 1}} = \{ 0 \}$, we consider a geodesic ray $H_s$, $s \in [0, +\infty[$, associated with the filtration and departing from a Hermitian metric $H$ on $E$.
	We denote $G_i := \mathcal{F}_{\lambda_i} / \mathcal{F}_{\lambda_{i + 1}}$, $q_i = \rk{\mathcal{F}_{\lambda_i}}$, $i = 1, \ldots, q$, and assume that the filtration is decreasing, i.e. $\lambda_1 < \lambda_2 < \cdots < \lambda_q$.
	\par 
	Consider a local holomorphic frame $e_1, \ldots, e_r$ of $E$ adapted to the filtration in the sense that $e_{r - q_i + 1}, \ldots, e_r$ form a holomorphic frame $\mathcal{F}_{\lambda_i}$ for any $i = 1, \ldots, q$.
	We now construct a (non-holomorphic) frame $f_1, \ldots, f_r$ by projecting orthogonaly $e_{r - q_i + j}$, $j = 1, \ldots, q_i - q_{i + 1}$, onto the orthogonal complement of $\mathcal{F}_{\lambda_{i + 1}}$.
	It is then immediate that
	\begin{equation}\label{eq_f_hs_calc}
		\| f_{r - q_i + j} \|_{H_s} = \exp(- s \lambda_i / 2) \cdot \| f_{r - q_i + j} \|_{H_0}, \qquad \text{for } i = 1, \ldots, q \text{ and } j = 1, \ldots, q_i - q_{i + 1}.
	\end{equation}
	\par 
	We denote by $\alpha_s$, $s \in [0, +\infty[$, the connection form of the Chern connection on $(E, H_s)$ with respect to the frame $f_1, \ldots, f_r$.
	We write it in a matrix form
	\begin{equation}\label{eq_alph_beta_gen_form}
	\alpha_s 
	=
	\begin{bmatrix}
    	\beta_{11, s} & \beta_{12, s} &  \dots  & \beta_{1q, s} \\
  		\beta_{21, s} & \beta_{22, s} &  \dots  & \beta_{2q, s} \\
  		\vdots & \vdots & \ddots & \vdots \\
  		\beta_{q1, s} & \beta_{q2, s} &  \dots  & \beta_{qq, s}
	\end{bmatrix}
	\end{equation}
	where $\beta_{ij, s}$, $i, j = 1, \ldots, q$, are differential forms of degree $1$ with values in ${\rm{Hom}}(G_j, G_i)$. 
	Recall the following well-known calculation, cf. \cite[Propositions 1.6.4 - 1.6.6]{KobaVB}, \cite[Lemmas 3.1, 3.2]{JonsMcClShiva}.
	\begin{lem}\label{prop_beta_form}
		For $i > j$, $\beta_{ij, s}$ has only $(0, 1)$-differential form components, and for $i < j$, $\beta_{ij, s}$ has only $(1, 0)$-differential form components.
		Moreover, $(\beta_{ij, 0})^* = - \beta_{ji, 0}$, and we have
		\begin{equation}
			\beta_{ij, s}=
			\begin{cases}
				\beta_{ij, 0}, & \text{if } i > j \\
				\exp(-s (\lambda_j - \lambda_i)) \cdot \beta_{ij, 0}, & \text{if } i \leq j.
			\end{cases}
		\end{equation}
	\end{lem}
	We will now apply Lemma \ref{prop_beta_form} for the study of geodesic rays associated with the Harder-Narasimhan filtrations.
	More precisely, consider the Harder-Narasimhan filtration of a vector bundle $E$ as in (\ref{eq_HN_filt}).
	We use the notations (\ref{eq_resol_filtr}) for a resolution of the Harder-Narasimhan filtration, and let $G_i^{HN} = \tilde{\mu}_0^* \mathcal{F}_{\lambda_i} / \tilde{\mu}_0^* \mathcal{F}_{\lambda_{i + 1}}$, $i = 1, \ldots, q$.
	We fix a $\delta$-approximate critical Hermitian structure $H$ on $E$ and construct the geodesic ray, $H_s$, $s \in [0, +\infty[$, associated with the resolution of the Harder-Narasimhan filtration and departing from $H$, as in Theorem \ref{thm_ray_apprx}.
	\par 
	We choose a cover of $B_0$ by open subsets $U_u$, $u = 1, \ldots, v$, with a subordinate partition of unity $\rho_u : X \to [0, 1]$, and local holomorphic frames $e_1^u, \ldots, e_r^u$ of $\mu_0^* E$ over $U_u$ adapted to the filtration $\tilde{\mu}_0^* \mathcal{F}$ as before (\ref{eq_alph_beta_gen_form}).
	We denote by $\alpha_s^{HN, u}$, $s \in [0, +\infty[$, $u = 1, \ldots, v$, the connection form of the Chern connection on $(E, H_s)$ with respect to the frame $f_1^u, \ldots, f_r^u$ associated with $e_1^u, \ldots, e_r^u$ in the same way as $f_1, \ldots, f_r$ was associated with $e_1, \ldots, e_r$ in (\ref{eq_alph_beta_gen_form}).
	We denote by $\beta_{i j, s}^{HN, u}$, $i, j = 1, \ldots, q$, the components of the associated connection form as in (\ref{eq_alph_beta_gen_form}).
	\par 
	We denote by $R^{G_i^{HN}}$, $i = 1, \ldots, q$, the curvature of the Chern connection on $G_i^{HN}$ with the metric induced by $H$.
	For $b \in U_u$, we denote by $dz_l(b) \in T_b^{1, 0 *} B_0$, $l = 1, \ldots, m$, some orthogonal frame (with respect to $\mu_0^* \omega_B$).
	For $i, j = 1, \ldots, q$, $i < j$, $u = 1, \ldots, v$, $l = 1, \ldots, m$, we define the section $A_{i j l}^u$ of ${\rm{Hom}}(G_j^{HN}, G_i^{HN})$ as follows
	\begin{equation}\label{eq_beta_dec_1}
		\beta_{i j, 0}^{HN, u}
		=
		\sum_{l = 1}^m
		d z_l \cdot A_{i j l}^u.
	\end{equation}
	Then by Lemma \ref{prop_beta_form}, for any $i > j$, we have 
	\begin{equation}\label{eq_beta_dec_2}
		\beta_{i j, 0}^{HN, u}
		=
		- \sum_{l = 1}^m
		d \overline{z}_l \cdot A_{j i l}^{u *}.
	\end{equation}
	We also denote by $R^E_{i j, s} \in \ccal^{\infty}(B, \wedge^{1, 1} T^* B \otimes {\rm{Hom}}(G_j^{HN}, G_i^{HN}))$ the components of the curvature tensor of the Chern connection on $(\mu_0^* E, H_s)$.
	\begin{lem}\label{lem_bnd_off_diag_curv}
		For any $i, j = 1, \ldots, q$, $i \neq j$, the following estimates hold 
		\begin{align}\label{eq_bnd_off_diag_curv}
			&
			\sum_{u = 1}^{v}
			\big\| 
			\rho_u \cdot d \beta_{i j, 0}^{HN, u} \wedge \mu_0^* \omega_B^{m - 1} 
			\big\|_{L^1(U_u, H_0)}
			\leq
			2 \pi \delta r q 8^{q + 3},
			\\
			\label{eq_bnd_off_diag_curv00}
			& 
			\sum_{u = 1}^{v}
			\big\| 
			\rho_u \cdot \beta_{i j, 0}^{HN, u} \wedge \beta_{j i, 0}^{HN, u}  \wedge \mu_0^* \omega_B^{m - 1} 
			\big\|_{L^1(U_u, H_0)}
			\leq
			2 \pi \delta r 8^{q - \max(i, j) + 1}.
		\end{align}
		Moreover, we have
		\begin{equation}\label{eq_bnd_off_diag_curv0}
			\big\| 
			R^{G_i^{HN}}  \wedge \mu_0^* \omega_B^{m - 1} 
			-
			\lambda_i {\rm{Id}}_{G_i^{HN}} \cdot \mu_0^* \omega_B^m 
			\big\|_{L^1(B_0, H_0)}
			\leq
			2 \pi \delta r 8^{q - i + 1}.
		\end{equation}
	\end{lem}
	\begin{proof}
		Remark first that by Lemma \ref{prop_beta_form}, in order to prove (\ref{eq_bnd_off_diag_curv}), (\ref{eq_bnd_off_diag_curv00}), in their full generality, it is enough to establish them for $i < j$.
		By the definition of $\beta_{i j, 0}^{HN, u}$, cf. \cite[(V.14.6)]{DemCompl}, we have
		\begin{equation}\label{eq_bnd_off_diag_curv2}
			\imun R^E_{i i, 0} =  \imun R^{G_i^{HN}} + \sum_{\substack{j = 1 \\ j \neq i}}^{q} \sum_{u = 1}^{v} \rho_u \imun \beta_{i j, 0}^{HN, u} \wedge \beta_{j i, 0}^{HN, u}.
		\end{equation}
		However, directly from (\ref{eq_beta_dec_1}) and (\ref{eq_beta_dec_2}), we see that 
		\begin{multline}\label{eq_bnd_off_diag_curv3}
			 \beta_{i j, 0}^{HN, u} \wedge \beta_{j i, 0}^{HN, u} \wedge \mu_0^* \omega_B^{m - 1}
			 =
			 \begin{cases}
			 	 &
			 	 - \sum_{l = 1}^m dz_l \wedge d\overline{z}_l \cdot A_{i j l}^{u} A_{i j l}^{u *}  \wedge \mu_0^* \omega_B^{m - 1}, \quad \text{if } i < j,
				 \\
				 &
				 \sum_{l = 1}^m dz_l \wedge d\overline{z}_l \cdot A_{j i l}^{u *} A_{j i l}^{u} \wedge \mu_0^* \omega_B^{m - 1}, \quad \text{if } i > j.
			 \end{cases}
		\end{multline}
		Also, by the Chern-Weil theory, we have
		\begin{equation}\label{eq_bnd_off_diag_curv4}
			\int_B \tr{\imun R^{G_i^{HN}}} \wedge \mu_0^* \omega_B^{m - 1}
			=
			2 \pi
			\int_B c_1(G_i^{HN}) \mu_0^* [\omega_B]^{m - 1}.
		\end{equation}
		Since $H$ is an $\delta$-approximate critical Hermitian structure, we have
		\begin{equation}\label{eq_bnd_off_diag_curv5}
			\Big\| 
			\imun R^E_{i i, 0} \wedge \mu_0^* \omega_B^{m - 1}
			-
			2 \pi \lambda_i
			{\rm{Id}}_{G_i^{HN}}
		 	\mu_0^* \omega_B^m
		 	\Big\|_{L^1(B_0, H_0)}^{{\rm{tr}}}
			\leq 
			2 \pi \delta r.
		\end{equation}
		And by the definition of $\lambda_i$, we have
		\begin{equation}
			\lambda_i \int_B \omega_B^m
			=
			\int_B c_1(\mathcal{F}_{\lambda_i} / \mathcal{F}_{\lambda_{i + 1}}) [\omega_B]^{m - 1}.
		\end{equation}
		Remark, however, that the difference between $c_1(\mu_0^* \mathcal{F}_{\lambda_i})$ and $c_1(\tilde{\mu}_0^* \mathcal{F}_{\lambda_i})$ is supported over the exceptional locus of $\mu_0$.
		Hence, we have
		\begin{equation}\label{eq_bnd_off_diag_curv6}
			\int_{B_0} c_1(G_i^{HN}) \mu_0^* [\omega_B]^{m - 1}
			=
			\int_B c_1(\mathcal{F}_{\lambda_i} / \mathcal{F}_{\lambda_{i + 1}}) [\omega_B]^{m - 1}.
		\end{equation}
		A combination of (\ref{eq_bnd_off_diag_curv4})-(\ref{eq_bnd_off_diag_curv6}) gives us 
		\begin{equation}\label{eq_bnd_off_diag_curv66}
			\int_B \tr{\imun R^{G_i^{HN}}} \wedge \mu_0^* \omega_B^{m - 1}
			\geq 
			\int_B \tr{\imun R^E_{ii, 0}} \wedge \mu_0^* \omega_B^{m - 1} - 2 \pi \delta r.
		\end{equation}
		\par Remark now that by (\ref{eq_bnd_off_diag_curv3}), in the sum (\ref{eq_bnd_off_diag_curv2}), each summand is negative for $i = q$ (it goes in line with the fact that the the curvature of a Hermitian vector bundle decreases under taking subbundles, cf. \cite[Theorem V.14.5]{DemCompl}).
		Hence, we have
		\begin{equation}\label{eq_bnd_off_diag_curv666}
			\imun R^{G_q^{HN}} \wedge \mu_0^* \omega_B^{m - 1}
			\leq
			\imun R^E_{qq, 0} \wedge \mu_0^* \omega_B^{m - 1}.
		\end{equation}
		From (\ref{eq_bnd_off_diag_curv66}) and (\ref{eq_bnd_off_diag_curv666}), we obtain 
		\begin{equation}\label{eq_bnd_off_diag_curv6666}
			\Big\|
				R^E_{qq, 0} \wedge \mu_0^* \omega_B^{m - 1}
				-
				R^{G_q^{HN}} \wedge \mu_0^* \omega_B^{m - 1}
			\Big\|_{L^1(U_u, H_0)}^{{\rm{tr}}}
			\leq
			2 \pi \delta r.
		\end{equation}
		From (\ref{eq_bnd_off_diag_curv5}) and (\ref{eq_bnd_off_diag_curv6666}), we obtain (\ref{eq_bnd_off_diag_curv0}) for $i = q$.
		Moreover, from (\ref{eq_bnd_off_diag_curv2}) and (\ref{eq_bnd_off_diag_curv6666}), we get
		\begin{equation}\label{eq_bnd_off_diag_curv7}
			\sum_{u = 1}^{v}
			\Big\| \rho_u \cdot \sum_{j = 0}^{q - 1} \imun \beta_{q j, 0}^{HN, u} \wedge \beta_{j q, 0}^{HN, u}  \wedge \mu_0^* \omega_B^{m - 1} \Big\|_{L^1(U_u, H_0)}^{{\rm{tr}}}
			\leq
			2 \pi \delta r.
		\end{equation}
		By the linearity of the trace and again the fact that in the sum under the norm (\ref{eq_bnd_off_diag_curv7}), each summand is positive by (\ref{eq_bnd_off_diag_curv3}), we see that (\ref{eq_bnd_off_diag_curv7}) implies (\ref{eq_bnd_off_diag_curv00}) for $i = q$.
		\par 
		Let us now describe the first step of induction, i.e. establish (\ref{eq_bnd_off_diag_curv00}) and (\ref{eq_bnd_off_diag_curv0}) for $i = q - 1$.
		First, by the similar argument as in (\ref{eq_bnd_off_diag_curv666}), we get
		\begin{equation}\label{eq_bnd_off_diag_curv667}
			\imun R^{G_{q-1}^{HN}} \wedge \mu_0^* \omega_B^{m - 1}
			+
			\sum_{u = 1}^{v} \rho_u \imun \beta_{q-1 q, 0}^{HN, u} \wedge \beta_{q q-1, 0}^{HN, u}
			\leq
			\imun R^E_{q-1 q-1, 0} \wedge \mu_0^* \omega_B^{m - 1}.
		\end{equation}
		But then we already know that (\ref{eq_bnd_off_diag_curv00}) holds for $i = q$, and since the statement is symmetric in $i, j$, it also holds for $j = q$.
		From this, (\ref{eq_bnd_off_diag_curv66}) and (\ref{eq_bnd_off_diag_curv667}), we obtain 
		\begin{equation}\label{eq_bnd_off_diag_curv6677}
			\Big\|
				R^E_{q-1 q-1, 0} \wedge \mu_0^* \omega_B^{m - 1}
				-
				R^{G_{q-1}^{HN}} \wedge \mu_0^* \omega_B^{m - 1}
			\Big\|_{L^1(B_0, H_0)}
			\leq
			2 \pi \delta r 9.			
		\end{equation}
		From (\ref{eq_bnd_off_diag_curv5}) and (\ref{eq_bnd_off_diag_curv6677}), we get (\ref{eq_bnd_off_diag_curv0}) for $i = q - 1$.
		From (\ref{eq_bnd_off_diag_curv2}), the fact that the estimate (\ref{eq_bnd_off_diag_curv00}) holds for $i = q$ and (\ref{eq_bnd_off_diag_curv6677}), we get (\ref{eq_bnd_off_diag_curv00}) for $i = q - 1$.
		The rest of induction is done similarly.
		\par 
		Let us establish (\ref{eq_bnd_off_diag_curv}).
		Since $H$ is an $\delta$-approximate critical Hermitian structure, for $i \neq j$,
		\begin{equation}\label{eq_bnd_off_diag_curv8}
			\sum_{u = 1}^{v}
			\Big\| \rho_u \cdot \Big( d \beta_{i j, 0}^{HN, u} + \sum_{k = 1}^q \beta_{i k, 0}^{HN, u} \wedge \beta_{k j, 0}^{HN, u} \Big)  \wedge \mu_0^* \omega_B^{m - 1} \Big\|_{L^1(U_u, H_0)}
			\leq
			2 \pi \delta.
		\end{equation}
		However, by Cauchy-Schwarz inequality, we have
		\begin{multline}\label{eq_bnd_off_diag_curv9}
			\Big( \sum_{u = 1}^{v} \Big\| \rho_v \cdot \beta_{i k, 0}^{HN, u} \wedge \beta_{k j, 0}^{HN, u}  \wedge \mu_0^* \omega_B^{m - 1} \Big\|_{L^1(U_u, H_0)}^{{\rm{tr}}} \Big)^2
			\\
			\leq
			\Big( \sum_{u = 1}^{v} 
			\Big\| \rho_v \cdot \beta_{i k, 0}^{HN, u} \wedge \beta_{k i, 0}^{HN, u} \wedge \mu_0^* \omega_B^{m - 1} \Big\|_{L^1(U_u, H_0)}^{{\rm{tr}}}
			\Big)
			\cdot
			\\
			\cdot
			\Big( \sum_{u = 1}^{v} 
			\Big\| \rho_v \cdot \beta_{j k, 0}^{HN, u} \wedge \beta_{k j, 0}^{HN, u} \wedge \mu_0^* \omega_B^{m - 1} \Big\|_{L^1(U_u, H_0)}^{{\rm{tr}}}
			\Big).
		\end{multline}
		We obtain (\ref{eq_bnd_off_diag_curv}) from (\ref{eq_bnd_off_diag_curv00}), (\ref{eq_bnd_off_diag_curv8}) and (\ref{eq_bnd_off_diag_curv9}).
	\end{proof}
	As an application of Lemmas \ref{prop_beta_form} and \ref{lem_bnd_off_diag_curv}, we establish the following result.
	\begin{lem}\label{lem_bnd_curv_offdi}
		For any $i \neq j$, $s > 0$, the following estimates hold 
		\begin{equation}\label{eq_bnd_curv_offdi}
		\begin{aligned}
			&
			\big\| 
			\imun R^E_{i j, s} \wedge \mu_0^* \omega_B^{m - 1} 
			\big\|_{L^1(B_0, H_s)}
			\leq
			2 \pi \delta r q 8^{q + 4},
			\\
			&
			\big\| 
			\imun R^E_{i i, s} \wedge \mu_0^* \omega_B^{m - 1} 
			-
			2 \pi \lambda_i \cdot {\rm{Id}}_{G_i^{HN}} \wedge \mu_0^* \omega_B^m
			\big\|_{L^1(B_0, H_s)}
			\leq
			2 \pi \delta r q 8^{q + 3},
		\end{aligned}
		\end{equation}
	\end{lem}
	\begin{proof} 
		We will establish the first bound of (\ref{eq_bnd_curv_offdi}) under the assumption $i < j$, as the other case is completely analogous.
		By Lemma \ref{prop_beta_form}, for $i \leq j$, we can write
		\begin{equation}\label{eq_bnd_curv_offdi3}
		\begin{aligned}
			R^E_{i j, s}
			=
			d \beta_{i j, 0}^{HN, u} \cdot \exp(- s (\lambda_j - \lambda_i))
			& +
			\sum_{k = 1}^i
			\beta_{i k, 0}^{HN, u} \wedge \beta_{k j, 0}^{HN, u}  \cdot \exp(- s (\lambda_j - \lambda_k))
			\\
			& +
			\sum_{k = i + 1}^j
			\beta_{i k, 0}^{HN, u} \wedge \beta_{k j, 0}^{HN, u}  \cdot \exp(- s (\lambda_j - \lambda_i))
			\\
			& +
			\sum_{k = j + 1}^q
			\beta_{i k, 0}^{HN, u} \wedge \beta_{k j, 0}^{HN, u}  \cdot \exp(- s (\lambda_k - \lambda_i)).
		\end{aligned}
		\end{equation}
		Remark now that by (\ref{eq_f_hs_calc}) for a section $A$ of ${\rm{Hom}}(G_j^{HN}, G_i^{HN})$, we have
		\begin{equation}\label{eq_bnd_curv_offdi4}
			\big\| 
			A \mu_0^* \omega_B^m
			\big\|_{L^1(U_u, H_s)}
			=
			\big\| 
			A \mu_0^* \omega_B^m
			\big\|_{L^1(U_u, H_0)}
			\cdot
			\exp(s (\lambda_j - \lambda_i) / 2).
		\end{equation}
		Then by (\ref{eq_bnd_curv_offdi3}) and (\ref{eq_bnd_curv_offdi4}) it is clear that we have
		\begin{multline}\label{eq_bnd_curv_offdi5}
			\big\| 
			\imun R^E_{i j, s} \wedge \mu_0^* \omega_B^{m - 1} 
			\big\|_{L^1(B_0, H_s)}
			\leq
			\big\| 
			d \beta_{i j, 0}^{HN, u} \wedge \mu_0^* \omega_B^{m - 1} 
			\big\|_{L^1(B_0, H_0)}
			\\
			+
			\sum_{k = 1}^q
			\big\| 
			\beta_{i k, 0}^{HN, u} \wedge \beta_{k j, 0}^{HN, u}  \wedge \mu_0^* \omega_B^{m - 1} 
			\big\|_{L^1(B_0, H_0)}.
		\end{multline}
		The first bound of (\ref{eq_bnd_curv_offdi}) for $i < j$ then follows directly from Lemma \ref{lem_bnd_off_diag_curv}, (\ref{eq_bnd_off_diag_curv9}) and (\ref{eq_bnd_curv_offdi5}). 
		\par 
		Similarly, by (\ref{eq_bnd_curv_offdi3}), we have
		\begin{equation}\label{eq_bnd_curv_offdi1}
			R^E_{i i, s}
			=
			R^{G_i^{HN}}
			+
			\sum_{\substack{k = 1 \\ k \neq i}}^q
			\beta_{i k, 0}^{HN, u} \wedge \beta_{k i, 0}^{HN, u}  \cdot \exp(- s |\lambda_i - \lambda_k|).
		\end{equation}
		The second bound from (\ref{eq_bnd_curv_offdi}) then follows from Lemma \ref{lem_bnd_off_diag_curv} and (\ref{eq_bnd_curv_offdi1}).
	\end{proof}
	\begin{sloppypar}
	\begin{proof}[Proof of Theorem \ref{thm_ray_apprx}.]
		Remark that the components of the weight operator, $A(H_s, \tilde{\mu}_0^* \mathcal{F}^{HN})_{i j} \in {\rm{Hom}}(G_j^{HN}, G_i^{HN})$, in the frame $f_1^u, \ldots, f_r^u$ are given by $
			A(H_s, \tilde{\mu}_0^* \mathcal{F}^{HN})_{i j}
			=
			\lambda_i \delta_{i j} {\rm{Id}}_{G_i^{HN}}$,
		where $\delta_{i j}$ is the Kronecker delta.
		The result now follows directly from this and Lemma \ref{lem_bnd_curv_offdi}.
	\end{proof}
	\end{sloppypar}

	\section{Mehta-Ramanathan type formula for the Wess-Zumino-Witten functional}\label{sect_mr}
	\begin{sloppypar}
	The main goal of this section is to describe an application of Theorem \ref{thm_main} giving a Mehta-Ramanathan type formula for the Wess-Zumino-Witten functional.
	Roughly, this formula says that for projective families, the value of the Wess-Zumino-Witten functional can be determined from the values of the respective functionals on the restriction of our family to generic curves.
	\end{sloppypar}
	\par 
	To state our result precisely, we assume that $[\omega_B] \in H^2(X, \mathbb{Z})$ (in particular, $B$ is projective) and $m \geq 2$.
	By Bertini's theorem, there is $l_0 \in \nat$, such that for any $l \geq l_0$, a generic curve $C_l$, $\iota_l: C_l \hookrightarrow B$ obtained as an intersection of $m - 1$ generic divisors in the class $l [\omega_B]$, is regular. 
	Denote by $Y_l$ the pull-back of $\iota_l$ and $\pi$, and by $\pi_l : Y_l \to C_l$, $i_l : Y_k \hookrightarrow X$ the natural corresponding maps, verifying the following commutative diagram
	\begin{equation}\label{eq_sh_exct_seq0}
	\begin{tikzcd}
		Y_l \arrow[hookrightarrow]{r}{i_l} \arrow{d}{\pi_l} & X \arrow{d}{\pi} \\
 		C_l \arrow[hookrightarrow]{r}{\iota_l} & B.
	\end{tikzcd}
	\end{equation}
	The main result of this section goes as follows.
	\begin{thm}\label{thm_mr_wzw}
		For any $l \geq l_0$, $t \in \real$, the value $\inf \int_{Y_l} |\beta^{n + 1}|$, where the infimum is taken over all smooth closed $(1, 1)$-forms $\beta$ in the class $i_l^*( c_1(L) -  t \pi^* [\omega_B])$, which are positive along the fibers of $\pi_l$, is independent on the choice of a generic curve $C_l$.
		Moreover, the limit below exists, and the following formula holds
		\begin{equation}\label{eq_mr_wzw}
			{\rm{WZW}}(c_1(L) - t \pi^* [\omega_B], \omega_B) 
			=
			\lim_{l \to \infty} \frac{1}{l^{m - 1}} \inf \int_{Y_l} |\beta^{n + 1}|.
		\end{equation}
	\end{thm}
	\par 
	The proof of Theorem \ref{thm_mr_wzw} is based on a combination of Theorem \ref{thm_main} and Mehta-Ramanathan type theorem for the measures $\eta^{HN}$.
	Recall that the classical Mehta-Ramanathan theorem from \cite{MehtaRamMathAnn}, \cite{MehtaRama} says that if a vector bundle $E$ is (semi)stable, then for $l \in \nat$ big enough, the vector bundle $i_l^* E$ is (semi)stable over a generic curve $C_l$.
	In particular, the Harder-Narasimhan slopes of an arbitrary vector bundle can be recovered from the Harder-Narasimhan slopes of the restriction of this vector bundle to generic curves of sufficiently large degree.
	\par 
	The main result of \cite{FinHNI} roughly says that a weak uniform version of the Mehta-Ramanathan theorem holds if instead of a single vector bundle $E$, we consider a sequence of vector bundles $E_k$, $k \in \nat$, given by direct images.
	More precisely, we denote by $\eta^{HN, l}$, $l \in \nat^*$, the measure constructed similarly to $\eta^{HN}$, but associated with the family $\pi_l: Y_l \to C_l$, $i_l^* L$, $[\omega_B]|_{C_l}$, where $C_l$ is the generic curve as above (it is standard that the Harder-Narasimhan slopes of the restrictions of a vector bundle to curves, given by complete intersections, are independent on the choice of \textit{generic} curve, cf. \cite[Corollary 3.8]{FinHNI}).
	The main result of \cite{FinHNI} goes as follows.
	\begin{thm}[{\cite[Theorem 1.2]{FinHNI}}]\label{thm_mr_HN}
		The measures $\eta^{HN, l}$ converge weakly to $\eta^{HN}$, as $l \to \infty$.
	\end{thm}
	\par 
	\begin{sloppypar}
	\begin{proof}[Proof of Theorem \ref{thm_mr_wzw}]
		By applying Theorem \ref{thm_main} to $\pi_l: Y_l \to C_l$, $i_l^* L$ and $[\omega_B]|_{C_l}$, we obtain
		\begin{equation}\label{eq_main_2}
			\inf \int_{Y_l} |\beta^{n + 1}| = \int_{x \in \real} |x - t| d \eta^{HN, l}(x)
			\cdot
			\int_{Y_l} c_1(i_l^* L)^n \pi_l^* \iota_l^* [\omega_B]
			\cdot
			(n + 1),
		\end{equation}
		where the infimum is taken over $\beta$ as in (\ref{eq_mr_wzw}).
		This formula and the fact that $\eta^{HN, l}$ doesn't depend on the choice of generic curve implies that the value $\inf \int_{Y_l} |\beta^{n + 1}|$ doesn't depend on the choice of generic curve either.
		Remark that by the definition of $C_l$, we have $\int_{Y_l} c_1(i_l^* L)^n \pi_l^* \iota_l^* [\omega_B] = l^{m - 1} \cdot \int_X c_1(L)^n \pi^* [\omega_B]^m$.
		The result now follows from this, Theorem \ref{thm_mr_HN}, (\ref{eq_main}) and (\ref{eq_main_2}).
	\end{proof}		
	\end{sloppypar}

	\section{Asymptotic cohomology and the absolute Monge-Ampère functional}\label{sect_dem_conj}
	The primary objective of this section is to prove Corollary \ref{thm_andgr}. 
	We achieve this by interpreting Conjecture 1 in terms of a related conjecture on the sharp lower bound for the absolute Monge-Ampère functional. 
	Then, we apply Theorem \ref{thm_main} alongside calculations involving the Harder-Narasimhan measures.
	Let us introduce some notations first. 
	We fix a compact complex manifold $Y$ of dimension $n + 1$.
	For an arbitrary class $[\alpha] \in H^{1, 1}(Y)$ and a smooth closed $(1, 1)$-differential form $\alpha$, we introduce the absolute Monge-Ampère functional as follows
	\begin{equation}\label{defn_ama_func}
		|{\rm{MA}}|(\alpha) := \int_Y |\alpha^{n + 1}|, \qquad |{\rm{MA}}|([\alpha]) = \inf |{\rm{MA}}|(\alpha),
	\end{equation}
	where the infimum is taken over all smooth closed $(1, 1)$-forms $\alpha$ in the class $[\alpha]$.
	We fix a holomorphic line bundle $F$.
	\begin{prop}\label{prop_ma_low_bnd}
		For an arbitrary smooth closed $(1, 1)$-form $\alpha$ in the class $c_1(F)$, we have
		\begin{equation}\label{eq_ma_abs_sum}
			|{\rm{MA}}|(\alpha) \geq \sum_{i = 0}^{n + 1} \hat{h}^q(Y, F).
		\end{equation}
	\end{prop}
	\begin{proof}
		By the definition of the sets $Y(\alpha, q)$ from (\ref{eq_holmi}), we can rewrite 
		\begin{equation}
			|{\rm{MA}}|(\alpha) = \sum_{i = 0}^{n + 1} \int_{Y(\alpha, q)} (-1)^q \alpha^{n + 1}.
		\end{equation}
		The result now follows directly from (\ref{eq_holmi}).
	\end{proof}
	\par 
	\noindent \textbf{Conjecture 2.} We have 
	$|{\rm{MA}}|(c_1(F)) = \sum_{i = 0}^{n + 1} \hat{h}^q(Y, F)$.
	\vspace*{0.3cm}	
	\par 
	Remark that from (\ref{eq_holmi}) and (\ref{eq_ma_abs_sum}), Conjecture 2 refines Conjecture 1. Moreover, Conjecture 2 is equivalent to the following statement: for any $\epsilon > 0$, there is a smooth closed $(1, 1)$-form $\alpha_{\epsilon}$ in the class $c_1(F)$, such that $\int_{Y(\alpha_{\epsilon}, q)} (-1)^q \alpha_{\epsilon}^n \leq \hat{h}^q(Y, F) + \epsilon$, for any $q = 0, \ldots, n + 1$.
	\par 
	We argue that Conjecture 2 refines the trivial lower bound given by the triangle inequality, i.e.	
	\begin{equation}\label{eq_rrh1001}
		\sum_{q = 0}^{n + 1} \hat{h}^q(Y, F) 
		\geq
		\Big| \int_Y c_1(F)^{n + 1} \Big|.
	\end{equation}
	Indeed, by Riemann-Roch-Hirzebruch theorem, we have
	\begin{equation}\label{eq_rrh1}
		\sum_{q = 0}^{n + 1} (-1)^q \dim H^q(Y, F^{\otimes k})
		=
		\int_Y {\rm{Td}}(TY) \cdot {\rm{ch}}(F^{\otimes k}),
	\end{equation}
	where ${\rm{Td}}$ and ${\rm{ch}}$ are Todd and Chern classes.
	We take the absolute value from each side of (\ref{eq_rrh1}), apply the triangle inequality on the left-hand side, divide by $k^{n + 1}$ and take a limit $k \to \infty$ to get (\ref{eq_rrh1001}).
	Remark also that if we knew that the $\limsup$ in (\ref{defn_as_coh}) is actually a limit, then from (\ref{eq_rrh1}), we would get immediately
	\begin{equation}\label{eq_rrh0}
		\sum_{q = 0}^{n + 1} (-1)^q \hat{h}^q(Y, F) = \int_Y c_1(F)^{n + 1}.
	\end{equation}	
	\par 
	We can now state the main result of this section, refining Corollary \ref{thm_andgr}.
	\begin{thm}\label{thm_maabs_conj}
		In the setting of Corollary \ref{thm_andgr}, Conjecture 2 holds.
		Moreover, in this setting, it is enough to consider relatively positive $\alpha$ in the definition of $|{\rm{MA}}|([\alpha])$, i.e. for an arbitrary Kähler form $\omega_B$ on $Y$, we have $|{\rm{MA}}|([\alpha]) = {\rm{WZW}}([\alpha], \omega_B)$.
	\end{thm}
	\begin{proof}
		We conserve the notations from Corollary \ref{thm_andgr}.
		We first argue that
		\begin{equation}\label{eq_hhat_van}
			\hat{h}^q(Y, F) = 0, \text{ for any } q > 1.
		\end{equation}
		By Serre vanishing theorem, it is immediate to see that Leray spectral sequence associated with $F^{\otimes k}$ and $\pi$ degenerates at the second page for $k \in \nat$ big enough.
		In particular, for any $q = 0, \ldots, n + 1$, and $k$ large enough, we have
		\begin{equation}\label{eq_leray}
			H^q(Y, F^{\otimes k}) =  H^q(C, E_k). 
		\end{equation}		 
		However, since $C$ is a Riemann surface, we clearly have $H^q(C, E_k) = 0$ for any $q > 1$. 
		Together with (\ref{eq_leray}), this implies (\ref{eq_hhat_van}).
		\par 
		Now, since the base of the fibration is $1$-dimensional, the value under the integral of the Wess-Zumino-Witten functional coincides with the absolute Monge-Ampère functional, and Theorem \ref{thm_main} for the class $[\omega_C] \in H^{1, 1}(C, \real)$ and a Kähler form $\omega_C \in [\omega_C]$, verifying $\int_C [\omega_C] = 1$, gives us
		\begin{equation}\label{eq_MAVol_opt}
			|{\rm{MA}}|(c_1(F))
			\leq
			{\rm{WZW}}(c_1(F), \omega_C)
			=
			\int_{x \in \real} |x| d \eta^{HN}(x)
			\cdot
			\int c_1(F)^n \pi^* [\omega_C]
			\cdot
			(n + 1).
		\end{equation}
		It is one of the two crucial places in the proof where we use the fact that $\dim B = 1$.
		\par 
		We argue that 
		\begin{equation}\label{eq_HN_asymp_coh}
			\int_{x \in \real} |x| d \eta^{HN}(x)
			\cdot
			\int c_1(F)^n \pi^* [\omega_C]
			\cdot
			(n + 1)
			=
			\hat{h}^0(Y, F) + \hat{h}^1(Y, F).
		\end{equation}
		Once (\ref{eq_HN_asymp_coh}) is established, the proof of Theorem \ref{thm_maabs_conj} follows by Proposition \ref{prop_ma_low_bnd}, (\ref{eq_hhat_van}), (\ref{eq_MAVol_opt}) and (\ref{eq_HN_asymp_coh}).
		\par 
		Let us now establish (\ref{eq_HN_asymp_coh}).
		We first show that 
		\begin{equation}\label{eq_x_HN}
			\int_{x \in \real} x d \eta^{HN}(x) \cdot
			\int c_1(F)^n \pi^* [\omega_C]
			\cdot
			(n + 1)
			=
			\int_Y c_1(F)^{n + 1}.
		\end{equation}
		Recall that $\eta^{HN}$ is the weak limit of $\eta^{HN}_k$ (the latter measures were defined in (\ref{eq_eta_defn})), as $k \to \infty$. In particular, we have
		\begin{equation}\label{eq_x_HN1}
			\int_{x \in \real} x d \eta^{HN}(x)
			=
			\lim_{k \to \infty}
			\frac{1}{k \cdot N_k} \sum_{i = 1}^{N_k}  \mu_i^k,
		\end{equation}
		where $N_k$ and $\mu_i^k$ were defined before (\ref{eq_eta_defn}).
		However, by the additivity of the degree, we obtain 
		\begin{equation}\label{eq_x_HN2}
			\sum_{i = 1}^{N_k}  \mu_i^k
			=
			\frac{\deg(E_k)}{\int_C [\omega_C]}.
		\end{equation}
		By Riemann-Roch-Grothendieck theorem, similarly to (\ref{eq_rrh0}) and (\ref{eq_rrh1}), it is easy to see that
		\begin{equation}\label{eq_x_HN3}
			\lim_{k \to \infty} \frac{\deg(E_k)}{k \cdot N_k}
			=
			\frac{\int_Y c_1(F)^{n + 1}}{(n + 1) \cdot p_* c_1(F)^n}
			.
		\end{equation}
		A combination of (\ref{eq_x_HN1}), (\ref{eq_x_HN2}) and (\ref{eq_x_HN3}) implies (\ref{eq_x_HN}).
		\par 
		Now, from (\ref{eq_x_HN}), we deduce
		\begin{multline}\label{eq_HN_asymp_coh2}
			\int_{x \in \real} |x| d \eta^{HN}(x)
			\cdot
			\int_Y c_1(F)^n \pi^* [\omega_C]
			\cdot
			(n + 1)
			\\
			=
			2
			\int_{x \geq 0} x d \eta^{HN}(x)
			\cdot
			\int_Y c_1(F)^n \pi^* [\omega_C]
			\cdot
			(n + 1)
			-
			\int_Y c_1(F)^{n + 1}.
		\end{multline}
		Recall, however, that Chen in \cite[Theorem 1.1]{ChenVolume} established that there is the following relation
		\begin{equation}\label{eq_chen_vol}
			\hat{h}^0(Y, F)
			=
			\int_{x \geq 0} x d \eta^{HN}(x)
			\cdot
			\int_Y c_1(F)^n \pi^* [\omega_C]
			\cdot
			(n + 1).
		\end{equation}
		This is the second crucial place in the proof where we use the fact that $\dim B = 1$.
		We argue that
		\begin{equation}\label{eq_rrh}
			\hat{h}^0(Y, F) - \hat{h}^1(Y, F) = \int_Y c_1(F)^{n + 1}.
		\end{equation}
		Once we get (\ref{eq_rrh}), a combination of it, (\ref{eq_HN_asymp_coh2}) and (\ref{eq_chen_vol}) would imply (\ref{eq_HN_asymp_coh}), which finishes the proof. 
		To establish (\ref{eq_rrh}).
		Recall that by (\ref{eq_rrh0}), it suffices to establish that $\limsup$ in (\ref{defn_as_coh}) is actually a limit.
		For $q = 0$, a famous Fujita's theorem, \cite{Fujita}, cf. \cite{DemEinLaz}, states that $\limsup$ in (\ref{defn_as_coh}) is actually a limit.
		For $q > 1$, there is nothing to prove due to (\ref{eq_hhat_van}).
		The only case which is left, $q = 1$, then follows from (\ref{eq_rrh1}) and the validity of the statement for $q \neq 1$.
	\end{proof}
	\par 
	In conclusion to this section, let us point out a relation -- originally described in \cite{DemAndrGrau} -- between Corollary \ref{thm_andgr} and the Andreotti-Grauert theorem. 
	Recall that the Andreotti-Grauert theorem asserts the vanishing of some cohomology groups, associated with high tensor powers of a line bundle, carrying a metric with curvature having enough of positive eigenvalues at every point.
	The converse of this statement asks for the existence of metrics on the line bundle with certain positivity constraints on the curvature, provided that some cohomology groups vanish.
	It is known to hold in some special cases, cf. \cite{YangAndrGra}, but it fails in general, see \cite{Ottem}.
	What Conjecture 1 asks is a version of this converse statement, saying that \textit{asymptotic} vanishing of the cohomology implies the existence of a \textit{sequence of metrics} with suitable curvature.
	
	\section{Kobayashi-Hitchin correspondence and Hessian quotient equations}\label{sect_j_eq}
	The goal of this section is to describe a connection between the Wess-Zumino-Witten equation, the Hermite-Einstein and the Hessian quotient equations.
	\par 
	Let us first explain the relation with Hermite-Einstein equation and show that Corollary \ref{cor_wzw} can be seen as a generalization of the Kobayashi-Hitchin correspondence to fibrations which are not necessarily associated with vector bundles.
	\par 
	More precisely, let $F$ be a holomorphic vector bundle of rank $r$ over $B$. 
	Let $L := \mathcal{O}(1)$ be the hyperplane bundle over $X := \mathbb{P}(F^*)$, and let $\pi : \mathbb{P}(F^*) \to B$ be the natural projection.
	\par 
	Let us recall how to calculate the limiting Harder-Narasimhan measure $\eta^{HN}$ for this family.
	First, it is classical that the vector bundles $E_k = R^0 \pi_* L^{\otimes k}$ are isomorphic with ${\rm{Sym}}^k F$.
	It is also easy to see that for any $k \in \nat^*$, the Harder-Narasimhan slopes of ${\rm{Sym}}^k F$ can be easily expressed in terms of the Harder-Narasimhan slopes, $\mu_1, \ldots, \mu_k$ of $F$.
	To describe the limit of this relation, we denote by $\Delta = \{(x_1, \ldots, x_r) : x_1 + \cdots + x_r = 1, x_i \geq 0 \}$ the $r - 1$-simplex, and by $d \lambda$ the Lebesgue measure on $\Delta$, normalized so that $\int_{\Delta} d \lambda = 1$.
	We denote by $\phi : \Delta \to \real$ the map $(x_1, \ldots, x_r) \mapsto x_1 \mu_1 + \cdots + x_r \mu_r$.
	Then according to \cite[Proposition 3.5]{ChenVolume}, we have 
	\begin{equation}\label{eq_chen_volume_calccc}
		\eta^{HN} = \phi_* d \lambda.
	\end{equation}
	In particular, the measure $\eta^{HN}$ is a singleton if and only if $\mu_1 = \cdots = \mu_k$, i.e. $F$ is semistable.
	\par 
	Let us now describe the geometric side of Theorem \ref{thm_main} in this specific setting.
	We fix a Hermitian metric $h^F$ on $F$.
	We endow $L$ with the metric $h^L$ induced by $h^F$.
	It is then a classical calculation, cf. \cite{KobFinEins}, \cite[Lemma 3.3 and Remark 3.4]{FinHNII}, that the metric $h^F$ solves the Hermite-Einstein equation $\frac{\imun}{2 \pi} R^{h^F} \wedge \omega_B^{m-1} = \lambda \omega_B^m$  if and only if $\omega := c_1(L, h^L)$ solves
	\begin{equation}\label{eq_cor13_eq}
		\omega^{n + 1} \wedge \pi^* \omega_B^{m - 1}
		=
		\lambda (n + 1) \cdot
		\omega^n \wedge \pi^* \omega_B^m.
	\end{equation}
	Moreover, following a question raised by Kobayashi \cite{KobFinEins}, it was established by Feng-Liu-Wan in \cite[Proposition 3.5]{FengLiuWang} that if there is a metric $h^L$ on $L$ such that $\omega := c_1(L, h^L)$ solves (\ref{eq_cor13_eq}), then one can cook up from it a Hermitian metric on $F$, solving the Hermite-Einstein equation (in particular, $F$ is then polystable).
	\par 
	What Corollary \ref{cor_wzw} tells us in this specific setting is that if the vector bundle $F$ is semistable, then (\ref{eq_cor13_eq}) can be solved approximately. Due to the explanations above, this gives us a weak version of Kobayashi-Hitchin correspondence, and Corollary \ref{cor_wzw} in its full generality is a generalization of this correspondence to more general fibrations.
	We say “weak version" for two reasons: first, Kobayashi-Hitchin correspondence asserts existence of a Hermitian metric on $F$.
	As it is not clear what is the fibered analogue of a Hermitian metric on $F$, Corollary \ref{cor_wzw} can only provide the existence of a Finsler metric (a metric on $\mathscr{O}(1)$).
	Second, Corollary \ref{cor_wzw} proves that $L^1$-approximate solutions exists for semistable vector bundles, but Kobayashi-Hitchin correspondence provides $L^{\infty}$-approximate solutions, cf. \cite[Theorem 6.10.13]{KobaVB}.
	\par 
	Let us make an additional comment on the auxiliary equation (\ref{eq_aux_ma}).
	Explicit calculations along the lines of (\ref{eq_chen_volume_calccc}) would yield that for a fixed metric $h^F$ on $F$, in the orthonormal basis of $F_b$, $b \in B$, adapted to the Harder-Narasimhan filtration, in homogeneous coordinates $[x_1 : \cdots : x_r]$ on $\mathbb{P}(F^*_b)$, for the Fubini-Study metric $\omega$ induced by $h^F$, we have
	\begin{equation}
		HN(\omega)[x_1 : \cdots : x_r]
		=
		\frac{\sum_{i = 1}^{r} \mu_i \cdot |x_i|^2}{\sum_{i = 1}^{r} |x_i|^2}.
	\end{equation}
	By repeating the calculation leading to (\ref{eq_cor13_eq}), one sees that approximate solutions to (\ref{eq_aux_ma}) in this particular case can be constructed from $\delta$-approximate critical Hermitian structures on $F$.
	\begin{proof}[Proof of Corollary \ref{cor_wzw}]
		We argue that Theorem \ref{thm_main} establishes the equivalence between statements a) and c). 
		Specifically, condition a) corresponds to the vanishing of the right-hand side of (\ref{eq_main}), while condition c) corresponds to the vanishing of the left-hand side of (\ref{eq_main}).
		But as Theorem \ref{thm_main} asserts the identity between the two sides, a) and c) are equivalent.
		\par 
		Let us establish the equivalence between a) and b).
		Clearly, in the notations of (\ref{eq_eta_defn}), we have
		\begin{equation}\label{eq_max_slope_1}
			\mu_{N_k}^k
			=
			\sup
			\Big\{
				\mu(\mathcal{F}_k) : \mathcal{F}_k \text{ is a coherent subsheaf of } E_k, \, \rk{\mathcal{F}_k} > 0
			\Big\}.
		\end{equation}
		Moreover, by (\ref{eq_x_HN1}), we have
		\begin{equation}\label{eq_max_slope_2}
			\lim_{k \to \infty} \frac{\mu(E_k)}{k}
			=
			\int_{x \in \real} x d \eta^{HN}(x)
		\end{equation}
		However, by \cite[Theorem 1.1]{FinHNI}, we have 
		\begin{equation}\label{eq_max_slope_3}
			\lim_{k \to \infty} \frac{\mu_{N_k}^k}{k} =  \esssup \eta^{HN}.
		\end{equation}
		A combination of (\ref{eq_max_slope_1}), (\ref{eq_max_slope_2}) and (\ref{eq_max_slope_3}), yields that the inequality from b) is equivalent to the fact that $\eta^{HN}$ is the Dirac mass.
		By calculations as in (\ref{eq_x_HN3}), we see that
		\begin{equation}\label{eq_max_slope_4}
			\lim_{k \to \infty} \frac{\mu(E_k)}{k}
			=
			\frac{\int_X c_1(L)^{n + 1} \cdot \pi^* [\omega_B]^{m - 1} }{(n + 1) \cdot \int_X c_1(L)^n \cdot \pi^* [\omega_B]^m}.
		\end{equation}
		A combination of (\ref{eq_max_slope_2}) and (\ref{eq_max_slope_4}) then determines the support of $\eta^{HN}$.
		Since it is a probability measure, which is a Dirac mass, the support fully determines it. This finishes the proof of the equivalence between a) and b).
		\par 
		Lastly, let us establish the equivalence between a) and d) under an additional assumption $\dim B = 1$.
		Recall that in \cite{FinHNII}, we defined asymptotic semistability of $E_k := R^0 \pi_* L^k$ as follows: for any $\epsilon > 0$, there is $k_0 \in \nat$, such that for any $k \geq k_0$, for any quotient sheave $\mathcal{Q}_k$ of $E_k$, $\rk{\mathcal{Q}_k} > 0$, we have $\mu(\mathcal{Q}_k) \geq \mu(E_k) - \epsilon k$.
		In \cite[Theorem 1.3]{FinHNII}, by relying on the assumption $\dim B = 1$, we established that asymptotic semistability is equivalent to the validity of the inequality from condition d) for all irreducible complex analytic subspaces $Y \subset X$ projecting surjectively over $B$.
		Remark, however, that if $Y \subset X$ doesn't project surjectively over $B$, by the Grauert's theorem, the image is a proper analytic subset of $B$ (a point). 
		The term on the right-hand side then vanishes, but the term on the left-hand side is positive by the relative ampleness of $L$.
		Hence the inequality part from condition d) is satisfied in full.
		\par   
		Moreover, in the proof of this result, see \cite[after (4.18)]{FinHNII}, we established that asymptotic semistability is equivalent in the notations of (\ref{eq_eta_defn}) to
		\begin{equation}\label{eq_cond_pt_mass111}
			\lim_{k \to \infty} \frac{\mu_1^k}{k} 
			=
			\lim_{k \to \infty} \frac{\mu_{N_k}^k}{k}.
		\end{equation}
		However, by \cite[Proposition 6.1]{FinHNI}, if $\dim B = 1$, then $\lim_{k \to \infty} \mu_1^k / k = \essinf \eta^{HN}$.
		From this and (\ref{eq_max_slope_3}), the condition (\ref{eq_cond_pt_mass111}) is equivalent to $\essinf \eta^{HN} = \esssup \eta^{HN}$, which means precisely that $\eta^{HN}$ is a Dirac mass.
		As described after (\ref{eq_max_slope_4}), when $\eta^{HN}$ is a Dirac mass, its support is given by the $\lambda$ giving the identity from d), which finishes the proof.
	\end{proof}
	\par 
	Let us now describe a connection with the so-called \textit{Hessian quotient equation}.
	We fix a Kähler form $\chi$ and a Kähler class $[\omega]$ on $X$. 
	The Hessian quotient equation, introduced by Sz{\'e}kelyhidi in \cite[(185)]{SzekQuotHess}, is then given by 
	\begin{equation}\label{eq_quot_hess}
		\omega^{n + 1} \wedge \chi^{m - 1} = \lambda (n + 1) \cdot \omega^n \wedge \chi^m,
	\end{equation}
	where $\lambda$ is a certain (topological) constant, and $\omega$ is the unknown Kähler form from $[\omega]$.
	\par 
	Remark that (\ref{eq_cor13_eq}) has the same form as (\ref{eq_quot_hess}) with only one difference that instead of the positive $(1, 1)$-form $\chi$, we have a semi-positive form $\pi^* \omega_B$. 
	This -- seemingly minor -- change breaks down many of the known techniques for the study of (\ref{eq_quot_hess}), as the linearization of (\ref{eq_cor13_eq}), unlike (\ref{eq_quot_hess}), in not elliptic.
	Also, the notion of subsolution from \cite[(12)]{SzekQuotHess} collapses in this degenerate setting, as it is easy to see there is no relatively positive $(1, 1)$-form $\omega$, which verifies the inequality $\omega^n \wedge \pi^* \omega_B^{m - 1} - \lambda n \omega^{n - 1} \wedge \pi^* \omega_B^m > 0$ (and if we replace the sign $>$ with $\geq$, then any form verifying this bound automatically solves (\ref{eq_cor13_eq})).
	\par 
	Another difference between (\ref{eq_cor13_eq}) and (\ref{eq_quot_hess}) is that we only assume that $[\omega]$ is relatively Kähler in Corollary \ref{cor_wzw} instead of the Kähler assumption from (\ref{eq_quot_hess}). 
	However, this issue is a minor one, as (\ref{eq_cor13_eq}) doesn't change much if one changes $\omega$ to $\omega + T \pi^* \omega_B$ for some $T > 0$ big enough, and $[\omega] + T \pi^* [\omega_B]$ becomes Kähler for $T$ big enough.
	\par 
	We will now show that, a least when $\dim B = 1$, when $L$ is ample, the forms $\omega_{\epsilon}$ appearing in Corollary \ref{cor_wzw} can be obtained as solutions to the Hessian quotient equation, (\ref{eq_quot_hess}).
	\par 
	More precisely, for $\epsilon > 0$, we consider a $(1, 1)$-form $\chi_{\epsilon} := \pi^* \omega_B + \epsilon \omega_0$, where $\omega_0$ is an arbitrary Kähler form from $[\omega]$.
	Clearly, the form $\chi_{\epsilon}$ is Kähler for $\epsilon > 0$.
	We denote $\lambda_{\epsilon} := \int_X [\omega]^{n + 1} / ( \int_X [\omega]^n  [\chi_{\epsilon}] (n + 1) )$.
	\begin{prop}\label{prop_dp_cor}
		Assume that $\dim B = 1$, and $\eta^{HN}$ is the Dirac mass. Then for any $\epsilon > 0$, the $J$-equation, given by
		\begin{equation}\label{eq_dp_cor}
			\omega^{n + 1} = \lambda_{\epsilon} (n + 1) \cdot \omega^n \wedge \chi_{\epsilon},
		\end{equation}
		admits a solution, $\omega := \omega_{\epsilon, 0}$, from the class $[c_1(L)]$.
		Moreover, the $(1, 1)$-form $\omega_{\epsilon}$ from Corollary \ref{cor_wzw} can be taken as $\omega_{\epsilon', 0}$ for some $\epsilon' > 0$ small enough.
	\end{prop}
	This result will be established as a consequence of the numerical criteria for the existence of the solutions to (\ref{eq_dp_cor}).
	Recall the following statement proved by Datar-Pingali \cite[Theorem 1.1]{DatarPingali}, see also Székelyhidi \cite{SzekQuotHess} and G. Chen \cite{ChenGao} for related results.
	\begin{thm}\label{thm_datar_ping}
		For any Kähler form $\chi$ on $X$ from the class $[\chi]$ and a Kähler class $[\omega]$, the following conditions are equivalent.
		\par 
		a) For $\lambda = \int_X [\omega]^{n + 1} / (\int_X [\omega]^n [\chi] (n + 1))$, there is a unique Kähler form $\omega$ from $[\omega]$, verifying $\omega^{n + 1} = \lambda (n + 1) \cdot \omega^n \wedge \chi$, and so that the $(n, n)$-form $\omega^n - \lambda n \omega^{n - 1} \wedge \chi$ is positive.
		\par 
		b) 	For any irreducible complex analytic subspace $Y \subset X$ of dimension $k + 1$, $k \in \nat$, $k < n$, we have 
	$\int_Y [\omega]^{k + 1} > \lambda (k + 1) \cdot \int_Y [\omega]^k [\chi]$.
	\end{thm}
	\begin{proof}[Proof of Proposition \ref{prop_dp_cor}]
		In order to show that the equation (\ref{eq_dp_cor}) has solutions, we will verify that the second condition from Theorem \ref{thm_datar_ping} is satisfied. 
		\par
		Let us show that if the inequality $\int_Y c_1(L)^{k + 1} \geq \lambda_{\epsilon} (k + 1) \cdot \int_Y c_1(L)^k \cdot [\chi_{\epsilon}]$ holds for $\epsilon = 0$, then it holds (with a sign $>$ instead of $\geq$ if $k < n$) for any $\epsilon > 0$.
		For this, we introduce
		\begin{equation}\label{eq_a_b_calc_num_cr}
			a := \frac{\int_Y [\omega]^{k + 1}}{(k + 1) \cdot \int_Y [\omega]^k [\chi]},
			\qquad
			b := \frac{\int_X [\omega]^{n + 1}}{(n + 1) \cdot \int_X [\omega]^n [\chi]}.
		\end{equation}
		Then an easy manipulation shows that the inequality $\int_Y c_1(L)^{k + 1} > \lambda_{\epsilon} (k + 1) \cdot \int_Y c_1(L)^k \cdot [\chi_{\epsilon}]$ is equivalent to $a (1 + \epsilon (n + 1) b) > b (1 + \epsilon (k + 1) a)$, which clearly holds for any $\epsilon > 0$ as long as $a \geq b > 0$, $k \neq n$, which is true by our assumption.
		\par 
		Hence, the second condition in Theorem \ref{thm_datar_ping} holds for any $\epsilon > 0$.
		By Theorem \ref{thm_datar_ping}, for any $\epsilon > 0$, there is a Kähler form $\omega_{\epsilon, 0}$, verifying $\omega_{\epsilon, 0}^{n + 1} = \lambda_{\epsilon} (n + 1) \omega_{\epsilon, 0}^n \wedge \chi_{\epsilon}$.
		\par 
		Now, by the triangle inequality, we have
		\begin{equation}
			\big| \omega_{\epsilon, 0}^{n + 1} - \lambda (n + 1) \omega_{\epsilon, 0}^n \wedge \pi^* \omega_B \big|
			\leq
			 (\lambda - \lambda_{\epsilon}) (n + 1) \omega_{\epsilon, 0}^n \wedge \chi_{\epsilon}
			+
			\epsilon \lambda (n + 1) \omega_{\epsilon, 0}^n \wedge \omega_0,
		\end{equation}
		which immediately implies the result, as by Chern-Weil theory, we have $\int_X \omega_{\epsilon, 0}^n \wedge \chi_{\epsilon} = \int_X [\omega]^n [\chi_{\epsilon}]$, $\int_X \omega_{\epsilon, 0}^n \wedge \omega_0 = \int_X [\omega]^n [\omega_0]$, and these quantities remain bounded, as $\epsilon \to 0$.
	\end{proof}
	\par
	It will be interesting to know if for $\dim B \geq 2$ the forms from Corollary \ref{cor_wzw} can also be taken as solutions of some auxiliary differential equations.
	\par
	Also, taken into account the fact that the continuity method and the method of geometric flows plays a crucial role in Kobayashi-Hitchin correspondence and the study of Hessian quotient equations, it is interesting to know if these methods can be used to give alternative proofs for the results from this paper, obtained through quantization. 
	This is particularly relevant as in the situations analogous to the one of Theorem \ref{thm_main}, cf. \cite{DonLowCalabi}, \cite{XiaCalabi} \cite{HisamCalabi}, \cite{DerSzek}, the minimizing sequences for other functionals were obtained through solutions of some geometric flows.
	\par 
	As a concluding remark, we would like to mention some related works concerning the equation (\ref{eq_wzw}) in the setting when $B$ has a boundary.
	First, when $B$ is an annuli in $\comp$, weak solutions to the Dirichlet problem associated with (\ref{eq_wzw}) correspond to the Mabuchi geodesics in the space of all Kähler potentials, \cite{Semmes} \cite{DonaldSymSp}, and they always exist, see \cite{ChenGeodMab} (cf. also \cite{WuKRWZW} for a related result on pseudoconvex domains in $\comp^m$).
	Moreover, the results \cite{PhongSturm}, \cite{RubinZeld}, \cite{SongZeldToric}, \cite{WuKRWZW} show that the solution to this equation can be obtained as a dequantization of solutions to a Dirichlet problem associated with the Hermite-Einstein equations on $E_k$, as $k \to \infty$.
	\par 
	Of course in the setting of a manifold with boundary, the solution to (\ref{eq_wzw}) minimizes the Wess-Zumino-Witten functional, and the solutions to the Hermite-Einstein equations minimize the respective Hermitian Yang-Mills functionals (both miniminal values are zero).
	The major difference between these results and the ones from this article are due to the fact in the boundaryless setting considered here, neither (\ref{eq_wzw}), nor Hermite-Einstein equation have solutions.

\bibliography{bibliography}

		\bibliographystyle{abbrv}

\Addresses

\end{document}